\documentclass{amsart}

\usepackage{geometry}
\usepackage{algpseudocode}
\usepackage{algorithm}

\usepackage{amsfonts}
\usepackage{mathtools}
\usepackage{amssymb}
\usepackage{amsthm}

\usepackage{mathptmx}

\usepackage{graphicx}
\graphicspath{{figures/}} % Location of the graphics files
\usepackage{subcaption}
\usepackage{tikz}
\usepackage[all]{xy}

\usepackage{amsmath}

\usepackage[square,numbers]{natbib}

\usepackage[pdftex, breaklinks]{hyperref}

\hypersetup{
     colorlinks = true,
     linkcolor = blue,
     anchorcolor = blue,
     citecolor = blue,
     filecolor = blue,
     urlcolor = blue
     }

\newtheorem{theorem}{Theorem}
\theoremstyle{definition}
\newtheorem{example}[theorem]{Example}

\theoremstyle{plain}

\newtheorem{proposition}[theorem]{Proposition}

\theoremstyle{definition}
\newtheorem{definition}{Definition}[section]

\newtheorem*{remark}{Remark}

\theoremstyle{remark}

\newcommand{\rI}{{\rm I}}

\newcommand{\F}{\mathbb{F}}

\newcommand{\R}{\mathbb{R}}
\newcommand{\bT}{\mathbb{T}}
\newcommand{\V}{\mathbb{V}}
\newcommand{\W}{\mathbb{W}}
\newcommand{\Z}{\mathbb{Z}}
\newcommand{\bG}{\mathbb{G}}
\newcommand{\bH}{\mathbb{H}}

\newcommand{\bV}{\mathbb{V}}

\newcommand{\A}{\mathcal{A}}
\newcommand{\B}{\mathcal{B}}
\newcommand{\I}{\mathcal{I}}
\newcommand{\K}{\mathcal{K}}
\newcommand{\U}{\mathcal{U}}
\newcommand{\cA}{\mathcal{A}}
\newcommand{\cC}{\mathcal{C}}
\newcommand{\cE}{\mathcal{E}}
\newcommand{\cF}{\mathcal{F}}
\newcommand{\cG}{\mathcal{G}}
\newcommand{\cH}{\mathcal{H}}
\newcommand{\cI}{\mathcal{I}}

\newcommand{\cM}{\mathcal{M}}
\newcommand{\cO}{\mathcal{O}}

\newcommand{\cQ}{\mathcal{Q}}
\newcommand{\cR}{\mathcal{R}}
\newcommand{\cS}{\mathcal{S}}
\newcommand{\cU}{\mathcal{U}}
\newcommand{\cV}{\mathcal{V}}

\newcommand{\permaviss}{\textsc{PerMaViss}}
\newcommand{\imagekernel}{\texttt{image\_kernel}}
\newcommand{\cech}{\v{C}ech }

\newcommand{\cmax}{{\rm max}}

\newcommand{\clog}{{\rm log}}

\newcommand{\Ho}{{\rm H}}
\newcommand{\PH}{{\rm PH}}

\newcommand{\Id}{{\rm Id}}
\newcommand{\Ker}{{\rm Ker}}
\newcommand{\Img}{{\rm Im}}
\newcommand{\dimn}{{\rm dim}}

\newcommand{\Tot}{{\rm Tot}}
\newcommand{\GK}{{\rm GK}}
\newcommand{\GZ}{{\rm GZ}}
\newcommand{\IB}{{\rm IB}}

\newcommand{\bR}{{\bf R}}
\newcommand{\SpCpx}{{\rm \bf SpCpx}}
\newcommand{\Vect}{{\bf Vect}}
\newcommand{\vect}{{\bf vect}}
\newcommand{\PMod}{{\bf PMod}}

\newcommand{\morph}[3]{#1 : #2 \rightarrow #3} 
\newcommand{\surjmorph}[3]{#1 : #2 \twoheadrightarrow #3}
\newcommand{\injmorph}[3]{#1 : #2 \hookrightarrow #3}
\newcommand{\finset}[4]{\{{#1}_{#2} \}_{#3 \leq #2 \leq #4}}
\newcommand{\gset}[3]{\{{#1}_{#2} \}_{#2 \in #3}}

\newcommand{\al}[1]{\alpha_{#1}}
\newcommand{\be}[1]{\beta_{#1}}

\newcommand{\bOnePar}[2]{{\bf 1}_{#1}\left(\, #2 \,\right)}
\newcommand{\bOne}[1]{{\bf 1}_{#1}}
\newcommand{\bigBarSum}{
  \mathop{
    \vphantom{\bigoplus} 
    \mathchoice
      {\vcenter{\hbox{\resizebox{\widthof{$\displaystyle\bigoplus$}}{!}{$\boxplus$}}}}
      {\vcenter{\hbox{\resizebox{\widthof{$\bigoplus$}}{!}{$\boxplus$}}}}
      {\vcenter{\hbox{\resizebox{\widthof{$\scriptstyle\oplus$}}{!}{$\boxplus$}}}}
      {\vcenter{\hbox{\resizebox{\widthof{$\scriptscriptstyle\oplus$}}{!}{$\boxplus$}}}}
  }\displaylimits 
}

\newcommand{\barSum}{\boxplus}
\newcommand{\barRel}[3]{#1 \sim {\rm I}(#2,#3)}
\newcommand{\RORel}[3]{#1 \sim [#2,#3)}	
\newcommand{\GInt}[2]{{\rm I}(#1, #2)}

\title{Distributing Persistent Homology Via Spectral Sequences}

\author{\' Alvaro Torras Casas	}

\address{School of Mathematics, Cardiff University, Senghennydd Road, Cardiff, CF24 4AG}
\email{TorrasCasasA@cardiff.ac.uk}
\thanks{ \' Alvaro Torras Casas is supported by an EPSRC grant with reference EP/N509449/1}

\begin{document}

\maketitle 

\begin{abstract}
We set up the theory for a distributed algorithm for computing persistent
homology.  For this purpose we develop linear algebra of persistence modules.
We present bases of persistence modules, and give motivation as for the
advantages of using them.  Our focus is on developing efficient methods for the
computation of homology of chains of persistence modules.  Later we give a
brief, self contained presentation of the Mayer-Vietoris spectral sequence.
Then we study the Persistent Mayer-Vietoris spectral sequence and  present a
solution to the extension problem.  Finally, we review \permaviss, a method
implementing these ideas. This procedure distributes simplicial data, 
while focusing on merging homological information.   

\end{abstract}



\section{Introduction}

\subsection{Motivation}

Persistent homology has existed for about two decades \citep{Edelsbrunner2002}.
This tool of applied topology has played a central role in applications, such as
the study of geometric structure of sets of points lying in $\R^n$,
see~\citep{Delfinado1995, Edelsbrunner2002}.  This introduced the field of
Topological Data Analysis which, very soon, was applied to a multitude of
problems, see \citep{Carlsson2009, Ghrist} for a survey article and an
introduction. Among others, persistent homology has been applied to study
coverage in sensor networks \citep{deSilvaGhrist2007},  pattern detection
\citep{Robins2016}, classification and recovery of signals \citep{Robinson2014}
and it has also had an impact on shape recognition using machine learning
techniques, see \citep{AdamsEmerson2017, DiFabio2012}.  All these applications
motivate the need for fast algorithms for computing persistent homology. The
usual algorithm used for these computations was introduced in
\citep{Edelsbrunner2002}, with some later additions to speed up such as those of
\citep{Chen2011, Chen2013, DeSilva2011}. In~\citep{Milosavljevic2011} persistent homology is proven to be computable in matrix multiplication time. However, since these matrices become large very quickly, the computations are generally very expensive, both in
terms of computational time and in memory required. 

In practice computing the persistent homology of a given
filtered complex is equivalent to computing its matrices of differentials and
perform successive  Gaussian eliminations; see \citep{EdelsbrunnerHarer2010,
Edelsbrunner2002}.
In recent years, some methods have been developed for the parallelization of
persistent homology. The first approach was introduced in~\citep{EdelsbrunnerHarer2010} 
as the \emph{spectral sequence algorithm}, and was successfully implemented in 
\citep{BauerKerberReininghaus2014b}. This consists in dividing the original
matrix $M$ into groups of rows, and sending these to different processors. These
processors will, in turn, perform a local Gaussian Elimination and share the
necessary information between them, see \citep{BauerKerberReininghaus2014b}.  
On the other hand, a more topological approach is presented in~\citep{lewismorozov2015}. It
uses the~\emph{blow-up complex} introduced in~\citep{ZomorodianCarlsson2008}. 
This approach first takes a cover $\cC$ of a filtered simplicial complex $K$, 
and uses the result that the persistent homology of $K$ is isomorphic to that of 
the blow-up complex $K^\cC$. This proceeds by 
computing the sparsified persistent homology for each cover, and then use this
information to reduce the differential of $K^\cC$ efficiently. 
Both of these parallelization methods have provided substantial speedups
compared to the standard method presented in~\citep{Edelsbrunner2002}.

Following the 
ideas on~\citep{ZomorodianCarlsson2008}, having an understanding of how persistence
barcodes relate to a cover can help us obtain better representatives. On this basis, it would be 
desirable to have a method that leads to the speedups from 
\citep{BauerKerberReininghaus2014b, lewismorozov2015}, 
while still keeping cover information from~\citep{ZomorodianCarlsson2008}. 
Further, it would also be desirable to 
drop all restrictions in covers, and consider 
functional covers such as those used in the \emph{mapper algorithm}, see~\citep{Singh2007}. 
This last point limits substantially the use of the blowup-complex, since the number
of simplices grows very quickly when we allow the intersections to grow. In fact, in
the extreme case where a complex $K$ is covered by $n$ copies of $K$, the 
blowup complex $K^\cC$ has size $2^n |K|$. 

\subsection{The Persistence Mayer Vietoris spectral sequence and related literature}

Since distribution is an important issue in persistent homology, it is worth
exploring which classical tools of algebraic topology could be used in this
context.  A very well-known tool for distributing homology computations is
the Mayer-Vietoris spectral sequence, see~\citep{Chow2006} for a quick introduction to 
spectral sequences. It is no surprise that these objects work in in this context, since 
they have been
employed for similar problems for a long time, see~\citep{BoTu1982} or~\citep{McCleary2001}.
Since the category of persistence modules and persistence morphisms is an
abelian category, the process of computing a spectral sequence should be more or
less straightforward. 
However, there is always the question of how we implement this in practice. 
Furthermore, this approach has been already proposed in \citep{Lipsky2011},
although without a solution to the extension problem. Later, spectral sequences
were used for distributing computations of
cohomology groups in a field in~\citep{CGN2016}, and recently 
in~\citep{Yoon2018} and~\citep{yoon2020persistence} spectral 
sequences are used for distributing persistent homology computations.  However, all
of \citep{CGN2016,Yoon2018,yoon2020persistence} assume that the nerve of the 
cover is one dimensional. 

The first problem when dealing with spectral sequences is that we need to be
able to compute images, kernels and quotients. Needless to say, these should be
computed in an optimal way.  
This question has already been studied in~\citep{CohenSteiner2009}, where the
authors give a very efficient algorithm. However, there are couple of problems
that come up when using~\citep{CohenSteiner2009} in spectral sequences:

\begin{enumerate}

\item In~\citep{CohenSteiner2009} the authors assume that a
given morphism is induced by the inclusion $X \subseteq Y$ of two given filtered
simplicial complexes. This is not the case in spectral sequences, where the maps
in the second, third and higher pages are not induced by a simplicial
morphism. Furthermore, even when computing the first page this is not the case.
Indeed, the \cech differentials are not inclusions at all, where each
simplex is mapped to its copy on different covers. 
This means that the algorithm in~\citep{CohenSteiner2009}
needs to be adapted to our case. 

\item A key assumption in~\citep{CohenSteiner2009} is that the
filtrations in $X$ and $Y$ are both general. This is a fairly broad premise
in cases such as when both $X$ and $Y$ are Vietoris Rips complexes on two point
clouds. However, in spectral sequences this hypothesis \emph{hardly
ever} holds. Indeed, this follows from the fact that 
a simplex might be contained in various overlapping covers.  As
one can see in table 2 from~\citep{CohenSteiner2009}, 
the authors assume that there are only 6 possible
combinations of births and deaths in images, kernels and cokernels. When
generality does not hold, the number of cases is arbitrary.

\end{enumerate}

Thus, if we want to compute images, kernels and cokernels, we will need to be
able to overcome these two difficulties first. Also, notice that a good solution
should lead to the representatives, as these are needed for the spectral sequence. 

The other difficulty that one might encounter in spectral sequences comes with
the extension problem. That is, once we have computed the spectral sequence, we
still need to recompose \emph{broken barcodes} in order to recover the global
persistent homology. Within the context of persistent homology, the extension
problem first appeared in section~6 from~\citep{Govc_2017}. There the authors
give an approximate result that holds in the case of acyclic coverings. This
allows them to compare the persistent homology to the lower row of the infinity
page in the spectral sequence. This leads to an $\epsilon$-interleaving between
the global persistent homology and that of the filtered nerve. 
Later, the extension problem appeared in the PhD Thesis of Hee Rhang Yoon~\citep{Yoon2018}, 
and also in the recent joint work with Robert Ghrist~\citep{yoon2020persistence}. In section 4.2.3
from Yoon's Thesis, the author gives a detailed solution for the extension
problem in the case when the nerve of the cover is one dimensional.  
  
\subsection{Original Contribution}

In this paper, we set the theoretical foundations for a distributed method on the input
data.  In order to do this, we use the algebraic power of the Mayer-Vietoris
spectral sequence.  Since the aim is to build up an explicit algorithm,  we
need to develop linear algebra of persistence modules, as done through
Section~\ref{sec:persistence-modules}. There, we define \emph{barcode bases}
and also we develop an operation $\barSum$ that allows us to determine whether
a group of \emph{barcode vectors} are linearly independent or not. This
machinery, although it might seem artificial, is the key to understanding what it
really means to subtract columns from left to right in the Gaussian
elimination outlined in \imagekernel, see Algorithm~\ref{cde:img-kernel}. 
Also, it helps us to encapsulate all
the information related to a persistence morphism in a matrix that depends on the 
choice of two barcode bases. This is analogous to the case of linear algebra, where a linear morphisms
is given in terms of a matrix relative to a domain and codomain basis. This
approach has the advantage that \imagekernel~addresses the two issues
raised above with regards to~\citep{CohenSteiner2009}. In fact, \imagekernel~works 
for morphisms between any pair of \emph{tame} persistence modules. 

Next in section \ref{sec:MayerVietorisSS}, we give a detailed review of the Mayer-Vietoris
spectral sequence in the homology case. 
This is followed by section~\ref{sub:extension-problem}, where we give a solution to the extension
problem. The solution is given by a careful consideration of the total
complex homology, together with the use of barcode basis machinery developed 
in section~\ref{sec:persistence-modules}.  In secton~\ref{sec:permaviss} we introduce \permaviss, 
an algorithm for 
computing the persistence Mayer-Vietoris spectral sequence and solving the extension problem. 
The advantage of this procedure is that all the simplicial information 
is enclosed within local matrices. 
This has one powerful consequence; this method consists in
computing local Gaussian eliminations plus computing \imagekernel~on matrices whose order is that
of homology classes.
In particular, given enough processors and a `good' cover
of our data, one has that the complexity is about
$$
\cO(X^3) + \cO(H^4),
$$
where $X$ is the order of the maximal local complex and $H$ is the overall number of
nontrivial persistence bars on the whole dataset. For more details on this, 
we refer the reader to section~\ref{sec:complexity_permaviss}.

By using the ideas in this text we developed
\permaviss, a Python3 library that computes the Persistence Mayer-Vietoris spectral sequence. 
In the results from~\citep{permaviss}, one can see that
nontrivial higher differentials come up and also the extension
problem is a fairly frequent phenomenon of nontrivial solution. 
This supports the idea that the spectral sequence adds more information on
top of persistent homology. 
Finally, we outline
future directions, both for the study of the Persistence Mayer Vietoris
spectral sequence and future versions of \permaviss. 

\section{Preliminaries}

\subsection{Simplicial Complexes}
\label{sub:simplicial_complexes}
\begin{definition}
Given a set $X$, a \emph{simplicial complex} $K$ is a subset of the power set
$K \subseteq P(X)$ such that if $\sigma \in K$, then for all subsets 
 $\tau \subseteq \sigma$ we have that $\tau \in K$.  An element $\sigma \in K$ will be called a
$n$-\emph{simplex} whenever $|\sigma| = n+1$, whereas a subset 
$\tau \subseteq \sigma$ will be called a \emph{face}.  Thus, if a simplex is contained
in $K$ all its faces must also be contained in $K$.  Given a simplicial complex
$K$, we denote by $K_n$ the set containing all the $n$-simplices from $K$.
Given a pair of simplicial complexes $K$ and $L$, if 
$L \subseteq K$, then we say that $L$ is a \emph{subcomplex} of $K$. 
Also, given a mapping $f:K \rightarrow L$ between two simplicial complexes $K$
and $L$, we call $f$ a \emph{simplicial} morphism whenever $f(K_n) \subseteq
\bigcup_{l=0}^n L_l$ for all $n \geq 0$.  The category composed of simplicial
complexes and simplicial morphisms will be denoted by $\SpCpx$. 
\end{definition}

Let $\F$ be a field. For each $n \geq 0$ we define the free vector space over
the $n$-simplices of $K$ as
$$
S_n(K) \coloneqq \F [K_n].
$$
We also consider linear maps $d_n : S_n(K) \rightarrow S_{n-1}(K)$ usually
called \emph{differentials}, defined by
\begin{equation}\label{eq:differential}
    d_n([v_0,\ldots,v_n]) = \sum 
    \limits_{i = 0}^n (-1)^i [v_0,\ldots, \hat{v_i}, \ldots, v_n];
\end{equation}
where the hat notation, $\hat{v_i}$, is used to indicate omission of a vertex. 
Setting $S_n(K) = 0$ for all $n < 0$ we put all of these in a sequence
\begin{equation}
\label{seq:simplicial-chain-cpx}
\xymatrix{
	0 \  &
	S_0(K) \ar[l]_{0} &
	S_1(K) \ar[l]_{d_1} &
	S_2(K) \ar[l]_{d_2} &
	\cdots \ar[l]_{d_3} 
}
\end{equation}
It follows from formula (\ref{eq:differential}) that the composition of two
consecutive differentials vanishes: $d_n \circ d_{n-1} = 0$ for all $n \geq 0$.  
In this case we say that~(\ref{seq:simplicial-chain-cpx}) is a \emph{chain complex}.
As a consequence, we have that $\Img(d_{n+1}) \subseteq \Ker(d_n)$, and we can define the \emph{homology} with coefficients in
$\F$ to be
$$
\Ho_n(K;\F) = \dfrac{\Ker(d_n)}{\Img(d_{n+1})},
$$
for all $n \geq 0$. In general, $\F$ will be understood by the context and the
notation $\Ho_n(K)$ might be used instead. On the other hand, we consider
the \emph{augmentation map}
$\varepsilon : S_0(K) \rightarrow \F$ defined by the assignement $s \mapsto 1_\F$, 
for any simplex $s \in S_0(K)$. Then, we define
the \emph{reduced homology} by 
$$
\widetilde{\Ho}_0(K;\F) = \dfrac{\Ker(\varepsilon)}{\Img(d_{1})},
$$
and $\widetilde{\Ho}_n(K;\F) = \Ho_n(K;\F)$ for all $n > 0$. Consider the  
chain complex $\widetilde{S}_*(K)$, obtained by augmenting~(\ref{seq:simplicial-chain-cpx}) 
by $\varepsilon$ and a copy of $\F$ in degree $-1$:
$$
\xymatrix{
    0 \                                 &
    \F          \ar[l]_0                &
    S_0(K)      \ar[l]_{\varepsilon}    &
    S_1(K)      \ar[l]_{d_1}            &
    S_2(K)      \ar[l]_{d_2}            &
    \cdots      \ar[l]_{d_3} 
}
$$
Then one can see that computing reduced homology is the same as computing 
homology on $\widetilde{S}_*(K)$. 

\begin{definition}[Standard $m$-simplex]
Given $m > 0$, we define $\Delta^m = P(\{0, 1, \ldots, m \})$, which will be called 
the \emph{standard $m$-simplex}.
This leads to a chain complex $\widetilde{S}_*(\Delta^m)$
$$
\xymatrix{
	0 \  &
	\F     \ar[l]_{0}  &
	S_0(\Delta^m) \ar[l]_{\epsilon} &
	S_1(\Delta^m) \ar[l]_{d_1} &
	S_2(\Delta^m) \ar[l]_{d_2} &
	\cdots \ar[l]_{d_3}       &
	S_n(\Delta^m) \ar[l]_{d_n} &
        0 \ar[l]
}
$$
\end{definition}
By a standard result $\widetilde{S}_*(\Delta^m)$ is \emph{exact}, that is,  
$\widetilde{\Ho}_n(\Delta^m) = 0$ for all $n \geq 0$. 
For a proof, see Theorem 8.3 in~\citep{Munk1984}.
\begin{definition}
Let $K$ be a simplicial complex. A finite set $\U = \{ U_i \}_{i = 1}^m$ of
subcomplexes from $K$, is said to be a \emph{cover} of $K$ whenever $K = \bigcup_{i = 1}^m U_i$.   
For each simplex $\sigma \in \Delta^m$, we will use the notation $U_\sigma = \bigcap_{i \in \sigma} U_i$.
Altogether, we define the \emph{nerve} of $\U$ as the simplicial complex
$$
N^\U = \Big\{\, \sigma : U_\sigma \neq \emptyset \,\Big\} \subseteq \Delta^m.
$$
This leads to an augmented chain complex $\widetilde{S}_*(N^\U)$ with differentials
denoted by $d^{N^\U}_*$. 
In particular, given a simplex $\sigma \in N^\U$, we have a simplicial injection
$f^\sigma :\Delta^{|\sigma|} \hookrightarrow N^\U$. This induces an injection of chain complexes
$f^\sigma_*: \widetilde{S}_*(\Delta^{|\sigma|}) \hookrightarrow \widetilde{S}_*(N^\U)$ whose 
image $f^\sigma_*\Big(\, \widetilde{S}_*(\Delta^{|\sigma|})\,\Big)$
is exact. 
\end{definition}
\begin{definition}[\cech chain complex]
Let $K$ be a simplicial complex and let $\cU = \{ U_i \}_{i=1}^{m}$ be a cover
of $K$ by $m$ subcomplexes. For each simplex $s \in K$, there exists
a simplex $\sigma(s) \in N^\U$ with maximal cardinality $|\sigma(s)|$, so that $s \in U_{\sigma(s)}$. 
Then, for a fixed degree $n \geq 0$, we define the $(n, \cU)$-\cech chain complex by 
$$
\check{C}_*(n, \cU;\F) = \bigoplus_{s \in K_n} f^{\sigma(s)}_*\Big(\,\widetilde{S}_*\big(\,\Delta^{|\sigma(s)|}\,\big)\,\Big).
$$
\end{definition}
For $k \geq -1$, we will use the notation $(\tau)_s$ with $s \in K_n$ and 
$\tau \in \widetilde{S}_k\big(\,\Delta^{|\sigma(s)|}\,\big)$, to denote an element in $\check{C}_k(n, \cU;\F)$
that is zero everywhere except for $\tau$ in the component indexed by $s$. Then
the image of the $k$-\cech differential is defined by the assignement
$\check{\delta}_k^\U((\tau)_s) = (d_k^{N^\U} \tau)_s$. 
Notice that by definition the \cech complex is a chain complex and is exact. 
Also, one can see that
$$
\check{C}_{-1}(n, \cU;\F) \simeq S_n(K)
$$
follows easily. On the other hand, for each $k \geq 0$ we define an isomorphism
$$
\psi_k : \check{C}_k(n, \cU;\F) \simeq \bigoplus \limits_{\sigma \in N^\U_k} S_n(U_\sigma)
$$
by sending
$(\tau )_s$ to $(s)_\tau$ for any pair of simplices $s \in K_n$ and 
$\tau \in f^{\sigma(s)}_k\Delta^{|\sigma(s)|}$. 
In particular, we can rewrite the $(n, \cU)$-\cech chain complex as a sequence
\begin{equation}
\label{seq:cech_sequence}
\xymatrix@C=0.6cm{
	0		&
	S_n(K)	\ar[l]	&
	\bigoplus \limits_{\sigma \in \substack \Delta^m_0} S_n(U_\sigma) \ar[l]_{\delta_0}	&
	\bigoplus \limits_{\sigma \in \substack \Delta^m_1} S_n(U_\sigma) \ar[l]_{\delta_1}	&
	\bigoplus \limits_{\sigma \in \substack \Delta^m_2} S_n(U_\sigma) \ar[l]_{\delta_2}	&
	\cdots  \ar[l]
}
\end{equation}
where the differentials $\delta_i$ are chosen in order to commute with the $\psi_i$'s. That
is, one has that, for any pair of simplices $\sigma \in N^\U_k$ and $s \in (U_\sigma)_n$, we have 
equalities 
\begin{eqnarray*}
\delta_k((s)_\sigma) &=&    \psi_k \circ \check{\delta}_k \circ \psi_k^{-1} \big(\,(s)_\sigma\,\big) \\ 
                     &=&    \psi_k \circ \check{\delta}_k \big(\,(\sigma)_s\,\big) \\ 
                     &=&    \psi_k \big(\,(d^{N^\U}_k\sigma)_s\,\big) \\ 
                     &=&    \Big(\,\big\{\,d^{N^\U}_k(\sigma)\,\big\}_\tau \cdot s\,\Big)_{\tau \in N^\U_{k-1}}, \\
\end{eqnarray*}
where $\big\{\,d^{N^\U}_k(\sigma)\,\big\}_\tau \in \F$ is the coefficient of $d^{N^\U}_k(\sigma)$ in 
the simplex $\tau \in N^\U_{k-1}$. 

\begin{remark}
Alternatively, the \cech chain complex can be defined straight away as 
the sequence~(\ref{seq:cech_sequence}). Then, one can see that this is an 
exact chain complex by using cosheaf theory. Namely, given a simplicial
complex $K$, we consider the topology where the open sets are given by 
subcomplexes. Then, for each integer $n \geq 0$, one has the simplicial
\emph{precosheaf} as an assignement
$$
\cS_n : V \mapsto S_n(V)
$$
for each subcomplex $V \subseteq K$. This precosheaf is in fact a
\emph{flabby cosheaf}. Then, using 2.5, 4.3, and 4.4 from section VI.
in~\citep{Bredon1997}, one has exactness of the \cech chain complex.  
\end{remark}

\subsection{Persistence Modules}
Let $\bR$ be the category of real numbers as a poset, where $\hom_\bR (s,t)$
contains a single morphism whenever $s \leq t$, and is empty otherwise.  
Let $\F$ be a field and let $\Vect$ denote the category of $\F$-vector spaces. Also
let $\vect \subset \Vect$ be the subcategory of finite dimensional
$\F$-vector spaces. 

\begin{definition}
A \emph{filtered simplicial complex} is a functor $K \colon \bR \rightarrow \SpCpx$, 
such that $K_s \subseteq K_t$ for any pair $s \leq t$ in $\bR$. Notice that the 
results from subsection~\ref{sub:simplicial_complexes} also hold for 
filtered simplicial complexes. 
\end{definition}

\begin{definition}
Let $K$ be a filtered simplicial complex and $n \geq 0$. We define
the $n$-persistent homology of $K$ as the composed functor $\Ho_n(K) \colon \bR \rightarrow \Vect$. 
We will also denote this by $\PH_n(K)$. 
\end{definition}

\begin{definition}
A \emph{persistence module} $\V$ is a covariant functor $\bV
: \bR \rightarrow \Vect$. That is, to any $r \in \bR$,  $\V$ assigns a vector space in $\Vect$
which will be denoted either by $\V(r)$ or  $\V^r$. Additionally, to
any pair of real numbers $s \leq t$, there is a linear morphism $\V(s \leq t)
: \V^s \rightarrow \V^t$. These morphisms satisfy $\V(s
\leq s) = \Id_{\V^s}$ for any $s \in \bR$, and the relation
$\V(r \leq t) = \V(s \leq t) \circ \V(r \leq s)$ for all $r \leq s \leq t$ in
$\bR$. Given two persistence modules $\V$ and $\W$, a \emph{morphism of
persistence modules} is a natural transformation $\morph{f}{\V}{\W}$.
Thus, for any pair of real numbers $s \leq t$, there is a commuting square
$$
\xymatrix@C=1.5cm{
\V^s \ar[r]^{\V(s \leq t)} \ar[d]_{f^s} &
\V^t \ar[d]_{f^t} \\
\W^s \ar[r]^{\W(s \leq t)} &
\W^t.
}
$$
We denote by \PMod~the category of persistence modules and persistence morphisms. 
\end{definition}
Hence, whenever we are speaking about the \emph{naturality} of $f$ we will be referring to the
commutative square above. We say that a persistence morphism $f:\V \rightarrow \W$ 
is an \emph{isomorphism} whenever $f_t$ is an isomorphism for all $t \in \bR$. We write $\V \simeq \W$
to denote that $\V$ is isomorphic to $\W$. 
A pointwise finite dimensional (p.f.d.) persistence module is a functor $\morph{\V}{\bR}{\vect}$, 
where $\vect$ is the category of finite vector spaces.
\begin{definition}
A sequence of persistence modules and persistence morphisms
$$
\xymatrix{
\cdots \ar[r] &
\V^{k-1} \ar[r]^{f^{k-1}}       & 
\V^{k}   \ar[r]^{f^{k}}         & 
\V^{k+1} \ar[r]^{f^{k+1}}       & 
\cdots
}
$$
is a \emph{chain of persistence modules} whenever $f^{k} \circ f^{k-1} = 0$
for all $k \in \Z$. 
\end{definition}

\begin{example}
A special class of persistence modules will be the \emph{interval modules}. For any pair
of real numbers $s \leq t$, we denote by $\GInt{s}{t}$ the interval module
\begin{equation}
    \GInt{s}{t}(r) = \begin{cases}
    \F  \qquad   \mbox{for $r \in [s, t)$,} \\
    0   \qquad   \mbox{otherwise.}
    \end{cases}
    \label{def:right-open-interval}
\end{equation}
The morphisms  $\GInt{s}{t}(a \leq b)$ will be the identity for
any two $a,b \in [s,t)$ and will be $0$ otherwise.  
\end{example}

Notice that in an
analogous way we could have defined barcodes $\rI(s,t)$ over intervals of
the form $[s,t]$, $(s,t]$ or $(s,t)$, with $s \leq t$.
For a given interval $\GInt{s}{t}$, the values $s$ and $t$ will be called respectively
the \emph{birth} and \emph{death} values.
Whenever $\bV$ is a p.f.d persistence module, then it can be uniquely
decomposed as a direct sum of barcodes  $\bigoplus_{i \in J}\GInt{s_i}{t_i}$,
as shown in~\citep{ChazalSilvaGlisseOudot}.
This means  that there is an isomorphism
$\V \simeq \bigoplus_{i \in J}\GInt{s_i}{t_i}$ of persistence modules.
This will be called the barcode decomposition of $\V$.
Throughout this text, we will mainly be studying
persistence modules that decompose into barcodes of the form (\ref{def:right-open-interval}).

\section{Homology of Persistence Modules}
\label{sec:persistence-modules}

\subsection{Barcode Bases}

In this section we will use the result from~\citep{ChazalSilvaGlisseOudot}
to introduce barcode bases. Our aim will be
to come up with an efficient way of computing homology in this category.
At the end we will introduce an
algorithm for computing images and kernels, and we 
will evaluate its computational complexity. 

\begin{definition}[Barcode Basis]
\label{def:barcode_basis}
A barcode basis $\B$ of a persistence module $\V$ is a choice of an
isomorphism, $\morph{\beta}{\bigoplus_{i \in I} \GInt{a_i}{b_i}}{\V}$. Each
direct summand of $\beta$ defines a restricted morphism from a barcode
$\morph{\beta_i}{\GInt{a_i}{b_i}}{\V}$, and will be called a \emph{barcode
generator}. We will usually denote a barcode basis $\B$ by the set of barcode
generators $\B = \gset{\beta}{i}{I}$.  

Within the context of definition~\ref{def:barcode_basis}, we would like to make some
notational remarks.  
\end{definition}

\begin{itemize}
\item Given a barcode generator $\beta \in \B$, we write 
$\RORel{\beta}{a}{b}$ to denote that $\beta$ is a natural transformation $\beta : \GInt{a}{b} \rightarrow \V$. 
In this case we say that $\beta$ is \emph{associated} to the interval $[a, b)$.  
\item Notice that if we choose $\beta \in \B$ with $\RORel{\beta}{a}{b}$ and $r \in \bR$, we have a 
linear transformation $\beta(r) : \GInt{a}{b}(r)\rightarrow \V(r)$. 
In particular, since $\GInt{a_i}{b_i}(r)$ is either $0$ or $\F$, the morphism $\beta(r)$ is 
uniquely determined by the image $\beta(r)(1_\F) \in \V(r)$. For the sake of simplicity, 
we will write
$\beta(r) \in \V(r)$ instead of $\beta(r)(1_\F) \in \V(r)$.  
\item For any given $r \in \bR$, we define the pointwise basis in $r$ by 
$$
\B^r = \Big\{\, \be{i} : i \in I,\;\; \be{i}(r) \neq 0 \,\Big\} \subseteq \B.
$$
In this case,  if $\beta \in \B^{r}$ and $\RORel{\beta}{a_\beta}{b_\beta}$, then
$a_\beta \leq r < b_\beta$ by naturality of $\beta$.
Also, evaluating all the elements from $\B^r$ on $1_\F$ leads to a vector
base $\B^r(1_\F)$ for $\V(r)$. 
\end{itemize}

\begin{remark}
We can think of a persistence module $\V$ as a sheaf over $\bR$, where $\bR$ is endowed with the 
topology where the open sets are
 either the intervals $[a, \infty)$ or $(a, \infty)$, for any $a \in \bR$. Thus the restriction morphism
$\rho_{b, a} : [a, \infty) \rightarrow [b, \infty)$ with values $a \leq b$ in $\bR$, correspond
to $\V(a \leq b) : \V^a \rightarrow \V^b$.
A barcode base is a set of global sections of the sheaf $\V$,
such that they form pointwise base of the vector spaces $\V^r$, for all $r \in \bR$. That is,
$\B \subset \V$ forms a barcode base for $\V$ if and only if
$\B^r$ forms a base of $\V^r$ for all $r \in \bR$.
\end{remark}

To make our work less cumbersome, we will only focus on very simple persistence modules. In fact,
these modules will be the only ones relevant for our later applications.

\begin{definition}
A \emph{tame} persistence module $\V$, is a p.f.d. persistence module that admits a finite barcode basis
$\B = \finset{\beta}{i}{1}{N}$ and all the barcodes $\be{i}$ are associated to an interval of the form
$[a_i, b_i)$, with $a_i \leq b_i \in \bR$ for all $1 \leq i \leq N$.
\end{definition}

Thus, whenever we are speaking about tame persistence modules, we will assume that
$\GInt{a}{b}$ denotes a barcode over $[a, b)$. The first problem one encounters
when working with a barcode basis $\B = \gset{\beta}{i}{I}$ is taking linear combinations.
Whenever we take a barcode generator $\be{1} \in \B$ we have a natural transformation
$\morph{\be{1}}{\GInt{a_1}{b_1}}{\V}$. However, this property does not need to hold for general
sums. For example, suppose that $\RORel{\be{1}}{0}{2}$ and $\RORel{\be{2}}{1}{3}$ are two barcode generators
from $\B$, then we can define the sum pointwise $\gamma(r) \coloneqq \be{1}(r) +  \be{2}(r) \in \V(r)$ for all
$r \in \bR$. Even though this $\gamma$ is well defined, this assignment does not define 
a natural transformation. This is depicted in Figure~\ref{fig:not-natural-sum}, where 
we have $\gamma(1) = \be{1}(1) + \be{2}(1) \neq \be{1}(1)  = \V(0 \leq 1) \gamma (0)$.
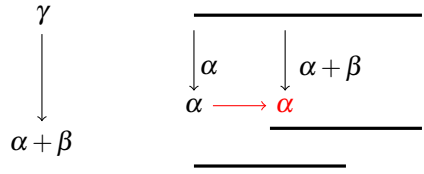
\begin{figure}
    \begin{center}
        \begin{tikzpicture}
        \node (ga) at (0,2) {$\gamma$};
        \node (albe) at (0,0.3) {$\alpha + \beta$};
        \draw[->] (ga) -- (albe);
        \begin{scope}[xshift=2cm]
        \draw[very thick] (0,0) -- (2,0);
        \draw[very thick] (1,0.5) -- (3,0.5);
        \draw[very thick] (0,2) -- (3,2);
        \draw[->] (0,1.8) -- (0, 1);
        \node (sum2) at (0.2,1.3) {$\alpha$};
        \node (al1) at (0,0.8) {$\alpha$};
        \node[color=red] (al2) at (1.2,0.8) {$\alpha$};
        \draw[color= red, ->] (al1) -- (al2);
        \draw[->] (1.2,1.8) -- (1.2, 1);
        \node (sum2) at (1.8,1.3) {$\alpha + \beta$};
        \end{scope}
        \end{tikzpicture}
        \caption{Sum of barcode generators might not be natural. }
        \label{fig:not-natural-sum}
    \end{center}
\end{figure}

More generally, assume that $\beta_1 \sim [a_1,b_1)$
and $\beta_2 \sim [a_2,b_2)$ with $a_1 < a_2 < b_2 < b_1$. In this case
$(\beta_1 + \beta_2)(s) = \V(r \leq s) (\beta_1 + \beta_2)(r)$ is not satisfied for some $r \leq s$ in $\bR$.
Something that we can do in order to `correct' this situation is to `chop down' the non-natural part.
That is, we consider the following operation
$$
\be{1}\barSum \be{2} \coloneqq \bOne{a_2}\big(\,\beta_1 + \beta_2\,\big)
$$
where we have used the step function $\morph{\bOne{s}}{\bR}{\F}$ defined by:
$$
\bOnePar{s}{t} =
\begin{cases}
0_\F \quad \mbox{if $t < s$,} \\
1_\F \quad \mbox{if $s \leq t$.}
\end{cases}
$$
Notice that in this case $\be{1} \barSum \be{2}$ is associated to the interval $\GInt{a_2}{b_1}$.
More generally, suppose we want to compute
 $\bigBarSum_{1 \leq j \leq m} k_j \beta_j$ with $k_j \in \F $ and $\beta_j \sim [a_i,b_i)$ for all $1 \leq j \leq m$. Taking
 into account the definition  of $\barSum$ for two terms and also the fact that $\bOne{a} \bOne{b} = \bOne{{\rm max}(a,b)}$,
we can inductively extend the definition:
 $$
\bigBarSum_{1 \leq j \leq m} k_j \beta_j \coloneqq \bOnePar{A}{\sum_{1 \leq j \leq m} k_j \beta_j}
 $$
 where
 $$
 A = \cmax\big\{\, a_j : 1 \leq j \leq m, \;\;k_j \neq 0 \,\big\}.
 $$
 In the trivial case of $k_j = 0$ for all $1 \leq j \leq m$,
 we will set to zero the above definition.
 On the other hand, considering the value
 $$
 B = \cmax\big\{\, b_j : 1 \leq j \leq m, \;\; k_j \neq 0 \,\big\}
 $$
 we have that $\bigBarSum_{1 \leq j \leq m} k_j \beta_j$ is associated to $\GInt{A}{B}$.
 Of course these $\beta_j$ do not need to form a basis, so perhaps the previous
 sum could have a more adjusted associated interval.
This operation will be of great use when working with persistence morphisms.

\begin{remark}
Let us introduce some properties of the step function $\bOne{s}$ for $s \in \bR$. 
For any $\RORel{\beta}{a}{b}$ one has $\bOne{s}\beta = 0$ whenever
$b \leq s$.  Also, suppose that $\{\RORel{\beta}{0}{2}, \RORel{\gamma}{0}{1}, \RORel{\tau}{0}{1}\}$
is a basis of $\V$. Then $\bOnePar{s}{\beta + \gamma}$
and $\bOnePar{s}{\beta + \tau}$ are linearly independent for all $s < 1$, but are equal for all $s \geq 1$.
Throughout this section it will be important to have these basic properties in mind.
\end{remark}

\begin{remark}
Alternatively, one can recall the definition of persistence modules as $\F[x^{[0,\infty)}]$-modules. Note that $\F[x^{[0,\infty)}]$ denotes the polynomial
ring with $\F$-coefficients and allowing all powers $x^r$ for $r \in [0, \infty)$, where by convention $x^0 = 1_\F$. Given a persistence module $\V$,
one defines a \emph{barcode vector} as a morphism of $\F[x^{[0,\infty)}]$-modules of either form:
$$
v : (x^{a_1}) \rightarrow \V \hspace{1cm}
v : \dfrac{(x^{a_1})}{(x^{b_1})} \rightarrow \V
$$
where $a_1,b_1 \in [0, \infty)$ and $a_1 < b_1$. These barcode vectors do not need to be injective. We denote by $\cV(\V)$ the set of all barcode vectors of $\V$.
The operation $\barSum$ and the step function $\bOne{s}$ have interpretations for
$\F[x^{[0,\infty)}]$-modules. Consider $v \in \cV(\V)$ with $\RORel{v}{a_1}{b_1}$ as defined above.
Then, the step function  $\bOne{s}:\cV(\V)\rightarrow \cV(\V)$ assigns
$\bOne{s}(v) = v$ whenever $s \leq a_1$, or the barcode vector;
$$
\bOne{s}(v) \coloneqq v|_{(x^{s})} : \dfrac{(x^{s})}{(x^{b_1})} \rightarrow \V
$$
for $a_1 < s$; the latter is defined to be the restriction of $v$ to the subideal $(x^s)$, since one has that $x^s = x^{a_1} x^{s - a_1}$. Suppose we have another
barcode vector $\RORel{w}{a_2}{b_2}$
with values $a_2  < b_2$  in $[0,\infty)$. Then, one defines the barcode sum $\barSum : \cV(\V) \times \cV(\V) \rightarrow \cV(\V)$ by setting
$$
v \barSum w \coloneqq  \bOne{A}(v) + \bOne{A}(w): \dfrac{(x^{A})}{(x^{B})} \rightarrow \V,
$$
where $A = \max\{a_1,a_2\}$ and $B=\max\{b_1,b_2\}$.
 In this context, a barcode basis $\B$ is a set
of barcode vectors such that:
\begin{enumerate}
	\item $\B$ generates $\cV(\V)$
	\item $\B$ are $\F[x^{[0,\infty)}]$-linearly independent with respect to $\barSum$.
\end{enumerate}
\end{remark}

Notice that, while there is a uniquely determined barcode decomposition of $\V$, the particular
choice of a basis is not unique. This is analogous to the case of vector spaces, where a vector
space can admit multiple bases but has always the same dimension. The main reason why we are introducing
barcode bases is because we would like to work with morphisms between persistence modules $f : \V \rightarrow \W$.
Even though the respective barcode decompositions of $\V$ and $\W$ are determined,  there is no unique `assignment'
of barcodes induced by $f$. In fact, it was proven in \citep[prop 5.10]{BauerLesnick2015} that matchings between barcodes of $\V$ and $\W$
cannot be defined in a functorial way.  The following example will illustrate this principle.

\begin{example}
    Consider two persistence modules:
    $$
    \V = \GInt{1}{4}
    \hspace{1cm}
    \W = \GInt{0}{3} \oplus \GInt{0}{2}
    $$
    with barcode bases $\{\RORel{\al{1}}{1}{4}\}$ and $\{\RORel{\be{1}}{0}{3}, \RORel{\be{2}}{0}{2}\}$ respectively.
    Let
    $\morph{f}{\V}{\W}$ be a morphism given by
    $f(\alpha_1) = \bOnePar{1}{\beta_1}$.
    Suppose that we had chosen an alternative barcode basis
    for $\W$ defined by setting $\beta_1' = \beta_1 + \beta_2$
    and $\beta_2' = \beta_2$.
    Thus, in this case we have that $f(\alpha_1) = \bOnePar{1}{\beta_1' - \beta_2'}$.
    Notice that the morphism $f$ will relate different intervals
     depending on the chosen barcode
    bases.  Therefore when studying morphisms we should not work directly with barcodes, but
barcode bases instead. This is illustrated in Figure~\ref{fig:not-canonical-rel}.
\end{example}

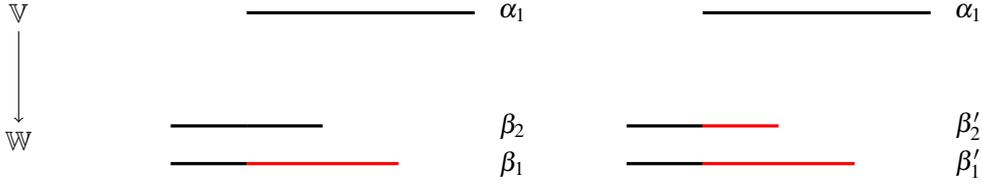
\begin{figure}
    \begin{center}
        \begin{tikzpicture}
        \node (v) at (0,2) {$\V$};
        \node (w) at (0,0.3) {$\W$};
        \draw[->] (v) -- (w);
        \begin{scope}[xshift=2cm]
        \draw[very thick] (0,0) -- (1,0);
        \draw[very thick] (0,0.5) -- (1,0.5);
        \draw[very thick, color=red] (1,0) -- (3,0);
        \draw[very thick] (1,0.5) -- (2, 0.5);
        \draw[very thick] (1,2) -- (4,2);
        \node (be1) at (4.5,0) {$\beta_1$};
        \node (be2) at (4.5,0.5) {$\beta_2$};
        \node (al1) at (4.5,2) {$\alpha_1$};
        \end{scope}
        \begin{scope}[xshift=8cm]
        \draw[very thick] (0,0) -- (1,0);
        \draw[very thick] (0,0.5) -- (1,0.5);
        \draw[very thick, color=red] (1,0) -- (3,0);
        \draw[very thick, color=red] (1,0.5) -- (2, 0.5);
        \draw[very thick] (1,2) -- (4,2);
        \node (beta1) at (4.5,0) {$\beta_1'$};
        \node (beta2) at (4.5,0.5) {$\beta_2'$};
        \node (alpha1) at (4.5,2) {$\alpha_1$};
        \end{scope}
        \end{tikzpicture}
        \caption{There is no canonical way of relating barcodes from $f$.}
        \label{fig:not-canonical-rel}
    \end{center}
\end{figure}

Let $\morph{f}{\V}{\W}$ be a morphism of tame persistence modules and 
consider two bases $\A$ and
$\B$ for $\V$ and $\W$ respectively.
For each barcode generator $\RORel{\alpha}{a}{b}$ in $\A$, we would like to define the image
$f(\alpha)$ in terms of $\B$. First notice that we have an expression for
$f(\alpha)(a)$ in terms of $\B^a$, since this forms a basis for $\W^a$. Thus there exist
coefficients $k_{\beta, \alpha} \in \F$ for all $\beta \in \B^{a}$ such that
$$
f(\alpha)(a) = \sum_{\beta \in \B^{a}} k_{\beta, \alpha} \beta(a).
$$
Therefore, since $f$ is natural, we can write the image $f(\alpha)$ as
$$
f(\alpha) \coloneqq \bOnePar{a}{\sum_{\beta \in \B^{a}} k_{\beta, \alpha} \beta}.
$$
Recall that if $\beta \in \B^{a}$ and $\RORel{\beta}{a_\beta}{b_\beta}$, then
$a_\beta \leq a \leq b_\beta$ by naturality of $\beta$.
 Also notice that if $b_\beta > b$, then $k_{\beta, \alpha} = 0$,
since otherwise $f$ would not be natural as a persistence morphism. Thus we can define the subset
of $\B^{a}$ associated to $\alpha$:
$$
\B(\alpha) \coloneqq \Big\{\, \beta  : \beta \in \B, \;\; \beta(a) \neq 0, \;\; \beta(b) = 0 \Big\} \subseteq \B^a \subseteq \B
$$
where $\RORel{\alpha}{a}{b}$. The set $\B(\alpha)$ contains the barcode generators $\beta \in \B$ such that the coefficients $k_{\beta, \alpha}$
might be non-zero.
This gives us a sharper description of $f(\alpha)$:
$$
f(\alpha) = \bOnePar{a}{\sum_{\beta \in \B(\alpha)} k_{\beta, \alpha} \beta} = \bOnePar{a}{\bigBarSum_{\beta \in \B(\alpha)} k_{\beta, \alpha} \beta}.
$$
Notice that there is no distinction between the expression above using $\barSum$ and the ordinary sum. This is because
we have already `cut away' the non-natural part of the sum.
In particular, if $\bigBarSum_{\beta \in \B(\alpha)} k_{\beta, \alpha} \beta$ is associated to $\GInt{A}{B}$, then we
can deduce that $A \leq a$ and $B \leq b$. To visualize this, consider Figure~\ref{fig:f-restricted-to-interval}
illustrating the restriction of $f$ to some
barcode $\barRel{\alpha}{a}{b}$. By pointwise-linearity and naturality of $f$, we have that 
$$
f \left( \bigBarSum_{\alpha \in \A} k_\alpha \alpha \right) =
\bigBarSum_{\alpha \in \A} k_\alpha f(\alpha)
$$
where $k_\alpha \in \F$ for all $\alpha \in \A$.

\begin{figure}
    \begin{center}
        \begin{tikzpicture}
            \draw[very thick] (0,0) -- (4,0);
            \draw[very thick] (0.3,0.2) -- (3.5,0.2);
            \draw[very thick] (0.4,0.4) -- (4.5,0.4);
            \draw[very thick] (2,1) -- (4,1);
            \draw[very thick] (2,2)--(5,2);
            \draw[very thick] node (I) at (6.5,2) {$\barRel{\alpha}{a}{b}$};
            \filldraw[fill= red, opacity = 0.3] (2,2) -- (4,2) -- (4,1) --
            (2,1) -- cycle;
            \draw node at (3,1.5) {$f$};
            \draw node at (6.5,1) {$\barRel{f(\alpha)}{a}{B}$};
            \draw[very thick, color= red] (2,0) -- (4,0);
            \draw[very thick, color= red] (2,0.2) -- (3.5,0.2);
            \draw (6.5,0) node {$\W$};
        \end{tikzpicture}
        \caption{Restriction of $f$ to  $\barRel{\alpha}{a}{b}$.
        Notice that  $\barRel{f(\alpha)}{a}{B}$ with $B \leq b$.}
        \label{fig:f-restricted-to-interval}
    \end{center}
\end{figure}
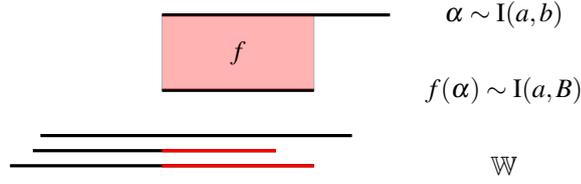

\subsection{Computing Kernels and Images}

Let $\morph{f}{\V}{\W}$ be a morphism of tame persistence modules.
The kernel of $f$ is a persistence module $\Ker(f)$ together with an
inclusion morphism $\injmorph{j}{\Ker(f)}{\V}$, such that
$\Ker(f)^r \simeq \Ker(f^r)$ for all $r \in \bR$. Therefore, if
$\K$ is a barcode basis for the kernel, then for each
barcode generator $\RORel{\kappa}{a}{b}$ we have
$$
j(\kappa) = \bOnePar{a}{\bigBarSum_{\alpha \in \A(\kappa)} k_{\alpha, \kappa} \alpha},
$$
where $k_{\alpha, \kappa}  \in \F$  for all $\alpha \in \A$. By `finding' a
basis for the kernel we mean that we want to find  $j(\K)$  in terms of the basis $\A$. Since $j$ is an
injection and $\RORel{\kappa}{a}{b}$ is a basis generator, the image $j(\kappa)$ needs to be non-zero along the
interval $[a,b)$. Thus if $\barSum_{\alpha \in \A(\kappa)} (k_{\alpha, \kappa}  \alpha)$ is associated to the interval $[A,B)$,
then  by injectivity $b \leq B$. On the other hand, since the image $j(\kappa)$ is associated to $[A,B)$, then $B \leq b$ by naturality of $j$,
whence we obtain the equality $B = b$. Since $\A(\kappa)$ is finite,
there must exist some $\alpha \in \A(\kappa)$ with
death value $b_\alpha = b$.

The image of $f$, which will be denoted as $\Img(f)$, is a persistence
module together with a projection  $\surjmorph{q}{\V}{\Img(f)}$, such that
$\Img(f)^r \simeq \Img(f^r)$ for all $r \in \bR$. Let $\A$ be a basis for $\V$ and
$\I$ be a basis for $\Img(f)$. Then for each generator $\gamma \in \I$ with
$\RORel{\gamma}{a}{b}$, there exist
coefficients $c_{\alpha, \gamma} \in \F$ such that
$$
\gamma = \bigBarSum_{\alpha \in \A} c_{\alpha, \gamma} \ q(\alpha).
$$
Notice that there was no need to multiply the above expression by $\bOne{a}$, since $\gamma$ is in the
image of $f$. Thus, by finiteness of $\cA$,  there must exist some $\alpha \in \cA$ such that $q(\alpha)$ has birth value $a$.
Additionally, we will have an inclusion $\iota : \Img(f) \hookrightarrow \W$ such that $f = \iota \circ q$.
Notice that
 $\injmorph{\iota}{\Img(f)}{\W}$ being an inclusion,  will have properties analogous to those discussed
 for the kernel.
Hence, there will be
coefficients $e_{\beta, \gamma} \in \F$ satisfying the equation:
$$
\iota(\gamma) = \bOnePar{a}{\bigBarSum_{\beta \in \B} e_{\beta, \gamma} \beta}.
$$
for each $\gamma \in \I$ with
$\RORel{\gamma}{a}{b}$.
Putting these two together and considering the image of $f$ in terms of $\B$, that is the equality $\left(f(\alpha)\right)_{\A} =
\left(   \beta \right)_{\B } \left(b_{\beta, \alpha} \right)_{\B \times \A}$,
we get the matrix equation:
$$
\left( e_{\beta, \gamma} \right)_{\B \times \cI} =
\left(b_{\beta, \alpha} \right)_{\B \times \A} \left( c_{\alpha, \gamma} \right)_{\A \times \cI}.
$$
In general, we start from the matrix $(b_{\beta, \alpha})_{\B \times \A}$ and will proceed to find
the coefficients $e_{\beta, \gamma}$ and $c_{\alpha, \gamma}$. This will be done by a process very similar to a Gaussian elimination.
Each non-zero column $\left( e_{\beta, \gamma} \right)_{\B}$ will lead to a barcode generator of the image. Its counterpart
$\left( c_{\alpha, \gamma} \right)_{\A}$ will lead to a basis for
 the kernel of $f$, although we will need to perform an additional Gaussian elimination.
See Figure~\ref{fig:decomposition-f-barcodes} for an illustration of these concepts.
Some of these observations have already been studied in \citep{BauerLesnick2015}.

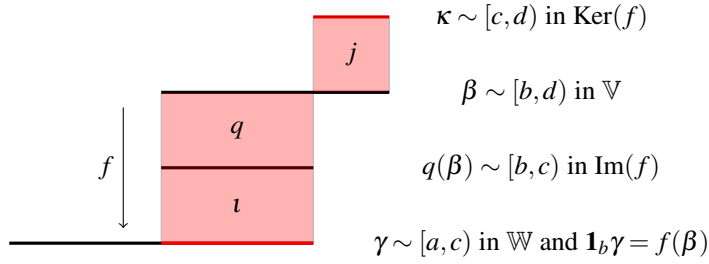
\begin{figure}
    \begin{center}
        \begin{tikzpicture}
        \draw[very thick] (0,0) -- (4,0);
        \draw[very thick] (2,1) -- (4,1);
        \draw[very thick] (2,2)--(5,2);
        \draw[very thick, color = red] (4,3)--(5,3);
        \draw[very thick] node (I1) at (7,3) {$\RORel{\kappa}{c}{d}$ in $\Ker(f)$};
        \draw[very thick] node (I) at (7,2) {$\RORel{\beta}{b}{d}$ in $\V$};
        \draw node at (7,1) {$\RORel{q(\beta)}{b}{c}$ in $\Img(f)$};
        \filldraw[fill= red, opacity = 0.3] (2,2) -- (4,2) -- (4,0) -- (2,0) -- cycle;
        \filldraw[fill= red, opacity = 0.3] (4,2) -- (5,2) -- (5,3) -- (4,3) -- cycle;
        \draw[very thick, color= red] (2,0) -- (4,0);
        \draw (7,0) node {$\RORel{\gamma}{a}{c}$ in $\W$ and $\bOne{b}\gamma = f(\beta)$};
        \draw node at (4.5, 2.5) {$j$};
        \draw node at (3, 1.5) {$q$};
        \draw node at (3, 0.5) {$\iota$};
        \draw[->] (1.5,1.8)--(1.5,0.2) node[midway, left] {$f$};
        \end{tikzpicture}
        \caption{Barcodes related to $f$ with parameters $a < b < c < d$ in $\bR$.}
        \label{fig:decomposition-f-barcodes}
    \end{center}
\end{figure}

A point to notice is that
there is a natural ordering for $\B$. For any pair of barcode generators $\RORel{\alpha}{a}{b}$ and
$\RORel{\beta}{c}{d}$, we will write $\alpha < \beta$ whenever $a < c$ or  when we have that $a = c$ and $d < b$.
As before, consider two finite barcode bases
$\A = \finset{\alpha}{i}{0}{n}$ and $\B = \finset{\beta}{j}{0}{m}$ for $\V$ and
$\W$ respectively.
Additionally, suppose that both $\A$ and $\B$ have
 total orderings. That is, even if two barcode generators are associated to the same interval $\RORel{\al{1},\al{2}}{a}{b}$,
 we have already made a choice $\al{1} < \al{2}$.
 Then we consider $M = \left( f(\alpha_1), \ldots, f(\alpha_n) \right)$ the matrix of $f$ in the bases $\A$ and $\B$.
The aim will be to transform $M$ performing left to right column additions so that we obtain a matrix
$$
\I = \left(\, 
    f(\al{1}) \,\middle\vert\, 
    f(\al{2}) \barSum k_{2,1}f(\al{1})  \,\middle\vert\,
    \ldots  \,\middle\vert\,
    f(\alpha_n) \barSum \bigBarSum_{i = 1}^{n-1} k_{n, i} f(\al{i}) 
    \,\right)
$$
for suitable $k_{i,j} \in \F$ and $0 \leq i < j \leq n$. This $\I$ will have the property that 
its non-zero columns form a basis for $\Img(f)$.  Also, we can find coefficients $q_{i,j} \in \F$ and $c_j \in \F$ for all $0 \leq i < j \leq n$,
such that the set
$$
\K = \left\{\bOnePar{c_j}{\alpha_j \barSum  \bigBarSum_{i = 1}^{j-1} q_{j, i} \alpha_i}\right\}_{0 \leq j \leq n}
$$
forms a basis for $\Ker(f)$. In the following we will present an
algorithm obtaining such bases. First we will go through an illustrative example encoding some of the basic principles of the procedure.

\begin{example}
Consider two persistence modules
 $$
 \V \simeq \GInt{1}{5} \oplus \GInt{1}{4} \oplus \GInt{2}{5}, \hspace{0.5cm}
 \W \simeq \GInt{0}{5} \oplus \GInt{0}{3} \oplus \GInt{1}{4}
 $$
 with barcode bases $(\al{1},\al{2},\al{3})$ and $(\be{1},\be{2},\be{3})$ respectively. Let the
 morphism $\morph{f}{\V}{\W}$ be given by the $\B \times \A$ matrix:
 $$
 F = \left(\begin{array}{c|ccc}
 \  & \al{1} & \al{2} & \al{3} \\ \hline
 \be{1} & 0 & 0 & 1\\
 \be{2} & 1 & 0 & 0 \\
 \be{3} & 1 & 1 & 1
 \end{array}\right).
 $$
 Then we will have matrices associated to $f$ which are constant between pairs of consecutive parameters in
 $-\infty < 0 < 1 < 2 < 3 < 4 < 5 < \infty$. Since $f$ is zero along $(-\infty, 1)$, we start considering
 the matrix associated to $[1,2)$, together with its reduction by columns,
 $$
F^1 =
 \left(\begin{array}{c|cc}
 \ & \al{1} & \al{2} \\ \hline \be{1} & 0 & 0 \\ \be{2} & 1 & 0\\ \be{3} & 1 & 1
 \end{array}\right), \hspace{1cm} {\rm reduced} \rightarrow
 R(F^1) =
 \left(\begin{array}{c|cc}
 \ & \al{1} & \al{2}-\al{1} \\ \hline \be{1} & 0 & 0 \\ \be{2} & 1 & -1 \\ \be{3} & 1 & 0
 \end{array}\right).
 $$
 Next we consider $F^2$ along
 the interval $[2,3)$, which will inherit the previous reduction. Since a generator on the domain is being born,
 we add a new column at the right end of $F^2$.  This will be
 reduced by subtracting the first two columns from the last one,
 $$
 F^2 = \left(\begin{array}{c|ccc}
 \  & \al{1} & \al{2}-\al{1} & \al{3} \\ \hline
 \be{1} & 0 & 0 & 1\\
 \be{2} & 1 & -1 & 0 \\
 \be{3} & 1 & 0 & 1
 \end{array}\right), \hspace{1cm} {\rm reduced} \rightarrow
 R(F^2) =
 \left(\begin{array}{c|ccc}
 \  & \al{1} & \al{2}-\al{1} & \al{3}-\al{2} \\ \hline
 \be{1} & 0 & 0 & 1\\
 \be{2} & 1 & -1 & 0 \\
 \be{3} & 1 & 0 & 0
 \end{array}\right).
 $$
 Now, we compute the matrix
 $F^3$ of $f$ along $[3,4)$.
 We start from $R(F^2)$ and we take out the second row,
 since its associated interval ends at $3$. Thus, we obtain $F^3$ which is already reduced,
 $$
 F^3 =
 \left(\begin{array}{c|ccc}
 \ & \al{1} & \al{2}-\al{1} & \al{3}-\al{2} \\ \hline
 \be{1} & 0 & 0 & 1 \\ \be{3} & 1 & 0 & 0
 \end{array}\right).
 $$
 Since the second column is zero this means that a barcode has finished on the image. Thus, we add
 $f(\al{2}-\al{1}) = -\bOne{1}\beta_2$ into $\I$. Additionally, we add $\bOnePar{3}{\al{2}-\al{1}}$ into $\K$.
 The next interval to consider is $[4,5)$. Now,
 before looking at the matrix $F^4$ of $f$ along $[4, 5)$, we consider $\bOne{4}(\K)$. That is, we look at the element $\bOne{4}(\alpha_2 - \alpha_1) = -\bOne{4}(\alpha_1)$.
 This already tells us extra information about the kernel of $F^4$. We check this when we compute $F^4$,
  $$
 F^4 =
 \left(\begin{array}{c|ccc}
 \ & \al{1} & \al{3}-\al{2} \\ \hline
 \be{1} & 0 & 1
 \end{array}\right).
 $$
 Notice that we do not need to add  $\bOne{4} (\al{1})$ into $\K$, since $\bOne{4} (\al{2}-\al{1}) = -\bOne{4}(\al{1}) \in \bOne{4}(\K)$.
 On the other
 hand, we add $f(-\alpha_1)= -\bOnePar{1}{\be{2} + \be{3}}$ into $\I$. Again, this is because a barcode generator has finished in the image of $f$. The reason why we are adding $f(-\alpha_1)$ instead of $f(\alpha_1)$ to $\I$, is because we detected this barcode from $\bOne{4}(\K)$. 
 Finally, since all generators in $\A$ die at $5$, we add $f(\al{3}-\al{2}) = \bOnePar{2}{\be{1}}$ into $\I$. Altogether we
 have obtained a basis for the kernel
 $$
 \K = \big\{\, \bOnePar{3}{\al{2}-\al{1}} \,\big\},
 $$
 and also a basis for the image:
 $$
 \I = \big\{\, -\bOnePar{1}{\be{2}},\;\; -\bOnePar{1}{\be{2} + \be{3}},\;\; \bOnePar{2}{\be{1}}\,\big\}.
 $$
Therefore we obtain isomorphisms $\Ker(f) \simeq \GInt{3}{5}$ and
$\Img(f) \simeq \GInt{1}{3} \oplus \GInt{1}{4} \oplus \GInt{2}{5}$ with respective barcode bases
 $\K$ and $\I$. This is illustrated on Figure~\ref{fig:PHEx-morphism-ker-im}. In practise, 
 instead of adding elements to $\I$, we will set $\I$ to be equal to $f(\A)_\B$ and perform
 the corresponding reductions until we obtain a basis for the image of $f$. 
 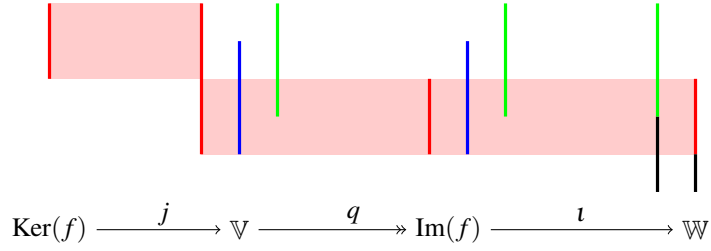
\begin{figure}
    \begin{center}
        \begin{tikzpicture}
        \fill[red, opacity = 0.2] (0,1.5) -- (2,1.5)--(2,2.5)--(0,2.5)--cycle;
        \fill[red, opacity = 0.2] (8.5,1.5) -- (2,1.5)--(2,0.5)--(8.5,0.5)--cycle;
        \draw[very thick, red] (0,1.5) -- (0,2.5);
        \node (ker) at (0, -0.5) {$\Ker(f)$};
        \draw[very thick, red] (2,0.5) -- (2,2.5);
        \draw[very thick, blue] (2.5,0.5) -- (2.5,2);
        \draw[very thick, green] (3,1) -- (3,2.5);
        \node (v) at (2.5,-0.5) {$\V$};
        \draw[very thick, red] (5,0.5) -- (5,1.5);
        \draw[very thick, blue] (5.5,0.5) -- (5.5,2);
        \draw[very thick, green] (6,1) -- (6,2.5);
        \node (im) at (5.25,-0.5) {$\Img(f)$};
        \draw[very thick] (8,0) -- (8,1);
        \draw[very thick, green] (8,1) -- (8,2.5);
        \draw[very thick, red] (8.5,0.5) -- (8.5,1.5);
        \draw[very thick] (8.5,0) -- (8.5,0.5);
        \draw[very thick, blue] (9,0.5) -- (9,2);
        \node (w) at (8.5,-0.5) {$\W$};
        \node (j) at (1.5, -0.3) {$j$};
        \draw[->] (ker) -- (v);
        \node (q) at (4, -0.3) {$q$};
        \draw[->>] (v) -- (im);
        \node (i) at (7, -0.3) {$\iota$};
        \draw[->] (im) -- (w);
        \end{tikzpicture}
        \caption{Decomposition of barcodes in image, kernel, domain and codomain of $f: \V \rightarrow \W$. The
        colors correspond to the different generators associated to $\I$ and $\K$.}
        \label{fig:PHEx-morphism-ker-im}
    \end{center}
\end{figure}
\end{example}

\subsection{Algorithm}

Here, we present an algorithm performing the above procedure. Suppose that $\morph{f}{\V}{\W}$ is a
morphism between two tame persistence modules. Let $\A$ and $\B$ be barcode bases for $\V$ and $\W$ respectively.
Suppose also that we know $f(\A)_{\B}$, the  matrix associated to $f$ with respect to barcode bases $\A$ and $\B$. 
We want to find a barcode basis for the image $\I$, and a barcode basis for the kernel $\K$.
In order to achieve this, $\cI$ will  start being set to be equal to the $|\B|\times |\A|$ matrix $f(\A)_\B$. 
Performing left to right column additions will lead to the nonzero columns of $\I$ forming a basis for the image. 
On the other hand, $\K$ will be a matrix with $|\A|+1$ rows and whose number of columns will `grow' as 
the computations develop. 
The extra row will be used for storing the parameter of the multiplying step function. 
Notice that $\K$ will have at most $|\A|$ columns, which is useful to know if 
we wanted to preallocate space for speed. 

Notice that there exist values $ -\infty = a_0 < a_1 <  \cdots < a_{n+1} = \infty$ such that $f$ is constant along $[a_i, a_{i+1})$ for each $0 \leq i \leq n$.
We start by computing the values $a_i$ for all $0 \leq i \leq n$.
We will denote by $\A^{a_i}(j)$ the index $1 \leq \A^{a_i}(j) \leq |\A|$ of the $j$-element from $\A^{a_i}$.
Also given a matrix $A$, we will denote by $A[j]$ the $j^{\rm th}$ column of $A$. 
The matrices $R^i$ will denote the successive Gaussian reductions as we increase the parameter $0 \leq i \leq n+1$. 
That is, we start with $\widetilde{R}^0$ which will be
the $|\B^{a_0}|\times |\A^{a_0}|$-matrix of $f$ along the interval $a_0 < a_1$, 
then we reduce it to $R^0$. 
Simultaneously, we perform exactly the same transformations to $\I$. 
In order to track these additions performed, we will use a $|\A|\times |\A|$ matrix $T$. This $T$ 
 will be the identity matrix $\Id_{|\A|}$. 
Thus, whenever we add columns in $R^0$ we perform the same additions in $T$. 
On the other hand, if some column $R^0[j]$ becomes zero, where $1 \leq j \leq |\A^{a_0}|$, 
we add $T[\A^{a_0}(j)]$ at the right end of the matrix of kernels $K^0$. 
Additionally, we append $T[\A^{a_0}(j)]$ to $\K$, with associated step function coefficient $a_0$. 
Since we require $\K$ to be linearly independent, we will introduce a
set $\texttt{pivots}$ for tracking the pivots of the elements in $\K$.  For each $T[\A^{a_0}(j)]$ that we add into $\K$, we add $\A^{a_0}(j)$ into $\texttt{pivots}$. Note that in this
first step there will be no repeated elements in $\texttt{pivots}$ and the matrix $\K$ will be already reduced. Once we finish, we jump to the next parameter $a_1$.

Let us go through the procedure for $a_1$. For this, we add or take out
rows and columns from $R^0$ and $K^0$ according to the life of each generator in $\A$ and $\B$; these changes are stored into $\widetilde{R}^1$ and $\widetilde{K}^1$, respectively.
Observe that $\widetilde{K}^1$ might not be reduced. 
Since we would like to obtain a basis for the kernel of $f$,
we reduce it further to $K^1=R(\widetilde{K}^1)$, performing the same additions on $\K$. 
Next we proceed to reduce $\widetilde{R}^1$. 
There is a trick we can use here to speed up the computations. 
For each $j$-column in $K^1$, if the pivot $p$ of the column is such that $\cA^{a_1}(p)$ is not in $\texttt{pivots}$, this means that the $p$ column in $\widetilde{R}^1$ will become zero after reducing.
Then we set $\widetilde{R}^1[p]$ to zero directly, substitute the column
$\I[\A^{a_1}(p)]$ by $f(K^{1}[j])$, and add $\cA^{a_1}(p)$ into $\texttt{pivots}$.
Here by $f(K^1[j])$ we mean the result after adding the columns from $f(\A)_\B$ with
coefficients given by $K^{1}[j]$. Notice that this is the same as performing left to right column additions to the column $\I[\A^{a_1}(p)]$, although 
we also permit this column be multiplied by a non-zero coefficient $t \in \F \setminus \{ 0\}$. 
After performing these preprocessing tasks, we reduce $\widetilde{R}^1$ into $R^1$, repeating the same transformations to  $T$ and $\I$.
Then we examine $R^1$, and look for columns $1 \leq j \leq |\A^{a_1}|$ of $R^1$, such that
$R^1[j] = 0$ and also $\A^{a_1}(j)$ is not in $\texttt{pivots}$. For each such column $j$, we 
append $T[\A^{a_1}(j)]$ at the right end of $K^1$, and also
into $\K$ with birth value $a_1$. Finally, we add $\A^{a_1}(j)$ into $\texttt{pivots}$. This finishes the iteration for $a_1$.

We repeat the previous step again for parameters $a_1 < a_2 < \cdots <a_n$.
On the $i$ iteration, where $2 \leq i \leq n$,  we assume
that  we have well defined matrices $R^{i-1}$ and $K^{i-1}$.
As before, we update these matrices into a  $\B^{a_{i}} \times \A^{a_{i}}$-matrix $\widetilde{R}^{i}$, and a matrix with $|\A^{a_{i}}|$ columns $\widetilde{K}^{i}$.
These updates are performed by adding and deleting columns as the barcodes from $\A$ and $\B$ are born or die respectively.
The rest of the procedure for $a_i$ is exactly as we outlined for $a_1$ earlier. Notice that while we are on the $i$th step, both
$K^i$ and $\K$ will have the same number of columns. 
An outline of this procedure is shown in Algorithm~\ref{cde:img-kernel}.

\begin{algorithm}
	\caption{\texttt{image\_kernel}}
	\hspace*{\algorithmicindent} \textbf{Input}: $\A$, $\B$, $f(\A)_{\B}$ \\
	\hspace*{\algorithmicindent} \textbf{Output}: $\K$, $\I$
	\begin{algorithmic}[1]
		\State Find values $a_0 < a_1 <  \cdots < a_n$ where a barcode generator dies, or is born in $\A$ or $\B$
		\State Set $\I = f(\A)_{\B}$, $\K = \emptyset$, $T = \Id_{|\A|}$, $R^{-1} = \emptyset$, $K^{-1} = \emptyset$, $\texttt{pivots} = \emptyset$
		\For{ $0 \leq i \leq n$}
			\State Update $\widetilde{R}^i$ and $\widetilde{K}^i$ from $R^{i-1}$, and $K^{i-1}$ respectively
			\State Reduce $\widetilde{K}^i$ obtaining $K^i$. Perform the same reductions to $\K$
			\For{each $j$-column of $K^i$ with pivot  $p$ such that $\cA^{a_i}(p) \notin \texttt{pivots}$}
				\State $\I[\A^{a_i}(p)] \gets f(K^i[j])$
				\State $\widetilde{R}^i[p] \gets 0$
				\State Add $\cA^{a_i}(p)$ into $\texttt{pivots}$
			\EndFor
			\State Reduce $\widetilde{R}^i$ into $R^i$. Perform the same reductions to $T$ and $\I$
			\For{$1 \leq j \leq |\A^{a_i}|$}
				\If{$R^i[j] = 0$ \& $\A^{a_i}(j) \notin \texttt{pivots}$}
					\State Append $T[\A^{a_i}(j)]$ at end of $K^i$, and also at $\K$ with step coefficient $a_i$
					\State Add $\A^{a_i}(j)$ into \texttt{pivots}
				\EndIf
			\EndFor
		\EndFor
		\State \Return $\K$ and $\I$ (optionally return $T$ for preimages)
	\end{algorithmic}
	\label{cde:img-kernel}
\end{algorithm}

\begin{proposition}
Algorithm~\ref{cde:img-kernel} computes $\K$ and $\I$ bases for the kernel and image of $f$. 
Furthermore, it takes at most $\cO(n M |\A|^2)$ time, where $M = \cmax(|\A|, |\B|)$.
\end{proposition}
\begin{proof}
The key observation is that $\K$ forms a barcode basis for $\Ker(f)$ if and only if $\K^r$ is a 
basis for $\Ker(f)^r$ for all $r \in \R$. Now, notice that $\K^r$ generates $\Ker(f)^r$ since all kernel
elements were sent to $\K$. On the other hand, each $\K^r$ is a linearly independent set, since we
have performed Gaussian eliminations that ensured this. Similarly, for any $r \in \R$ we have that
$\I^r$ generates all the columns from $f(\A)_\B^r$, and thus it generates $\Img(f)^r$. We have also 
ensured linear independence of $\I^r$ by the Gaussian elimination process. Thus, $\I$ is a barcode
basis for $\Img(f)$. 
 
Let us compute the complexity of the algorithm. We start noticing that $n$ comes from the outer loop. Then the Gaussian reduction of $\widetilde{K}^i$ might take at most
$\cO(|\A|^3)$ time. On the other hand the reduction of $\widetilde{R}^i$ might take $\cO(|\B||\A|^2)$ time.
	The first inner loop will  take less than $\cO(|\A|(\clog(|\A|) + |\A| |\B|))$ time, where the multiplying $|\A|$ comes from the iteration.
	Within round brackets, the first term comes from checking \texttt{pivots} by a hash table or similar, whereas the second comes from computing $f(K^i[j])$. The second inner loop takes $\cO(|\A|\clog(|\A|))$ time, where $|\A|$ is for the iteration and $\clog(|\A|)$ for checking \texttt{pivots}. 
Putting all together we obtain the following complexity:
\begin{eqnarray*}
    &   n\Big(\cO\big(\,|\A|^3\,\big) + \cO\big(\,|\B||\A|^2\,\big) + 
        \cO\big(\,|\A|(\clog(|\A|) + |\A||\B|)\,\big) + \cO\big(\,|\A|\clog(|\A|)\,\big)\Big) \\
    & {}=n\cO\big(\,M|\A|^2\,\big) = \cO\big(\,n M |\A|^2\,\big),
\end{eqnarray*}
where $M = \cmax\big(\,|\A|, |\B|\,\big)$.
\end{proof}

Notice that this estimate $\cO(n M |\A|^2)$ can be improved in practice, since most of the values $a_i$ will only indicate a single birth or death on either the image or the kernel. Coming up with an efficient algorithm for this task is an interesting question that goes beyond the scope of this paper.

\subsection{Computing Quotients}
\label{sec:persQuotients}

Now we consider the problem of computing quotients.
Suppose that we have inclusions $\bH \subseteq \bG \subseteq \V$ of finite persistence modules of  dimensions $H \leq G \leq B$ respectively.
Furthermore, suppose that $\cH = \finset{h}{j}{1}{H}$, $\cG =  \finset{g}{k}{1}{G}$ and $\B =  \finset{b}{i}{1}{B}$ are barcode bases
 for $\bH$, $\bG$ and $\V$ respectively. The aim will be to find a barcode basis for $\bG / \bH$. For each generator $h_j \in \cH$, we will use the superscript
 notation $\RORel{h_j}{a_j^h}{b_j^h}$ for the associated interval. Also $\cH$ will be ordered in a way such that $a^h_i \leq a^h_j$ whenever $1 \leq i \leq j \leq H$. The same conventions will
 be used for the bases $\B$ and $\cG$. Then there exists a matrix $M = (m_{i,j})_{H,B} \in \cM_{H \times B}(\F)$ such that
 $$
 \left(\begin{array}{c}
 	h_1 \\ h_2 \\ h_3 \\ \vdots \\ h_H
 \end{array} \right) =
 \left(\begin{array}{cccc}
 	\bOne{a_1^h} & 0 & \cdots & 0 \\
 	0 & \bOne{a_2^h} & \cdots & 0 \\
 	\vdots & \vdots & \ddots & \vdots \\
 	0 & 0  & \cdots & \bOne{a_H^h}
 \end{array} \right) \cdot
 \left(\begin{array}{cccc}
 	m_{1,1} & m_{1,2} & \cdots & m_{1,B} \\
 	m_{2,1} & m_{2,2} & \cdots & m_{2,B} \\
 	\vdots & \vdots & \ddots & \vdots \\
 	m_{H,1} & m_{H,2} & \cdots & m_{H,B} \\
 \end{array}\right) \cdot
 \left(\begin{array}{c}
 	b_1 \\ b_2 \\ b_3 \\ \vdots \\ b_B
 \end{array} \right)
 $$
 where the operation $\barSum$ is implicit on the equation. We will write this in the more compact form $h = \bOne{\cH} M b$.
 Similarly, there exists a matrix $N \in \cM_{G \times B}(\F)$ such that $g = \bOne{\cG} N b$.

Consider the inclusions 
$\iota_\bH : \bH \hookrightarrow \bV$ and
$\iota_\bG : \bG \hookrightarrow \bV$, and 
define the morphism $\cR : \bH \times \bG \rightarrow \bV$
by $\cR \coloneqq \iota_\bH + \iota_\bG$. Thus we have that:
$$
\left( 	\cR(\cH \mid \cG)			\right)_{\B} = 
\left(	\iota_\bH (\cH) \mid \iota_\bG (\cG)	\right)_{\B} =
\left(	M^T \bOne{\cH} \mid N^T \bOne{\cG} 	\right) =
\left(	(\bOne{\cH} M)^T \mid (\bOne{\cG} N)^T 	\right) = 
\left(	h^T \mid g^T 				\right)_{\B}
$$
Hence, in order to compute a basis for the quotient, all that we need to do is apply
\texttt{image\_kernel} to the matrix $\left(	h^T \mid g^T \right)_{\B}$. 
The last $|\cG|$ nontrivial generators from $\I$ lead to a basis for 
$\bG/\bH$. 

\subsection{Homology of Persistence Modules}
\label{sub:persistence-module-homology}

Consider a chain of tame persistence modules:
\begin{equation}
\label{eq:chainPersMod}
\xymatrix{
0 &
\bV_0 \ar[l] &
\bV_1 \ar[l]_{d_1} &
\bV_2 \ar[l]_{d_2} &
\cdots \ar[l]	&
\bV_n \ar[l]_{d_n},
}
\end{equation}
where each term has basis $\B_j$ for $0 \leq j \leq n$.
Then applying \texttt{image\_kernel} we will obtain bases $\I_{j-1}$ and $\K_j$ for the image and kernel of
$d_j$ for all $0 \leq j \leq n$. Proceeding as on the previous section, we consider matrices $(\cR_j(\I_j \mid \K_j))_{\B_j}$ and apply again \texttt{image\_kernel}.
This leads to bases $\cQ_j$ for the homology for all $0 \leq j \leq n$. 

\section{A review on the Mayer-Vietoris spectral sequence}
\label{sec:MayerVietorisSS}

In this section, we give an introduction to the Mayer-Vietoris spectral sequence. 
This section has no claims of originality.
These ideas come mainly from \citep{BoTu1982, McCleary2001}. 
The reason for including this is because we think it beneficial to outline a minimal, 
self-contained explanation of the procedure.
Also, we will be using this as a necessary background for Section~\ref{sec:persistence-mayer-vietoris}. 
For simplicity we will focus on ordinary homology over a field $\F$. 
Later on we will extend these ideas to the case of persistent homology over a field. 

Let $K$ be a simplicial complex, and $\U=\{ U_i\}_{0 \leq i \leq m}$ be a cover of $K$ by
subcomplexes. Suppose that we want to compute the homology of $K$ from the cover elements. 
Then a naive approach to solving the problem, would be to compute
the homology groups $\Ho_n(U_i)$, and proceed by adding all of them back together:
\begin{equation}
\label{eq:dummy_sum}
\Ho_n(K) = \bigoplus_{0 \leq i \leq m} \Ho_n(U_i).
\end{equation}
Unfortunately, this is hardly ever true and we will need to find other ways of
dealing with this merging of information. To introduce the distributed problem, we
forget about simplicial complexes, and go back to the domain of
topological spaces and open covers.

\subsection{The Mayer-Vietoris theorem}

Consider torus $\bT^2$ covered by two cylinders $U$ and $V$, as illustrated in
Figure~\ref{fig:two-covers-torus}. 
Then one sees that equality~(\ref{eq:dummy_sum}) does not hold 
in dimensions $0$ and~$2$:
$$
\Ho_0(\bT^2) = \F \ncong \F \oplus \F = \Ho_0(U) \oplus \Ho_0(V), \hspace{2cm}
\Ho_2(\bT^2) = \F \ncong 0 = \Ho_2(U) \oplus \Ho_2(V). 
$$
In order to amend this, one has to look at the information given by the intersection $U\cap V$. 
This information comes as \emph{identifications} and new \emph{loops}.
For example, $U$ and $V$ are connected through the intersection. Also, the loop going around each cylinder
$U$ and $V$ is identified in the intersection. These identifications are performed by taking the quotient
$$
I_n \coloneqq {\rm coker}\Big(\, \Ho_n(\,U\cap V\,)\rightarrow \Ho_n(U) \oplus \Ho_n(V) \,\Big)
$$
for all $n \geq 0$. 
Where the previous morphism is the \cech~differential $\delta_1^n : S_n(U\cap V) \rightarrow S_n(U) \oplus S_n(V)$. Additionally, the $1$-loops in the intersection merge to the same loop when included in each cylinder $U$ or $V$. This situation creates a $2$-loop or `void', see
Figure~\ref{fig:two-covers-torus}. Thus we have the $n$-loops detected by the kernel
$$
L_n \coloneqq \Ker\Big(\, \Ho_{n-1}(U\cap V) \rightarrow \Ho_{n-1}(U) \oplus \Ho_{n-1}(V) \,\Big)
$$
for all $n \geq 0$.  Notice that $n$-loops are found by $n-1$ information on the intersection. Putting all together, we have that 
$$
\Ho_0(\bT^2) \cong I_0 \cong \F, \qquad
\Ho_1(\bT^2) \cong I_1 \oplus L_1 \cong \F \oplus \F, \qquad
\Ho_2(\bT^2) \cong L_2 \cong \F.
$$
This leads to the expected result
$$
\Ho_k(\bT^2) \cong 
\begin{cases}
\F              \qquad  &\mbox{for $k = 0, 2$,} \\
\F \oplus \F    \quad  &\mbox{for $k = 1$,} \\
0               \qquad  &\mbox{otherwise.} 
\end{cases}
$$ 
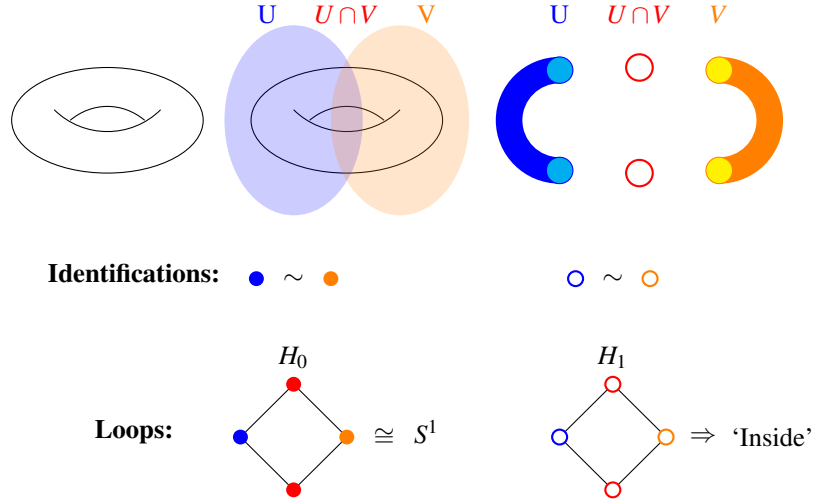
\begin{figure}
    \begin{center}
	\begin{tikzpicture}[scale=0.7]
	    \begin{scope}
    		\draw (-0.5,0) ellipse (1.8cm and 1cm);
    		\draw (-1.5,0.2) to[out=315,in=225] (0.5,0.2);
    		\draw (-1.2,0) to[out=40,in=140] (0.2,0);
	    \end{scope}
	    \begin{scope}[xshift=4cm]
    	        \draw (0,0) ellipse (1.8cm and 1cm);
    		\draw (-1,0.2) to[out=315,in=225] (1,0.2);
    		\draw (-0.7,0) to[out=40,in=140] (0.7,0);
    	    	\node[color=blue] at (-1.5,2) {U};
    	    	\fill[color=blue, opacity=0.2] (-1,0) ellipse (1.3 and 1.8);
    	    	\node[color=red] at (0,2) {$U\cap V$};
    	    	\fill[color=orange, opacity=0.2] (1,0) ellipse (1.3 and 1.8);
    	    	\node[color=orange] at (1.5,2) {V};
	    \end{scope}
	    \begin{scope}[xshift=8cm]
    	    	\node[color=blue] at (0,2) {U};
    	    	\fill[color=blue, thick] (90:1.2) arc (90:270:1.2) -- ++(270:-0.5) -- (-90:0.7) arc (-90:-270:0.7) -- cycle;
    	    	\filldraw[fill=cyan, draw=blue] (-90:0.95) circle (0.25);
    	    	\filldraw[fill=cyan, draw=blue] (90:0.95) circle (0.25);
	    \end{scope}
	    \begin{scope}[xshift=9.5cm]
    	    	\node[color=red] at (0,2) {$U\cap V$};
    	    	\draw[color=red, thick] (0,1) circle (0.25);
    	    	\draw[color=red, thick] (0,-1) circle (0.25);
	    \end{scope}
	    \begin{scope}[xshift=11cm]
    	    	\node[color=orange] at (0,2) {$V$};
    	    	\fill[color=orange, thick] (90:1.2) arc (90:-90:1.2) -- ++(-90:-0.5) -- (-90:0.7) arc (-90:90:0.7) -- cycle;
    	    	\filldraw[draw=orange,fill=yellow] (-90:0.95) circle (0.25);
    	    	\filldraw[draw=orange,fill=yellow] (90:0.95) circle (0.25);
	    \end{scope}
	    \begin{scope}[yshift=-6cm]
    	    	\node at (0,0.1) {\bf Loops:};
	    \end{scope}
	    \begin{scope}[xshift=2cm, yshift=-6cm]
	        \node at (1,1.5) {$H_0$};
	        % edges
    	    	\draw (0,0) -- (1,1) -- (2,0) -- (1,-1) -- cycle;
	        % points
    	    	\fill[blue] (0,0) circle (4pt);
    	    	\fill[red] (1,1) circle (4pt);
    	    	\fill[red] (1,-1) circle (4pt);
    	    	\fill[orange] (2,0) circle (4pt);
    	    	\node at (2.7,0) {$\cong$};
    	    	\node at (3.5,0.1) {$S^1$};
	    \end{scope}
	    \begin{scope}[xshift=8cm, yshift=-6cm]
	        \node at (1,1.5) {$H_1$};
	        % edges
    	    	\draw (0,0) -- (1,1) -- (2,0) -- (1,-1) -- cycle;
	        % points
    	    	\filldraw[draw=blue, fill=white, thick] (0,0) circle (4pt);
    	    	\filldraw[draw=red, fill=white, thick] (1,1) circle (4pt);
    	    	\filldraw[draw=red, fill=white, thick] (1,-1) circle (4pt);
    	    	\filldraw[draw=orange, fill=white, thick] (2,0) circle (4pt);
    	    	\node at (2.7,0) {$\Rightarrow$};
    	    	\node at (4,0) {`Inside'};
	    \end{scope}
	    \begin{scope}[yshift=-3cm]
    	    	\node at (0,0.1) {\bf Identifications:};
	    \end{scope}
	    \begin{scope}[xshift= 2cm, yshift=-3cm]
    	    	\fill[blue] (0.3,0) circle (4pt);
    	    	\node at (1,0) {$\sim$};
    	    	\fill[orange] (1.7,0) circle (4pt);
	    \end{scope}
	    \begin{scope}[xshift= 8cm, yshift=-3cm]
    	    	\filldraw[draw=blue, fill=white, thick] (0.3,0) circle (4pt);
    	    	\node at (1,0) {$\sim$};
    	    	\filldraw[draw=orange, fill=white, thick] (1.7,0) circle (4pt);
	    \end{scope}
	\end{tikzpicture}
        \caption{Torus covered by a pair of cylinders $U$ and $V$.}
    	\label{fig:two-covers-torus}
    \end{center}
\end{figure}

On a more theoretical level, what we have presented here is commonly known as the Mayer-Vietoris theorem. 
We can think of each homology group $\Ho_n(U\cup V)$ as a filtered object,
$$
\{0\} = F_{-1}\big(\,\Ho_n(U\cup V)\,\big)
\subset  F_0\big(\,\Ho_n(U\cup V)\,\big)
\subset  F_1\big(\,\Ho_n(U\cup V)\,\big) = \Ho_n(U\cup V).
$$
Then,  the Mayer-Vietoris theorem gives us the expressions
for the different ratios between consecutive filtrations,
$$
	F_0\big(\,\Ho_n(U\cup V)\,\big)  = I_n, \qquad
	\dfrac{F_1\big(\,\Ho_n(U\cup V)\,\big) }{ F_0\big(\,\Ho_n(U\cup V)\,\big)}   = L_n.
$$
In particular, since we are working with vector spaces  we obtain 
$$
\Ho_n(U \cup V) \cong I_n \oplus L_n
$$
for all $n \geq 0$. 

The above discussion gives rise to the \emph{total} chain complex,
$$
\Tot_n (\cS_*)= S_n(V) \oplus S_n(U) \oplus S_{n-1}(U \cap V),
$$
with morphism $d^{\rm Tot}_n = (d, d, d - \delta_1)$ for all $n \geq 0$. Notice that the first two morphisms do not change components, whereas the
third encodes the `merging' of information. This last morphism is represented by red arrows on the diagram:
$$
\xymatrix@R=0.7cm{
\Tot_{n+1} (\cS_*) \ar[d]^{d_{n+1}^\Tot}  & \cong & 
S_{n+1}(U)  \oplus  S_{n+1}(V)  \ar[d]^{d_{n+1}} 
& \oplus 	& S_{n}(U \cap V) \ar@[red][d]^{d_{n}}  \ar@[red][lld]^{\delta_1} \\
\Tot_{n} (\cS_*) \ar[d]^{d_{n}^\Tot} & \cong & 
S_{n}(U)  \oplus  S_{n}(V) \ar@[red][d]^{d_{n}} 
& \oplus 	& S_{n-1}(U \cap V) \ar[d]^{d_{n-1}} \ar@[red][lld]^{\delta_1} \\
\Tot_{n-1} (\cS_*) & \cong & S_{n-1}(U)  \oplus  S_{n-1}(V) & \oplus 	& S_{n-2}(U \cap V) 
}
$$
where the rectangle of red arrows is commutative. In particular, this implies that 
 $d_n^{\rm Tot} \circ d_{n+1}^{\rm Tot} = 0$ for all $n\geq 0$.
Computing the homology with respect to the total differentials and using the previous characterization 
of $I_n$ and $L_n$, one obtains
$$
\Ho_n (\Tot_*(\cS_*))  \cong  I_n \oplus L_n \cong \Ho_n(K).
$$
This result will be further generalized in proposition~\ref{prop:convSS}.

\subsection{The Mayer-Vietoris spectral sequence}
\label{subsec:MVss}

After this digression, we move back to a simplicial complex $K$ with a covering
$\U=\{ U_i\}_{i=0}^m$ by subcomplexes. In this case,  we need 
to take into account all the intersections between different subcomplexes. We can
extend the intuition from the previous subsection, by recalling the definition 
of the $(n, \U)$-\cech chain complex given on the preliminaries. 
Stacking all these sequences on top of each other,  and also multiplying differentials in odd rows by $-1$,
we obtain a diagram:
$$
\xymatrix@C=0.6cm{
	&
	\ \ar[d]	&
	\ \ar[d]	&
	\ \ar[d]	&
	\ \ar[d]	&
	\\
	0		&
	S_2(K)	\ar[l]	\ar[d]^{d}	&
	\bigoplus \limits_{\sigma \in \Delta_0^m} S_2(U_\sigma) \ar[l]_(0.6){\delta_0}
  \ar[d]^{d} &
	\bigoplus \limits_{\sigma \in \Delta_1^m} S_2(U_\sigma) \ar[l]_{\delta_1}
  \ar[d]^{d} &
	\bigoplus \limits_{\sigma \in \Delta_2^m} S_2(U_\sigma) \ar[l]_{\delta_2}
  \ar[d]^{d} &
	\cdots \ar[l]
	\\
	0		&
	S_1(K)	\ar[l]	\ar[d]^{d}	&
	\bigoplus \limits_{\sigma \in \Delta_0^m} S_1(U_\sigma) \ar[l]_(0.6){-\delta_0}
  \ar[d]^{d} &
	\bigoplus \limits_{\sigma \in \Delta_1^m} S_1(U_\sigma)  \ar[l]_{-\delta_1}
  \ar[d]^{d} &
	\bigoplus \limits_{\sigma \in \Delta_2^m} S_1(U_\sigma) \ar[l]_{-\delta_2}
  \ar[d]^{d} &
	\cdots \ar[l]
	\\
	0		&
	S_0(K)	\ar[l]	\ar[d]&
	\bigoplus \limits_{\sigma \in \Delta_0^m} S_0(U_\sigma) \ar[l]_(0.6){\delta_0}
  \ar[d] &
	\bigoplus \limits_{\sigma \in \Delta_1^m} S_0(U_\sigma) \ar[l]_{\delta_1}
  \ar[d] &
	\bigoplus \limits_{\sigma \in \Delta_2^m} S_0(U_\sigma) \ar[l]_{\delta_2}
  \ar[d] &
	\cdots \ar[l]
	\\
	 & 0 & 0 & 0 & 0
}
$$

This leads to a \emph{double complex} $(\cS_{*,*}, \bar{\delta}, d)$ defined as 
$$
\cS_{p,q} \coloneqq \bigoplus \limits_{\sigma \in \Delta_p^m} S_q(U_\sigma)
$$
for all $p,q \geq 0$, and also $\cS_{p,q} \coloneqq 0$ otherwise.  We denote $\bar{\delta} = (-1)^q \delta$, the 
\cech~differential multiplied by a $-1$ on odd rows. The reason for this change of sign is because we want $\cS_{*,*}$ to be a double 
complex, in the sense that the following equalities hold:
\begin{equation}
\label{eq:double-complex-equations}
\bar{\delta} \circ \bar{\delta} = 0, \hspace{1cm}
d \circ d = 0, \hspace{1cm}
\bar{\delta} \circ d + d \circ \bar{\delta} = 0.
\end{equation}
Since $\cS_{*,*}$ is a double complex, we can study the associated chain complex $\cS_*^\Tot$, commonly known as the \emph{total complex}. This is 
formed by taking the sums of anti-diagonals
$$
\cS_n^\Tot \coloneqq \bigoplus \limits_{p+q = n} \cS_{p,q}
$$
for each $n\geq 0$. The differentials on the total complex are defined by $d^{\rm Tot} = d + \bar{\delta}$, 
which satisfy  $d^{\rm Tot} \circ d^{\rm Tot} = 0$ from 
equations~(\ref{eq:double-complex-equations}), see Figure~\ref{diag:totalComplex} for a depiction of this. 
Later, in proposition~\ref{prop:convSS},  we will prove that $\Ho_n(K) \cong \Ho_n(\cS^\Tot_*)$ for all 
$n \geq 0$. The problem still remains difficult, since computing $\Ho_n(\cS^\Tot_*)$ directly might be even harder than
computing $\Ho_n(K)$. The key is that there is a divide and conquer method which allows us to break apart the calculation
of $\Ho_n(\cS^\Tot_*)$ into small, computable steps. 
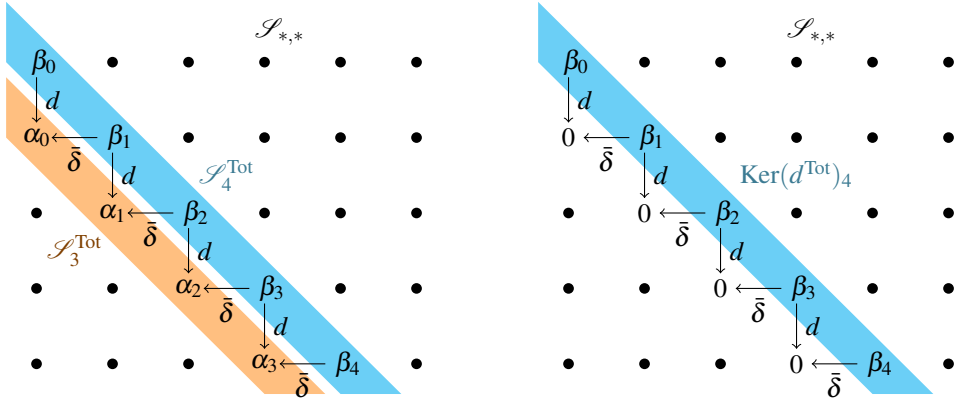
\begin{figure}
	\begin{center}
		\begin{tikzpicture}
		%draw lattice
		\foreach \x in {0,1,...,5}{
			\foreach \y in {0,1,...,4}{
        			\fill (\x,\y) circle (2pt);
        		}
		}
		%colour diagonals
		\fill[cyan!50] (4,-0.4)--(4.8,-0.4)--(-0.4,4.8)--(-0.4,4)--cycle;
		\fill[orange!50] (3,-0.4)--(3.8,-0.4)--(-0.4,3.8)--(-0.4,3)--cycle;
		%draw arrows 
		\draw[->] (3.8,0)--(3.2,0) node[midway, below] {$\bar{\delta}$};
		\draw[->] (2.8,1)--(2.2,1) node[midway, below] {$\bar{\delta}$};
		\draw[->] (1.8,2)--(1.2,2) node[midway, below] {$\bar{\delta}$};
		\draw[->] (0.8,3)--(0.2,3) node[midway, below] {$\bar{\delta}$};
		\draw[->] (0,3.8)--(0,3.2) node[midway, right] {$d$};
		\draw[->] (1,2.8)--(1,2.2) node[midway, right] {$d$};
		\draw[->] (2,1.8)--(2,1.2) node[midway, right] {$d$};
		\draw[->] (3,0.8)--(3,0.2) node[midway, right] {$d$};
		%alphas and betas
		\node at (4.1,0) {$\beta_4$};
		\node at (3.1,1) {$\beta_3$};
		\node at (2.1,2) {$\beta_2$};
		\node at (1.1,3) {$\beta_1$};
		\node at (0.1,4) {$\beta_0$};
		\node at (3,0) {$\alpha_3$};
		\node at (2,1) {$\alpha_2$};
		\node at (1,2) {$\alpha_1$};
		\node at (0,3) {$\alpha_0$};
		%now we include labels
		\draw[orange!50!black] (0.5,1.5) node {$\cS^\Tot_3$};
		\draw[cyan!50!black] (2.5,2.5) node {$\cS^\Tot_4$};
		\draw (3.2,4.4) node {$\cS_{*,*}$};
		\begin{scope}[xshift=7cm]
		%draw lattice
		\foreach \x in {0,1,...,5}{
			\foreach \y in {0,1,...,4}{
        			\fill (\x,\y) circle (2pt);
        		}
		}
		%colour diagonals
		\fill[cyan!50] (4,-0.4)--(4.8,-0.4)--(-0.4,4.8)--(-0.4,4)--cycle;
		\fill[white] (3,-0.4)--(3.8,-0.4)--(-0.4,3.8)--(-0.4,3)--cycle;
		%draw arrows 
		\draw[->] (3.8,0)--(3.2,0) node[midway, below] {$\bar{\delta}$};
		\draw[->] (2.8,1)--(2.2,1) node[midway, below] {$\bar{\delta}$};
		\draw[->] (1.8,2)--(1.2,2) node[midway, below] {$\bar{\delta}$};
		\draw[->] (0.8,3)--(0.2,3) node[midway, below] {$\bar{\delta}$};
		\draw[->] (0,3.8)--(0,3.2) node[midway, right] {$d$};
		\draw[->] (1,2.8)--(1,2.2) node[midway, right] {$d$};
		\draw[->] (2,1.8)--(2,1.2) node[midway, right] {$d$};
		\draw[->] (3,0.8)--(3,0.2) node[midway, right] {$d$};
		%zeros and betas
		\node at (4.1,0) {$\beta_4$};
		\node at (3.1,1) {$\beta_3$};
		\node at (2.1,2) {$\beta_2$};
		\node at (1.1,3) {$\beta_1$};
		\node at (0.1,4) {$\beta_0$};
		\node at (3,0) {$0$};
		\node at (2,1) {$0$};
		\node at (1,2) {$0$};
		\node at (0,3) {$0$};
		%now we include labels
		\draw[cyan!50!black] (3,2.5) node {$\Ker(d^\Tot)_4$};
		\draw (3.2,4.4) node {$\cS_{*,*}$};
		\end{scope}
		\end{tikzpicture}
	\end{center}
	\caption{$\cS_{*,*}$ represented as a lattice for convenience. On the left, the total complex $\cS^\Tot$ associated to $\cS_{*,*}$. 
	Here $(\beta_0, \ldots, \beta_4) \in \cS^\Tot_4$ maps to $(\alpha_0, \ldots, \alpha_3) \in \cS^\Tot_3$, where 
	$\alpha_i = d(\beta_i) + \bar{\delta}(\beta_{i+1})$ for all $0 \leq i \leq 3$. On the right, the kernel $\Ker(d^\Tot)_4$.}
	\label{diag:totalComplex}
\end{figure}

Let us start by computing the kernel $\Ker(d^\Tot_n)$, which is depicted in Figure~\ref{diag:totalComplex}. 
Recall that we will be  working with vector spaces and linear maps all throughout. 
Let $s = (s_{k,n-k})_{0 \leq k \leq n} \in \cS^\Tot_n$ be in $\Ker(d^\Tot_n)$. Then $s$ will 
satisfy the equations $d(s_{k,n-k}) = -\bar{\delta}(s_{k+1,n-k-1})$ for all $0 \leq k < n$. Thus, one 
can obtain kernel elements by considering  subspaces
$\GK_{p,q} \subseteq \cS_{p,q}$.  The subspace $\GK_{p,q}$ is composed of elements $s_{p,q} \in \cS_{p,q}$ such that $d(s_{p,q})=0$, and there exists a sequence $s_{p-r,q+r} \in \cS_{p-r,q+r}$ 
 satisfying equations $d(s_{p-r,q+r}) = -\bar{\delta}(s_{p-r-1,q+r+1})$ for all $0 < r \leq p$. Notice that $\GK_{p,q}$ is 
a subspace of $\cS_{p,q}$ since both $d$ and $\bar{\delta}$ are linear.  
We will see that one has (non-canonical) isomorphisms,
\begin{equation}\label{eq:non_canonical_kernel}
\Ker(d^\Tot_n) \cong \bigoplus \limits_{p+q=n} \GK_{p,q}.
\end{equation}
This is depicted in Figure~\ref{diag:totalKernel}.
It turns out that this is true only when we are working with vector spaces. 
Later, we will work with a more general case where such isomorphisms
do not hold. This will be known as the \emph{extension problem}. 
\begin{figure}
	\begin{center}
		\begin{tikzpicture}
		\begin{scope}[yshift=5cm]
			%draw lattice
			\foreach \x in {0,1,...,3}{
				\foreach \y in {0,1,...,3}{
        				\fill (\x,\y) circle (2pt);
        			}
			}
			%white points
			\fill[white] (3,0) circle (3pt);
			\fill[white] (2,1) circle (3pt);
			\fill[white] (1,2) circle (3pt);
			\fill[white] (0,3) circle (3pt);
			\fill[white] (2,0) circle (3pt);
			\fill[white] (1,1) circle (3pt);
			\fill[white] (0,2) circle (3pt);
			%fill cyan
			\fill[cyan!50] (0.4,2.2)--(0.4,3)--(-0.4,3.8)--(-0.4,3)--cycle;
			%surround blue area
			\draw[blue!50!black, very thick] (0.4,2.2)--(0.4,3)--(-0.4,3.8)--(-0.4,3)--cycle;
			%labels
			\node at (3.1,0) {$0$};
			\node at (2.1,1) {$0$};
			\node at (1.1,2) {$0$};
			\node at (2,0) {$0$};
			\node at (1,1) {$0$};
			\node at (0,2) {$0$};
			\node at (0.1,3) {$\beta_0$};
			%draw arrows 
			\draw[->] (2.8,0)--(2.2,0) node[midway, below] {$\bar{\delta}$};
			\draw[->] (1.8,1)--(1.2,1) node[midway, below] {$\bar{\delta}$};
			\draw[->] (0.8,2)--(0.2,2) node[midway, below] {$\bar{\delta}$};
			\draw[->] (0,2.8)--(0,2.2) node[midway, right] {$d$};
			\draw[->] (1,1.8)--(1,1.2) node[midway, right] {$d$};
			\draw[->] (2,0.8)--(2,0.2) node[midway, right] {$d$};
			%now we include labels
			\node[blue!50!black] at (0.5,3.5) {$\GK_{0,3}$};
		\end{scope}
		\begin{scope}[yshift=5cm, xshift=5cm]
			%draw lattice
			\foreach \x in {0,1,...,3}{
				\foreach \y in {0,1,...,3}{
        				\fill (\x,\y) circle (2pt);
        			}
			}
			%white points
			\fill[white] (3,0) circle (3pt);
			\fill[white] (2,1) circle (3pt);
			\fill[white] (1,2) circle (3pt);
			\fill[white] (0,3) circle (3pt);
			\fill[white] (2,0) circle (3pt);
			\fill[white] (1,1) circle (3pt);
			\fill[white] (0,2) circle (3pt);
			%fill cyan
			\fill[cyan!50] (1.4,1.2)--(1.4,2)--(-0.4,3.8)--(-0.4,3)--cycle;
			%surround blue area
			\draw[blue!50!black, very thick] (1.4,1.2)--(1.4,2)--(0.6,2.8)--(0.6,2)--cycle;
			%labels
			\node at (3.1,0) {$0$};
			\node at (2.1,1) {$0$};
			\node at (2,0) {$0$};
			\node at (1,1) {$0$};
			\node at (0,2) {$0$};
			\node at (1.1,2) {$\beta_1$};
			\node at (0.1,3) {$\beta_0$};
			%draw arrows 
			\draw[->] (2.8,0)--(2.2,0) node[midway, below] {$\bar{\delta}$};
			\draw[->] (1.8,1)--(1.2,1) node[midway, below] {$\bar{\delta}$};
			\draw[->] (0.8,2)--(0.2,2) node[midway, below] {$\bar{\delta}$};
			\draw[->] (0,2.8)--(0,2.2) node[midway, right] {$d$};
			\draw[->] (1,1.8)--(1,1.2) node[midway, right] {$d$};
			\draw[->] (2,0.8)--(2,0.2) node[midway, right] {$d$};
			%now we include labels
			\node[blue!50!black] at (1.5,2.5) {$\GK_{1,2}$};
		\end{scope}
		\begin{scope}
			%draw lattice
			\foreach \x in {0,1,...,3}{
				\foreach \y in {0,1,...,3}{
        				\fill (\x,\y) circle (2pt);
        			}
			}
			%white points
			\fill[white] (3,0) circle (3pt);
			\fill[white] (2,1) circle (3pt);
			\fill[white] (1,2) circle (3pt);
			\fill[white] (0,3) circle (3pt);
			\fill[white] (2,0) circle (3pt);
			\fill[white] (1,1) circle (3pt);
			\fill[white] (0,2) circle (3pt);
			%fill cyan
			\fill[cyan!50] (2.4,0.2)--(2.4,1)--(-0.4,3.8)--(-0.4,3)--cycle;
			%surround blue area
			\draw[blue!50!black, very thick] (2.4,0.2)--(2.4,1)--(1.6,1.8)--(1.6,1)--cycle;
			%labels
			\node at (3.1,0) {$0$};
			\node at (2,0) {$0$};
			\node at (1,1) {$0$};
			\node at (0,2) {$0$};
			\node at (2.1,1) {$\beta_2$};
			\node at (1.1,2) {$\beta_1$};
			\node at (0.1,3) {$\beta_0$};
			%draw arrows 
			\draw[->] (2.8,0)--(2.2,0) node[midway, below] {$\bar{\delta}$};
			\draw[->] (1.8,1)--(1.2,1) node[midway, below] {$\bar{\delta}$};
			\draw[->] (0.8,2)--(0.2,2) node[midway, below] {$\bar{\delta}$};
			\draw[->] (0,2.8)--(0,2.2) node[midway, right] {$d$};
			\draw[->] (1,1.8)--(1,1.2) node[midway, right] {$d$};
			\draw[->] (2,0.8)--(2,0.2) node[midway, right] {$d$};
			%now we include labels
			\node[blue!50!black] at (2.5,1.5) {$\GK_{2,1}$};
		\end{scope}
		\begin{scope}[xshift=5cm]
			%draw lattice
			\foreach \x in {0,1,...,3}{
				\foreach \y in {0,1,...,3}{
        				\fill (\x,\y) circle (2pt);
        			}
			}
			%white points
			\fill[white] (3,0) circle (3pt);
			\fill[white] (2,1) circle (3pt);
			\fill[white] (1,2) circle (3pt);
			\fill[white] (0,3) circle (3pt);
			\fill[white] (2,0) circle (3pt);
			\fill[white] (1,1) circle (3pt);
			\fill[white] (0,2) circle (3pt);
			%fill cyan
			\fill[cyan!50] (3,-0.4)--(3.8,-0.4)--(-0.4,3.8)--(-0.4,3)--cycle;
			%surround blue area
			\draw[blue!50!black, very thick] (3,-0.4)--(3.8,-0.4)--(2.6,0.8)--(2.6,0)--cycle;
			%labels
			\node at (2,0) {$0$};
			\node at (1,1) {$0$};
			\node at (0,2) {$0$};
			\node at (3.1,0) {$\beta_3$};
			\node at (2.1,1) {$\beta_2$};
			\node at (1.1,2) {$\beta_1$};
			\node at (0.1,3) {$\beta_0$};		
			%draw arrows 
			\draw[->] (2.8,0)--(2.2,0) node[midway, below] {$\bar{\delta}$};
			\draw[->] (1.8,1)--(1.2,1) node[midway, below] {$\bar{\delta}$};
			\draw[->] (0.8,2)--(0.2,2) node[midway, below] {$\bar{\delta}$};
			\draw[->] (0,2.8)--(0,2.2) node[midway, right] {$d$};
			\draw[->] (1,1.8)--(1,1.2) node[midway, right] {$d$};
			\draw[->] (2,0.8)--(2,0.2) node[midway, right] {$d$};
			%now we include labels
			\node[blue!50!black] at (3.5,0.5) {$\GK_{3,0}$};
		\end{scope}
		\begin{scope}[xshift=9cm,yshift=6cm]
			%fill rectangles
			\fill[orange!50] (0.8,0.8) rectangle (2.2,1.2);
			\fill[orange!50] (1.8,-0.2) rectangle (2.2,1.2);
			%surround orange area
			\draw[orange!50!black, very thick] (1.8,0.8) rectangle (2.2,1.2);
			%labels
			\node at (2,0) {$\alpha_2$};
			\node at (1,1) {$\alpha_1$};
			\node at (2,1) {$\beta_2$};
			%draw arrows 
			\draw[->] (2,0.8) -- (2,0.2);
			\draw[->] (1.8 ,1) -- (1.2 ,1);
			%labels
			\node[orange!50!black] at (1.6,2) {$\GZ_{2,1}^0$};
			\node[orange!50!black] at (1.6,1.5) {$= \cS_{2,1}$};
		\end{scope}
		\begin{scope}[xshift=11cm,yshift=6cm]
			%fill rectangles
			\fill[orange!50] (0.8,0.8) rectangle (2.2,1.2);
			%surround orange area
			\draw[orange!50!black, very thick] (1.8,0.8) rectangle (2.2,1.2);
			%labels
			\node at (2,0) {$0$};
			\node at (1,1) {$\alpha_1$};
			\node at (2,1) {$\beta_2$};
			%draw arrows 
			\draw[->] (2,0.8) -- (2,0.2);
			\draw[->] (1.8 ,1) -- (1.2 ,1);
			%labels
			\node[orange!50!black] at (1.4,2) {$\GZ_{2,1}^1$};
			\node[orange!50!black] at (1.8,1.5) {$= \Ker(d)_{2,1}$};
		\end{scope}
		\begin{scope}[xshift=10cm,yshift=3cm]
			%fill rectangles
			\fill[orange!50] (-0.4,1.8)  rectangle (1.2,2.2);
			\fill[orange!50] (2.4,0.2)--(2.4,1)--(1.2,2.2)--(0.6,2)--cycle;
			%surround blue area
			\draw[orange!50!black, very thick] (2.4,0.2)--(2.4,1)--(1.6,1.8)--(1.6,1)--cycle;
			%labels
			\node at (2,0) {$0$};
			\node at (1,1) {$0$};
			\node at (2.05,1) {$\beta_2$};
			\node at (1.05,2) {$\beta_1$};
			\node at (0,2) {$\alpha_0$};
			%draw arrows 
			\draw[->] (2,0.8) -- (2,0.2);
			\draw[->] (1,1.8) -- (1,1.2);
			\draw[->] (1.8 ,1) -- (1.2 ,1);
			\draw[->] (0.8 ,2) -- (0.2 ,2);
			%labels
			\node[orange!50!black] at (2.4,1.8) {$\GZ_{2,1}^2$};
		\end{scope}
		\begin{scope}[xshift=10cm]
			%fill rectangles
			\fill[orange!50] (2.4,0.2)--(2.4,1)--(-0.4,3.8)--(-0.4,3)--cycle;
			%surround blue area
			\draw[orange!50!black, very thick] (2.4,0.2)--(2.4,1)--(1.6,1.8)--(1.6,1)--cycle;
			%labels
			\node at (2,0) {$0$};
			\node at (1,1) {$0$};
			\node at (2.05,1) {$\beta_2$};
			\node at (1.05,2) {$\beta_1$};
			\node at (0,2) {$0$};
			\node at (0,3) {$\beta_0$};
			%draw arrows 
			\draw[->] (2,0.8) -- (2,0.2);
			\draw[->] (1,1.8) -- (1,1.2);
			\draw[->] (0,2.8) -- (0,2.2);
			\draw[->] (1.8 ,1) -- (1.2 ,1);
			\draw[->] (0.8 ,2) -- (0.2 ,2);
			%labels
			\node[orange!50!black] at (2.2,2.3) {$\GZ_{2,1}^3$};
			\node[orange!50!black] at (2.5,1.8) {$ = \GK_{2,1}$};
		\end{scope}
		\end{tikzpicture}
	\end{center}
	\caption{On the left, in cyan the four direct summands of $\Ker(d^\Tot)_4$. The corresponding  $\GK_{r,3-r}$ are framed to emphasize that
	they are respective subspaces of $\cS_{r,3-r}$ for all $0 \leq r \leq 3$. On the right, in orange the subspaces $\GZ^r_{2,1}$, eventually shrinking to $\GK_{2,1}$. For convenience, 
	we have labelled $\alpha_2 = d(\beta_2)$, $\alpha_1 = \bar{\delta}(\beta_2)$ and $\alpha_0 = \bar{\delta}(\beta_1)$.}
	\label{diag:totalKernel}
\end{figure}
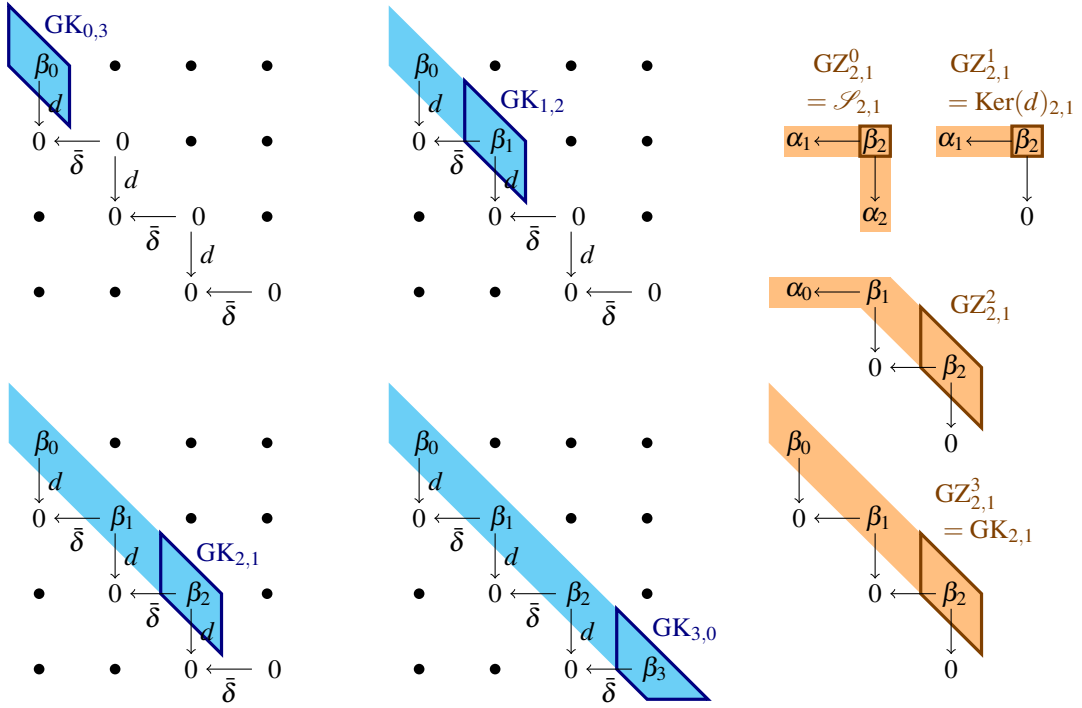

Hence, recovering the sets $\GK_{p,q}$ leads to the kernel of $d^\Tot_n$. The problem with this approach is that each subspace $\GK_{p,q}$ still requires a large
set of equations to be checked. A step-by-step way of computing these is by adding one equation at 
a time. For this we define the subspaces $\GZ^r_{p,q} \subseteq \cS_{p,q}$ where we add the first $r$ equations progressively. That is, we start setting $\GZ^0_{p,q} = \cS_{p,q}$. 
Then we define $\GZ^1_{p,q}$ to be elements $s_{p,q} \in \cS_{p,q}$ such that $d(s_{p,q})=0$, or equivalently $\GZ^1_{p,q} = \Ker(d)_{p,q}$. 
In an inductive way, for $r \geq 2$ we define $\GZ^r_{p,q}$ to be formed by elements $s_{p,q} \in \GZ^{r-1}_{p,q}$ such that there exists a sequence
$s_{p-k,q+k} \in \cS_{p-k,q+k}$ 
satisfying equations $d(s_{p-k, q+k}) = -\bar{\delta}(s_{p-k+1, q+k-1})$ for all $1 \leq k < r$. 
Then, for all $p,q \geq 0$, we have a decreasing sequence
$$
\GK_{p,q} = \GZ^{p+1}_{p,q} \subseteq  \GZ^{p}_{p,q} \subseteq \cdots \subseteq \GZ^0_{p,q} = \cS_{p,q}.
$$
For intuition see Figure~\ref{diag:totalKernel}, and also Figure~\ref{diag:kernelsFiltrationZ} for a depiction of $\GZ^2_{3,1}$ on a lattice.
A very compact way of expressing that is by the definition $\GZ^r_{p,q} = \Ker(d) \cap (\bar{\delta}^{-1} \circ d)^{r-1} (\cS_{p-r+1,q+r-1})$ for all $r \geq 1$, where by $(\bar{\delta}^{-1} \circ d)^r$ we 
mean composing $r$-times the preimage $\bar{\delta}^{-1} \circ d$. In particular, 
since $\GZ^r_{p,q} = \GZ^{p+1}_{p,q}$ for all $r \geq p+1$,  we sometimes use the convention $\GZ^\infty_{p,q} \coloneqq \GZ^{p+1}_{p,q} = \GK_{p,q}$.

Now we explain the notation $\GK_{p,q}$ and the isomorphism~(\ref{eq:non_canonical_kernel}). We start defining a \emph{vertical filtration} $F^*_V$ on $\cS_{*,*}$ by the following subcomplexes for all $r \geq 0$:
$$
F_V^r\left(\cS_{*,*}\right)_{p,q} \coloneqq 
\begin{cases}
\cS_{p,q}, {\rm \ whenever \ } p \leq r, \\
0, {\rm \ otherwise.}
\end{cases}
$$
Notice that this filtration increases with the index, so that we have inclusions
$F^r_V(\cS_{*,*}) \subseteq F^{r+1}_V(\cS_{*,*})$ for all $r \geq 0$.
Additionally, we obtain isomorphisms  $F_V^p(\cS_{*,*})/ F_V^{p-1}(\cS_{*,*}) \cong \cS_{p,*}$
for all $p \geq 0$.
The filtration $F_V^*$ respects the morphisms in $\cS_{*,*}$ in the sense that $d(F^t_V(\cS_{*,*})) \subset F^t_V(\cS_{*,*})$, and also
$\bar{\delta}(F^t_V(\cS_{*,*})) \subset F^t_V(\cS_{*,*})$. See Figure~\ref{diag:secondVertFiltrationQuotient} for a depiction of $F_V^*$.
Another point to notice is that $F^*_V$ will filter the total complex $\cS_*^\Tot$, respecting its 
differential $d^\Tot$. That is, $\cS^\Tot_n$ will be filtered by subcomplexes,
$$
F^r_V \cS^\Tot_n \coloneqq \bigoplus \limits_{\substack{p + q = n \\ p \leq r}} \cS_{p,q},
$$
for all $r \geq 0$. 
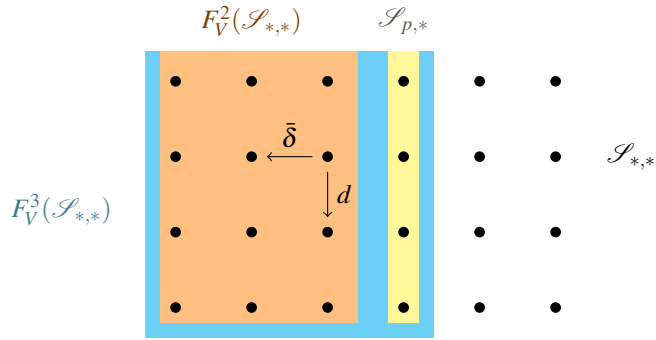
\begin{figure}
	\begin{center}
		\begin{tikzpicture}
		\fill[cyan!50] (-0.4,-0.4)--(3.4,-0.4)--(3.4,3.4)--(-0.4,3.4)--cycle;
		\fill[orange!50] (-0.2,-0.2)--(2.4,-0.2)--(2.4,3.4)--(-0.2,3.4)--cycle;
		\fill[yellow!50] (2.8,-0.2)--(3.2,-0.2)--(3.2,3.4)--(2.8,3.4)--cycle;
		\foreach \x in {0,1,...,5}{
			\foreach \y in {0,1,...,3}{
        			\fill (\x,\y) circle (2pt);
        		}
		}
		\draw[cyan!50!black] (-1.5,1.3) node {$F_V^3(\cS_{*,*})$};
		\draw[orange!50!black] (1,3.8) node {$F_V^2(\cS_{*,*})$};
		\draw[yellow!20!black] (3,3.8) node {$\cS_{p,*}$};
		\draw (6,2) node {$\cS_{*,*}$};
		\draw[->] (2,1.8)--(2,1.2) node[midway, right] {$d$};
		\draw[->] (1.8,2)--(1.2,2) node[midway, above] {$\bar{\delta}$};
		\end{tikzpicture}
	\end{center}
	\caption{ Note that $F_V^3(\cS_{*,*})/ F_V^2(\cS_{*,*}) \cong \cS_{3,*}$. 
	Also notice that the differentials $\bar{\delta}$ and $d$ respect the vertical filtration $F^*_V$.}
	\label{diag:secondVertFiltrationQuotient}
\end{figure}
In particular, notice that $\Ker(d^\Tot)$ also inherits the filtration $F_V^*$, where we will have filtration sets $F_V^p \Ker(d^\Tot)_n=F_V^p \cS^\Tot_n \cap \Ker(d^\Tot)_n$.
We define the \emph{associated modules} of $\Ker(d^\Tot)_n$ to be the quotients $G_V^p \Ker(d^\Tot)_n = F_V^p \Ker(d^\Tot)_n / F_V^{p-1} \Ker(d^\Tot)_n$, 
which can be checked to be isomorphic with $\GK_{p,q}$ for all $p+q = n$. This follows by considering morphisms
\begin{equation}\label{eq:morphGradedKernel}
\xymatrix@R=0.2cm{
 G_V^p \Ker(d^\Tot)_n \ar[r] &  \GK_{p,q},\\
 [(s_{0, n}, s_{1,n-1}, \ldots, s_{p,q}, 0, \ldots, 0)] \ar[r] & s_{p,q},
}
\end{equation}
which are well-defined since $s_{p,q}$ does not change for representatives of the same class. 
In fact, this morphism is injective since two classes with the same image will be equal by definition of $G_V^p \Ker(d^\Tot)_n$. 
On the other hand, the definition of $\GK_{p,q}$ ensures surjectivity. In particular, since we are working with vector spaces, 
we have that:
$$
\Ker(d^\Tot_n) \cong \bigoplus \limits_{p+q=n} G_V^p \Ker(d^\Tot)_n \cong \bigoplus \limits_{p+q=n} \GK_{p,q}.
$$
which justifies isomorphism~(\ref{eq:non_canonical_kernel}).

Next, we explain the notation $\GZ^r_{p,q}$. We introduce the objects
$$
Z^r_{p,q} \coloneqq \left\{
z \in  F^p_V\cS^\Tot_{p+q}\; :\; d^\Tot(z) \in F^{p-r}_V\cS_{p+q-1}^\Tot 
\right\} 
$$
for all $r \geq 0$. We can think of these as kernels of $d^\Tot$ up to some previous filtration. Then,
by definition, we have that $Z^0_{p,q} = F^p_V\cS^\Tot_{p+q}$ and 
$Z^{p+1}_{p,q} = Z^\infty_{p,q}= F^p_V \Ker(d^\Tot_{p+q})$.
Using a morphism analogous to~(\ref{eq:morphGradedKernel}), one can check that 
 the quotients $Z^{r+1}_{p,q}/Z^r_{p-1,q+1}$ are isomorphic to  $\GZ^{r+1}_{p,q}$ for all $p + q = n$. 
This is depicted in Figure~\ref{diag:kernelsFiltrationZ}.
Thus, computing these quotients increasing $r\geq 0$ leads
to the desired kernel $\Ker(d^\Tot)$. With a little more work, we can do the same for computing the homology.
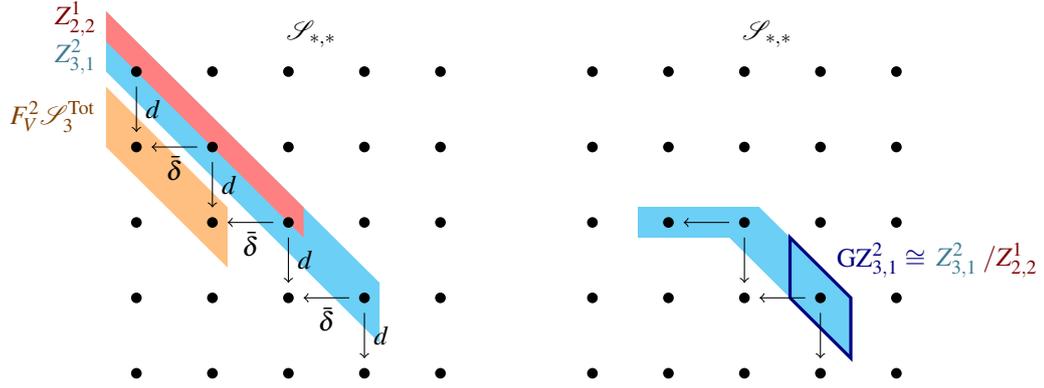
\begin{figure}
	\begin{center}
		\begin{tikzpicture}
		%colour diagonals
		\fill[cyan!50] (3.2,0.4)--(3.2,1.2)--(-0.4,4.8)--(-0.4,4)--cycle;
		\fill[red!50] (2.2,1.8)--(2.2,2.2)--(-0.4,4.8)--(-0.4,4.4)--cycle;
		\fill[orange!50] (1.2,1.4)--(1.2,2.2)--(-0.4,3.8)--(-0.4,3)--cycle;
		%draw lattice
		\foreach \x in {0,1,...,4}{
			\foreach \y in {0,1,...,4}{
        			\fill (\x,\y) circle (2pt);
        		}
		}
		%draw arrows 
		\draw[->] (2.8,1)--(2.2,1) node[midway, below] {$\bar{\delta}$};
		\draw[->] (1.8,2)--(1.2,2) node[midway, below] {$\bar{\delta}$};
		\draw[->] (0.8,3)--(0.2,3) node[midway, below] {$\bar{\delta}$};
		\draw[->] (0,3.8)--(0,3.2) node[midway, right] {$d$};
		\draw[->] (1,2.8)--(1,2.2) node[midway, right] {$d$};
		\draw[->] (2,1.8)--(2,1.2) node[midway, right] {$d$};
		\draw[->] (3,0.8)--(3,0.2) node[midway, right] {$d$};
		%now we include labels
		\draw[orange!50!black] (-1.1,3.4) node {$F^2_V\cS^\Tot_3$};
		\draw[cyan!50!black] (-0.8,4.2) node {$Z_{3,1}^2$};
		\draw[red!50!black] (-0.8,4.7) node {$Z_{2,2}^1$};
		\draw (2.3,4.5) node {$\cS_{*,*}$};
		\begin{scope}[xshift=7cm]
			%fill rectangles
			\fill[cyan!50] (-0.4,1.8)  rectangle (1.2,2.2);
			\fill[cyan!50] (2.4,0.2)--(2.4,1)--(1.2,2.2)--(0.6,2)--cycle;
			%surround blue area
			\draw[blue!50!black, very thick] (2.4,0.2)--(2.4,1)--(1.6,1.8)--(1.6,1)--cycle;
			\foreach \x in {0,1,...,4}{
				\foreach \y in {0,1,...,4}{
        				\fill (\x-1,\y) circle (2pt);
        			}
			}
			%draw arrows 
			\draw[->] (2,0.8) -- (2,0.2);
			\draw[->] (1,1.8) -- (1,1.2);
			\draw[->] (1.8 ,1) -- (1.2 ,1);
			\draw[->] (0.8 ,2) -- (0.2 ,2);
			%labels
			\node[blue!50!black] at (2.8,1.5) {$\GZ_{3,1}^2 \cong$};
			\node[cyan!50!black] at (3.8,1.5) {$Z^2_{3,1}$};
			\node[red!50!black] at (4.5,1.5) {$/ Z^1_{2,2}$};
			\draw (1.3,4.5) node {$\cS_{*,*}$};
		\end{scope}
		\end{tikzpicture}
	\end{center}
	\caption{ On the left the sets $Z^2_{3,1}$ and $Z^1_{2,2}$. On the right their respective quotient $\GZ^2_{3,1}$.}
	\label{diag:kernelsFiltrationZ}
\end{figure}

There is a procedure commonly known as a \emph{spectral sequence} which leads to $\Ho_n(\cS^\Tot_*)$ after
a series of small, computable steps. This is done in an analogous way as we did before for computing $\Ker(d^\Tot)$. 
In this case we will need to take the extra steps of taking quotients by the images of $d^\Tot$. 
First notice that the vertical filtration $F_V^*$ transfers to homology $\Ho_n(\cS^\Tot_*)$ by 
the inclusions $F^p_V\cS^\Tot_* \subseteq \cS^\Tot_*$ for all $p \geq 0$. That is, we have filtered sets:
$$
F_V^p \Ho_n(\cS^\Tot_*) \coloneqq \Img \big(\, \Ho_n(F^p_V\cS^\Tot_*) \longrightarrow \Ho_n(\cS^\Tot_*)\,\big)
$$
which induce a filtration on $\Ho_n(\cS^\Tot_*)$. For this filtration the associated modules 
 will be defined by the quotients $G_V^r \Ho_n(\cS^\Tot_*) = F_V^r \Ho_n(\cS^\Tot_*) / F_V^{r-1} \Ho_n(\cS^\Tot_*)$
 for all $r \geq 0$. Notice that in this case, since we are assuming that we are working over a field, there will be no 
extension problems and we will recover the homology by taking direct sums:
$$
\Ho_n(\cS^\Tot_*) \cong \bigoplus \limits_{r = 0}^n G_V^r \Ho_n(\cS^\Tot_*). 
$$
In Section~\ref{sec:persistence-mayer-vietoris}, we will be dealing with the situation where this is not true. 
Previously, we defined the sets $Z^r_{p,q}$ which are kernels up to filtration. In an analogous way we define
boundaries up to filtration by setting
\begin{align*}
B^r_{p,q} & \coloneqq \Big\{\,
d^\Tot(c) \in F^p_V\cS^\Tot_{p+q} \;:\; c \in F^{p+r}_V\cS_{p+q+1}^\Tot \,\Big\}
\end{align*}
for all $r \geq 0$, and $p,q \geq 0$. These are images of $d^\Tot$ coming from a previous filtration.
Notice that we will have relations $d^\Tot (Z^r_{p,q}) = B^r_{p-r,q+r-1}$ and also $d^\Tot(B^r_{p,q}) = 0$.
Additionally there is a sequence
of inclusions,
$$
B_{p,q}^0 \subset B_{p,q}^1 \subset \cdots \subset B_{p,q}^{q+1} = B_{p,q}^\infty \subset
Z_{p,q}^\infty = Z_{p,q}^p \subset \cdots Z_{p,q}^1 \subset Z_{p,q}^0,
$$
for all $p,q \geq 0$. 

From the previous discussion, we start defining the \emph{first page} of the spectral sequence as the quotient
$$
E^1_{p,q} \coloneqq \dfrac{Z^1_{p,q}}{Z^0_{p-1,q+1} + B^0_{p,q}} \cong \dfrac{ \GZ^1_{p,q}}{\Img\left(B^0_{p,q}\rightarrow \GZ^1_{p,q}\right)},
$$
for all $p, q \geq 0$. Recall that $\Ker(d)_{p,q} = \GZ^1_{p,q} = Z^1_{p,q} / Z^0_{p-1,q+1}$ and also one can see that
$\Img\left(B^0_{p,q}\rightarrow \GZ^1_{p,q}\right)$ is isomorphic to $\Img(d)_{p,q}$. Then 
we deduce that $E^1_{p,q} \cong \Ho_q(\cS_{p,*}, d)$. On this page $d^\Tot$ induces differentials 
$d^1: E^1_{p,q} \rightarrow E^1_{p-1,q}$. That is, noticing that $d^\Tot (Z_{p,q}^1) = B_{p-1,q}^1 \subset Z_{p-1,q}^1$
and also $d^\Tot ( Z^0_{p-1,q+1} + B^0_{p,q}) = d^\Tot ( Z^0_{p-1,q+1}) + 0 = B^0_{p-1,q}$ we will have that
$d^1: E^1_{p,q} \rightarrow E^1_{p-1,q}$ is well-defined. Notice that since $d^\Tot \circ d^\Tot = 0$ we will also have $d^1 \circ d^1 = 0$ and 
in particular one can define the homology on the first page $\Ho_{p,q}(E^1_{*,*}, d^1)$. Since 
$$
\Ker(d^1) = \dfrac{Z^2_{p,q}}{Z^2_{p,q} \cap (Z^0_{p-1,q+1} + B^0_{p,q})} = \dfrac{Z^2_{p,q}}{Z^1_{p-1,q+1} + B^0_{p,q}}, {\rm \ and \ } 
\Img(d^1) = \dfrac{B^1_{p,q}}{B^0_{p,q}}
$$
then the second page will be
$$
E^2_{p,q} \coloneqq \Ho_{p,q}(E^1_{*,*}, d^1) = \dfrac{\Ker(d^1)}{\Img(d^1)} =  \dfrac{Z^2_{p,q}}{Z^1_{p-1,q+1} + B^1_{p,q}}.
$$

The second page has differential $d^2$ induced by the total complex differential $d^\Tot$. 
Figure~\ref{diag:secondPage} illustrates this principle. 
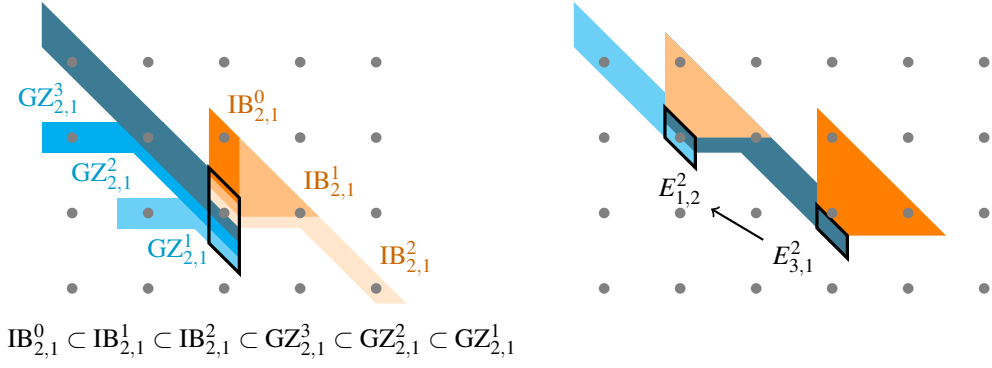
\begin{figure}
	\begin{center}
		\begin{tikzpicture}
		\begin{scope}
			%fill regions
			\fill[cyan!50] (2.2,0.2)--(2.2,1.2)--(2,1.4)--(1.6,0.8)--cycle;
			\fill[cyan!50] (0.6,0.8)  rectangle (2,1.2);
			\fill[cyan] (-0.4,1.8)  rectangle (1.2,2.2);
			\fill[cyan] (2.2,0.4)--(2.2,1.2)--(1.2,2.2)--(0.6,2)--cycle;
			\fill[cyan!50!black] (2.2,0.6)--(2.2,1.2)--(-0.4,3.8)--(-0.4,3.2)--cycle;
			\fill[orange!20] (4,-0.2)--(4.4,-0.2)--(1.8,2.4)--(1.8,1.2)--(2.2,0.8)--(3,0.8)--cycle;
			\fill[orange!50] (1.8,2.4)--(1.8, 1.35)--(2.2,0.95)--(3.25,0.95)--cycle;
			\fill[orange] (1.8,2.4)--(1.8,1.5)--(2.2,1.1)--(2.2,2)--cycle;
			%draw lines 
			\draw[very thick] (1.8,1.6)--(1.8,0.6)--(2.2,0.2)--(2.2,1.2)--cycle;
			%draw lattice
			\foreach \x in {0,1,...,4}{
				\foreach \y in {0,1,...,3}{
        				\fill[gray] (\x,\y) circle (2pt);
        			}
			}
			\node[orange!80!black] at (4.4,0.4) {$\IB^2_{2,1}$};
			\node[orange!80!black] at (3.4,1.4) {$\IB^1_{2,1}$};
			\node[orange!80!black] at (2.4,2.4) {$\IB^0_{2,1}$};
			\node[cyan!80!black] at (1.4,0.5) {$\GZ^1_{2,1}$};
			\node[cyan!80!black] at (0.4,1.5) {$\GZ^2_{2,1}$};
			\node[cyan!80!black] at (-0.3,2.5) {$\GZ^3_{2,1}$};
			\node at (2.5,-0.7) {$\IB^0_{2,1} \subset \IB^1_{2,1} \subset \IB^2_{2,1} \subset \GZ^3_{2,1} \subset \GZ^2_{2,1} \subset \GZ^1_{2,1}$};
		\end{scope}
		\begin{scope}[xshift=7cm]
			%fill in regions 
			\fill[cyan!50] (1.2,1.6)--(1.2,2.2)--(-0.4,3.8)--(-0.4,3.2)--cycle;
			\fill[cyan!50!black] (0.8,2.2)--(1.2,1.8)--(1.8,1.8)--(3.2,0.4)--(3.2,1)--(0.8,3.4)--cycle;
			\fill[orange!50] (0.8,2.4)--(1.2,2)--(2.2,2)--(0.8,3.4)--cycle;
			\fill[orange] (2.8,1.1)--(3.2,0.7)--(4.5,0.7)--(2.8,2.4)--cycle;
			%frames
			\draw[very thick] (1.2,1.6) -- (1.2,2) -- (0.8, 2.4)--(0.8,2)--cycle;
			\draw[very thick] (3.2, 0.4)--(3.2, 0.7)--(2.8, 1.1)--(2.8, 0.8)--cycle;
			%draw lattice
			\foreach \x in {0,1,...,5}{
				\foreach \y in {0,1,...,3}{
        				\fill[gray] (\x,\y) circle (2pt);
        			}
			}
			\node (A) at (1,1.3) {$E^2_{1,2}$};
			\node (B) at (2.5,0.4) {$E^2_{3,1}$};
			\draw[->, thick] (B) -- (A);
		\end{scope}
		\end{tikzpicture}
	\end{center}
	\caption{On the left, the different subspaces on $\cS_{2,1}$. 
	Here  $\IB_{2,1}^r = \Img\left(B^{r}_{2,1}\rightarrow \GZ^{r+1}_{2,1}\right)$, for all $0 \leq r \leq 2$. The framed 
	region represents $\cS_{2,1}$. Brighter colours represent bigger regions than darker colours.  Note that blue and orange colours
	have been assigned to $\GZ^*_{2,1}$ and $\IB^*_{2,1}$ respectively.
	On the right, the morphism $d^2: E_{3,1}^2 \rightarrow E_{1,2}^2$ on the second page. The two framed regions
	represent the codomain and domain of $d^2$, these have been assigned brighter and darker colours, respectively.}
	\label{diag:secondPage}
\end{figure}

Doing the same for all pages we obtain the definition of the $r$-page:
$$
E^r_{p,q}  \coloneqq \Ho_{p,q}(E^{r-1}_{*,*}, d^{r-1}) = \dfrac{Z^{r}_{p,q}}{Z^{r-1}_{p-1,q+1} + B^{r-1}_{p,q}}
$$
for all $r \geq 2$. Of course, we can express alternatively the $r$-page terms as:
$$
E^r_{p,q} \coloneqq \dfrac{ \GZ^r_{p,q}}{\Img\left(B^{r-1}_{p,q}\rightarrow \GZ^r_{p,q}\right)}.
$$
Thus, the $\infty$-page is:
$$
E^\infty_{p,q} = \dfrac{Z_{p,q}^\infty}{Z_{p-1,q+1}^\infty + B_{p,q}^\infty} \cong \dfrac{ \GK_{p,q}}{\Img\left(B^{\infty}_{p,q}\rightarrow \GK_{p,q}\right)}.
$$
Then, for $n=p+q$ one has the equality 
$$
G_V^p \Ho_n(\cS^\Tot_*) = \dfrac{F_V^p \Ho_n(\cS^\Tot_*)}{F_V^{p-1} \Ho_n(\cS^\Tot_*)}
	= \dfrac{ \Img \left( \Ho_n(F^p_V\cS^\Tot_*) \longrightarrow \Ho_n(\cS^\Tot_*)\right)}{ \Img \left( \Ho_n(F^{p-1}_V\cS^\Tot_*) \longrightarrow \Ho_n(\cS^\Tot_*)\right)}
	\cong \dfrac{ Z^\infty_{p,q} / B^\infty_{p,q}}{Z^\infty_{p-1,q+1} / B^\infty_{p-1,q+1}}
	\cong E^\infty_{p,q}
$$
since $B_{p-1,q+1}^\infty \subseteq B_{p,q}^\infty$.
Therefore, computing the spectral sequence is a way of approximating the associated module $G_V^p \Ho_n(\cS^\Tot_*)$.
Thus adding up all of these leads to the result $\Ho_n(\cS^\Tot_*)$. By convention, since 
$E^\infty_{p,q} \cong G^p_V\Ho_n(\cS^\Tot_*)$ we say that $E^*_{p,q}$ \emph{converges} to $\Ho_n(\cS^\Tot)$
and we denote this as
$$
E^*_{p,q} \Rightarrow \Ho_n(\cS^\Tot).
$$

\begin{remark}
Here we have adopted the definition of $Z^r_{p,q}$ and $B^r_{p,q}$ that one can find in \citep{McCleary2001}. Other sources such as 
\citep{BoTu1982} and \citep{Lipsky2011} use the same notation for other terms. 
\end{remark}

So far, we have studied spectral sequences for vertical filtrations. 
Similarly, there is a horizontal filtration,
$$
F^r_H \cS^\Tot_n \coloneqq \bigoplus{\substack{p + q = n \\ q \leq r}} \cS_{p,q},
$$ 
for all $r \geq 0$. We can apply the same argument to this filtration, to 
obtain a spectral sequence 
$$
{_H E}_{*,*}^* \Rightarrow \Ho_n(\cS^\Tot).
$$ 
 An intuitive way of thinking of this is by applying a symmetry about the diagonal $x=y$
on the previous discussion. Thus the first page is computed with the homology with respect to horizontal differentials, the second with respect to 
vertical differentials, and so on. This leads easily to the following widely known result:
\begin{proposition}\label{prop:convSS}
$\Ho_n(\cS^\Tot_*) \cong \Ho_n(K)$.
\end{proposition}
\begin{proof}
In order to turn to the 
 first page, we need to compute homology with respect to the horizontal
differentials $\delta$. As shown in the preliminaries, the \cech chain 
complexes are exact,  so that:
$$
{_HE}_{*,*}^1 \coloneqq H_{0,q}({_HE}_{*,*}^0, \bar{\delta}) =
\begin{cases}
	S_q(K) & {\rm \ if \ } p = 0 { \ \rm and \ } q \geq 0 \\
	0 & {\rm \ otherwise. }
\end{cases}
$$
After this one can compute the second page by the homology with respect 
to vertical differentials $d$ induced on the first page,
$$
{_HE}_{*,*}^2 \coloneqq H_{0,q}({_HE}_{*,*}^1, d) =
\begin{cases}
	H_q(K) & {\rm \ if \ } p = 0 { \ \rm and \ } q \geq 0 \\
	0 & {\rm \ otherwise.}
\end{cases}
$$
To proceed to the next page, we would need to consider homology with respect
to diagonal differentials, 
$$
d_2 : {_HE}_{p,q}^2 \longrightarrow  {_HE}_{p+1,q-2}^2.
$$
Since the second page $E_{p,q}^2$ has only one non-zero column $p=0$, 
computing homology with respect to $d_2$ leaves this page intact. 
The same happens when 
we consider for any $r > 2$ homology with respect to differentials 
$$
d_r : {_HE}_{p,q}^r \longrightarrow {_HE}_{p+r-1,q-r}^r.
$$
Thus, we say that ${_HE}_{p,q}^*$ has \emph{collapsed} on the second page, which is 
usually denoted as ${_HE}_{p,q}^2 = {_HE}_{p,q}^\infty$. Each diagonal has a unique nonzero
entry ${_HE}^\infty_{0,q} \cong H^q(K)$. In particular, we have isomorphisms
$$
{_HE}^\infty_{0,n} \cong \Ho_n(\cS^\Tot_*) \cong \Ho_n(K),
$$
for all $n \geq 0$. 
\end{proof}

Therefore, using proposition~\ref{prop:convSS}, we have
that the spectral sequence converges to the wanted result
$$
E^*_{p,q} \Rightarrow \Ho_n(K). 
$$
In particular, since we are in the category of
vector spaces, there are no extension problems. Thus,
we have an isomorphism
$$
\Ho_n(K) \cong \bigoplus \limits_{p+q = n} E^\infty_{p,q}.
$$
Throughout the following section, we will adapt this setting to 
the category of persistence modules.

\section{Persistent Mayer-Vietoris}
\label{sec:persistence-mayer-vietoris}

One can translate the method from section~\ref{sec:MayerVietorisSS}
to \PMod. The reason for this is that \PMod~is an \emph{abelian category}, 
since \Vect~is an abelian category and \bR~is a small category. The theory
of spectral sequences can be developed for arbitrary abelian categories. 
For an introduction to this, see chapter~5 in~\citep{Weibel1994}.

Suppose that we have covered a filtered simplicial complex $K$ with filtered 
subcomplexes $\cU = \{ U_i\}_{i \in I}$, so that $K = \bigcup_{i \in I} U_i$.
Then, we can compute 
the spectral sequence 
$$
	E^1_{p,q} = \bigoplus_{\sigma \in \Delta^m_p} \PH_q(U_\sigma) \Rightarrow \PH_n(K),
$$
where $p + q = n$. However, unlike the case of vector spaces, 
we might have that
$$
\bigoplus \limits_{p+q=n} E^\infty_{p,q} \ncong \PH_n(K).
$$
All that we know is that
$E^\infty_{p,q} \cong G^p \PH_{p+q}(K)$ for all $p,q \geq 0$. 
This  is the \emph{extension problem}, which we will solve  in Section~\ref{sub:extension-problem}.
After solving this problem we will obtain the persistent  homology for $K$. We will even recover more information.
Notice that 
as pointed out in \citep{Yoon2018}, the knowledge of which subset $J \subset I$ detects a 
feature from $\PH_n(K)$ can potentially add insight into the information given by 
ordinary persistent homology. The following example illustrates this.  

\begin{example}
Consider the case of a point
cloud $X$ covered by two open sets as in Figure~\ref{fig:extensionCircleSplit2}.
From Sections \ref{sec:persistence-modules} and \ref{sec:MayerVietorisSS}, we know how to compute the $\infty$-page
$(E_{*,*}^\infty)^r$ associated to any value $r \in \bR$.  In particular, when we take 
 $r=0.5$, then the combination of $U$ and $V$ detects a $1$-cycle. On the other hand, 
 when $r=0.6$ this cycle splits into two smaller cycles which are detected by 
$U$ and $V$ individually. Notice that if we want to come up with a persistent Mayer-Vietoris
method then we need to be able to track this behaviour. That is, we need to know how cycles develop
as $r$ increases. In particular, the barcode
$\GInt{0.5}{1}$ from $\PH_1(X)$ will be broken down into some smaller barcodes, see 
diagram~\ref{fig:barcodes-extension-circle}. These will be  $E^\infty_{1,0}\cong \GInt{0.5}{0.6}$
and also $E_{0,1}^\infty \cong \GInt{0.6}{1.0} \oplus \GInt{0.6}{1.0}$. 
The way we will solve this problem is by using the barcode basis machinery developed in 
Section~\ref{sec:persistence-modules}.  

\begin{figure}
    \begin{center}
	\begin{tikzpicture}
	    \begin{scope}
	        \fill[color=blue!30] (0.3,0.5) ellipse (1 and 1.2);
	        \fill[color=yellow, semitransparent] (1.7,0.5) ellipse (1 and 1.2);
	        \draw (0.3,1.3) node{$U$};
	        \draw (1.7,1.3) node{$V$};
	        % points
	    	\fill (0,0) circle (1pt);
	    	\fill (0,0.5) circle (1pt);
	    	\fill (0,1) circle (1pt);
	    	\fill (0.5,0) circle (1pt);
	    	\fill (0.5,1) circle (1pt);
	    	\fill (1,0) circle (1pt);
	    	\fill (1,1) circle (1pt);
	    	\fill (1,0.2) circle (1pt);
	    	\fill (1,0.8) circle (1pt);
	    	\fill (1.5,0) circle (1pt);
	    	\fill (1.5,1) circle (1pt);
	    	\fill (2,0) circle (1pt);
	    	\fill (2,0.5) circle (1pt);
	    	\fill (2,1) circle (1pt);
	    	\draw (1,-0.9) node{$r=0$};
	    \end{scope}
	    \begin{scope}[xshift=4cm]
	        \fill[color=blue!30] (0.3,0.5) ellipse (1 and 1.2);
	        \fill[color=yellow, semitransparent] (1.7,0.5) ellipse (1 and 1.2);
	        \draw (0.3,1.3) node{$U$};
	        \draw (1.7,1.3) node{$V$};
	        %edges at rad = 0.5
	    	\draw[color=red] (0,0) -- (2,0) -- (2,1) -- (0,1) -- cycle;
	    	\draw[color=red] (1,0) -- (1,0.2);
	    	\draw[color=red] (1,0.8) -- (1,1);
	        % points
	    	\fill (0,0) circle (1pt);
	    	\fill (0,0.5) circle (1pt);
	    	\fill (0,1) circle (1pt);
	    	\fill (0.5,0) circle (1pt);
	    	\fill (0.5,1) circle (1pt);
	    	\fill (1,0) circle (1pt);
	    	\fill (1,1) circle (1pt);
	    	\fill (1,0.2) circle (1pt);
	    	\fill (1,0.8) circle (1pt);
	    	\fill (1.5,0) circle (1pt);
	    	\fill (1.5,1) circle (1pt);
	    	\fill (2,0) circle (1pt);
	    	\fill (2,0.5) circle (1pt);
	    	\fill (2,1) circle (1pt);
	    	\draw (1,-0.9) node{$r=0.5$};
	    \end{scope}
	    \begin{scope}[xshift=8cm]
	        \fill[color=blue!30] (0.3,0.5) ellipse (1 and 1.2);
	        \fill[color=yellow, semitransparent] (1.7,0.5) ellipse (1 and 1.2);
	        \draw (0.3,1.3) node{$U$};
	        \draw (1.7,1.3) node{$V$};
	        %edges at rad = 0.5
	    	\draw[color=red] (0,0) -- (2,0) -- (2,1) -- (0,1) -- cycle;
	    	\draw[color=red] (1,0) -- (1,1);
	    	\filldraw[draw=red, fill = red!20] (0.5,0) -- (1,0) -- (1,0.2) -- cycle;
	    	\filldraw[draw=red, fill = red!20] (0.5,1) -- (1,1) -- (1,0.8) -- cycle;
	    	\filldraw[draw=red, fill = red!20] (1.5,0) -- (1,0) -- (1,0.2) -- cycle;
	    	\filldraw[draw=red, fill = red!20] (1.5,1) -- (1,1) -- (1,0.8) -- cycle;
	        % points
	    	\fill (0,0) circle (1pt);
	    	\fill (0,0.5) circle (1pt);
	    	\fill (0,1) circle (1pt);
	    	\fill (0.5,0) circle (1pt);
	    	\fill (0.5,1) circle (1pt);
	    	\fill (1,0) circle (1pt);
	    	\fill (1,1) circle (1pt);
	    	\fill (1,0.2) circle (1pt);
	    	\fill (1,0.8) circle (1pt);
	    	\fill (1.5,0) circle (1pt);
	    	\fill (1.5,1) circle (1pt);
	    	\fill (2,0) circle (1pt);
	    	\fill (2,0.5) circle (1pt);
	    	\fill (2,1) circle (1pt);
	    	\draw (1,-0.9) node{$r=0.6$};
	    \end{scope}
	\end{tikzpicture}
        \caption{As the radius increases, more edges are added. At radius
        $r=0.5$ a circle will be across the two covers $U$ and $V$. Later on,
        at radius $r=0.6$ this circle will be split into two. }
        \label{fig:extensionCircleSplit2}
    \end{center}
\end{figure}
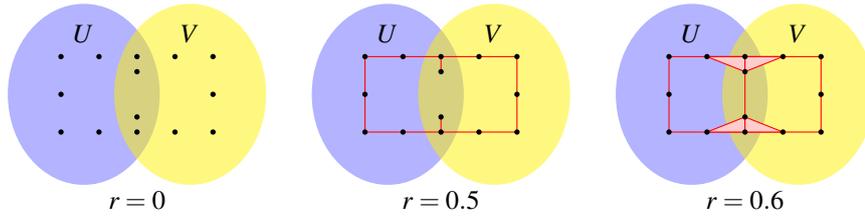

\begin{figure}
    \begin{center}
		\begin{tikzpicture}
        		\draw[red,ultra thick] (1, 2.2) -- (2, 2.2);
        		\draw[blue, ultra thick] (2, 1.8) -- (6, 1.8);
        		\draw[blue, ultra thick] (2, 1.6) -- (6, 1.6);
        		\draw[ultra thick] (1, 0.6) -- (6, 0.6);
        		\draw[ultra thick] (2, 0.4) -- (6, 0.4);
        		\draw[->] (-1,0) -- (7,0);
        		\draw[->] (0,-1) -- (0,3);
        		\draw[semitransparent] (1,-0.1)--(1,2.5);
        		\draw[semitransparent] (2,-0.1)--(2,2.5);
        		\draw[semitransparent] (6,-0.1)--(6,2.5);
        		\draw node at (1,-0.5) {$0.5$};
        		\draw node at (2,-0.5) {$0.6$};
        		\draw node at (6,-0.5) {$1.0$};
        		\draw node at (7,-0.5) {$r$};
        		\draw[red] node at (-0.5,2.2) {$E^\infty_{1,0}$};
        		\draw[blue] node at (-0.5,1.6) {$E^\infty_{0,1}$};
        		\draw node at (-0.7,0.4) {$\PH_1(K)$};
		\end{tikzpicture}
	\end{center}
	\caption{Barcode on associated module.}
       \label{fig:barcodes-extension-circle}
\end{figure}
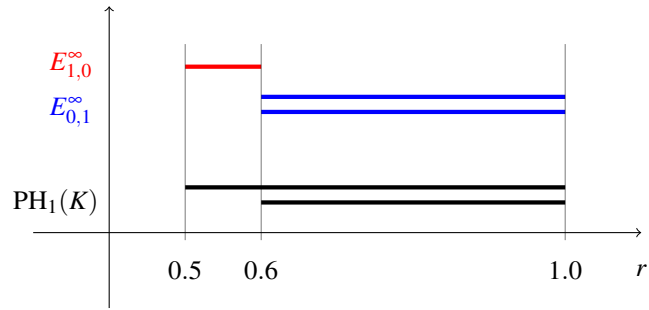
\end{example}

\subsection{The Extension Problem:}
\label{sub:extension-problem}

Recall the definition of the total complex, vertical filtrations and associated modules from 
section~\ref{sec:MayerVietorisSS}.
Through this section we study the extension problem, that is, we will  recover $\Ho_n(\cS_*^\Tot)$ from
the associated modules $G_V^p \left( \Ho_{n}\left( \cS_*^\Tot\right) \right)$. Also, we will assume that the spectral sequence collapses after a finite number of pages.
Consider the persistence module
$$
\bV = \bV(n) \coloneqq \Ho_n\left(\cS_*^\Tot\right),
$$
together with the corresponding filtration
\begin{equation}
\label{seq:split-filtered-persistence}
0 = F^{-1}_V \V\subset F^0_V \V \subset \cdots \subset  F^n_V \V= \V.
\end{equation}
We define the associated modules as the quotients $\bG^k = F^k \V / F^{k-1} \V$ for all $0 \leq k \leq n$.
This gives rise to short exact sequences,
\begin{equation}
\label{seq:short-exact-seq-filtration}
\xymatrix{
	0 \ar[r] &
	F^{k-1} \V \ar[r]^\iota &
	F^k  \V \ar[r]^{p^k} &
	\bG^k  \ar[r] & 0,
}
\end{equation}
for all $0 \leq k \leq n$. Adding up all associated modules we obtain a persistence module
$\bG \coloneqq \bigoplus_{i = 0}^n \bG^i$
 with an additional filtration given by $F^k \bG = \bigoplus_{i = 0}^k \bG^i$ for all $0 \leq k \leq n$.
Since $\bG^k \cong E^\infty_{k,n-k}$ for all $0 \leq k \leq n$, a spectral sequence algorithm will lead 
to a barcode basis for $\bG$. The extension problem consists in computing
 a basis $\B$ for $\bV$ from a basis $\cG$ of $\bG$. 

To start, notice that for each $r \in \bR$ the sequence~(\ref{seq:short-exact-seq-filtration}) 
splits, leading to morphisms
\begin{equation}
\label{eq:point-isomorphism}
	\cF^k(r) : \bG^k(r) \rightarrow F^k\V(r),
\end{equation}
such that $p^k(r) \circ \cF^k(r) = \Id_{\bG^k(r)}$ for all $0\leq k \leq n$. 
In particular, $\cF^k(r)$ is injective for all $0 \leq k \leq n$. 
On the other hand, for any class $[\beta_k]^\infty_{k,n-k}$ of $E^\infty_{k,n-k}$ with representative $\beta_k \in E^0_{k,n-k}$,
since $\beta_k \in \GK_{k,n-k}r$, we have that $d(\beta_k) = 0$ and there exists a sequence of $\beta_i \in \cS_{i,n-i}r$ such that
$d(\beta_{i}) = -\bar{\delta}(\beta_{i+1})$ for
all $0 \leq i < k$. The choice of this sequence determines $\cF^k(r)$,  so that
$$
\cF^k(r)([\beta_k(r)]^\infty_{k,n-k}) = [(\beta_0(r), \beta_1(r), \ldots, \beta_k(r),0, \ldots, 0)]^\Tot_n.
$$
Notice that if we already computed $\cG$ from the Mayer-Vietoris spectral sequence, then there is no need to
do any extra computations to obtain these morphisms $\cF^k(r)$. All we need to do is to store our previous results.
Adding over all $0 \leq k \leq n$ we obtain the isomorphism 
$\cF(r) = \bigoplus_{k=0}^n \cF^k(r) : \bigoplus_{k=0}^n \bG^k(r) \rightarrow \bV(r)$. 
This last morphism is an isomorphism since all its summands are injective, their images have mutual trivial intersection, and the dimensions of the domain and codomain coincide. 

Recall that $\bG$ has induced morphisms $\bG(r \leq s)$ from $\V(r \leq s)$ for all values
$r \leq s$ in $\bR$. Given a basis $\cG$ for $\bG$, we would like to compute a basis $\B$ for $\V$ from this
information. Notice that this is not a straightforward problem since (\ref{eq:point-isomorphism}) does not
imply that one has an isomorphism  $\cF : \bG \rightarrow \V$. A point to start is to define the image along
each generator in $\cG$.  That is, for each barcode generator $\RORel{g_i}{a_i}{b_i}$
in $\cG$, we choose an image at the start $\cF(a_i)(g_i(a_i))$.
After, we set
$\cF(r)(g_i(r)) \coloneqq  \V(a_i < r) \circ \cF(a_i)(g_i(a_i))$ for all $a_i < r < b_i$.
This leads to commutativity of $\cF$ along each generator $g_i$. Nevertheless this is still far from even defining
a morphism $\cF : \bG \rightarrow \V$.

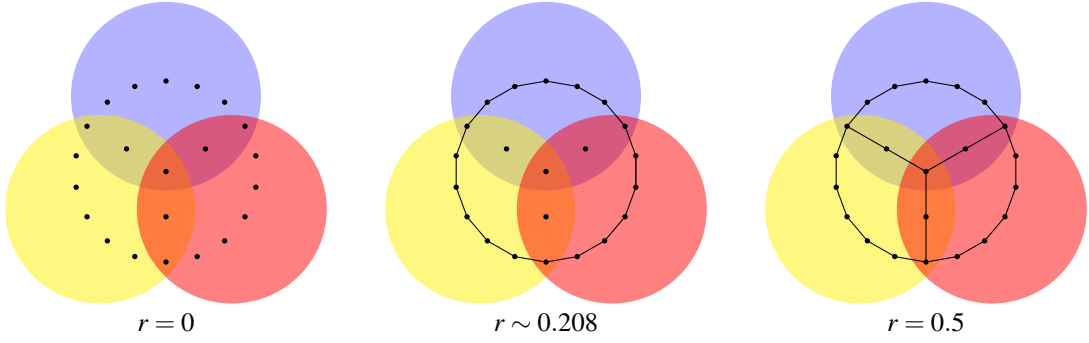
\begin{figure}
    \begin{center}
	\begin{tikzpicture}
	    \begin{scope}
	        \fill[color=blue!30] (90:1cm) circle (1.25);
	        \fill[color=yellow, semitransparent] (210:1cm) circle (1.25);
	        \fill[color=red, semitransparent] (330:1cm) circle (1.25);
	        % now we include the points
	        \fill (30:0.6cm) circle (1pt);
	        \fill (150:0.6cm) circle (1pt);
	        \fill (270:0.6cm) circle (1pt);
	        %points on circle
	        \fill (0,0) circle (1pt);
	        \fill (30:1.2cm) circle (1pt);
	        \fill (50:1.2cm) circle (1pt);
	        \fill (70:1.2cm) circle (1pt);
	        \fill (90:1.2cm) circle (1pt);
	        \fill (110:1.2cm) circle (1pt);
	        \fill (130:1.2cm) circle (1pt);
	        \fill (150:1.2cm) circle (1pt);
	        \fill (170:1.2cm) circle (1pt);
	        \fill (190:1.2cm) circle (1pt);
	        \fill (210:1.2cm) circle (1pt);
	        \fill (230:1.2cm) circle (1pt);
	        \fill (250:1.2cm) circle (1pt);
	        \fill (270:1.2cm) circle (1pt);
	        \fill (290:1.2cm) circle (1pt);
	        \fill (310:1.2cm) circle (1pt);
	        \fill (330:1.2cm) circle (1pt);
	        \fill (350:1.2cm) circle (1pt);
	        \fill (10:1.2cm) circle (1pt);
	        %label
	        \node at (0,-2) {$r = 0$};
	    \end{scope}
	    \begin{scope}[xshift=5cm]
	        \fill[color=blue!30] (90:1cm) circle (1.25);
	        \fill[color=yellow, semitransparent] (210:1cm) circle (1.25);
	        \fill[color=red, semitransparent] (330:1cm) circle (1.25);
	        % now we include the points
	        \fill (30:0.6cm) circle (1pt);
	        \fill (150:0.6cm) circle (1pt);
	        \fill (270:0.6cm) circle (1pt);
	        %points on circle
	        \fill (0,0) circle (1pt);
	        \fill (30:1.2cm) circle (1pt);
	        \fill (50:1.2cm) circle (1pt);
	        \fill (70:1.2cm) circle (1pt);
	        \fill (90:1.2cm) circle (1pt);
	        \fill (110:1.2cm) circle (1pt);
	        \fill (130:1.2cm) circle (1pt);
	        \fill (150:1.2cm) circle (1pt);
	        \fill (170:1.2cm) circle (1pt);
	        \fill (190:1.2cm) circle (1pt);
	        \fill (210:1.2cm) circle (1pt);
	        \fill (230:1.2cm) circle (1pt);
	        \fill (250:1.2cm) circle (1pt);
	        \fill (270:1.2cm) circle (1pt);
	        \fill (290:1.2cm) circle (1pt);
	        \fill (310:1.2cm) circle (1pt);
	        \fill (330:1.2cm) circle (1pt);
	        \fill (350:1.2cm) circle (1pt);
	        \fill (10:1.2cm) circle (1pt);
	        %draw lines on circle
	        \draw (350:1.2cm) -- (10:1.2cm);
	        \foreach \x in {10,30,50,...,350}{
	        	\draw (\x:1.2cm) -- (\x+20:1.2cm);
	        }
	        %label
	        \node at (0,-2) {$r \sim 0.208$};
	    \end{scope}
	    \begin{scope}[xshift=10cm]
	        \fill[color=blue!30] (90:1cm) circle (1.25);
	        \fill[color=yellow, semitransparent] (210:1cm) circle (1.25);
	        \fill[color=red, semitransparent] (330:1cm) circle (1.25);
	        % now we include the points
	        \fill (30:0.6cm) circle (1pt);
	        \fill (150:0.6cm) circle (1pt);
	        \fill (270:0.6cm) circle (1pt);
	        %points on circle
	        \fill (0,0) circle (1pt);
	        \fill (30:1.2cm) circle (1pt);
	        \fill (50:1.2cm) circle (1pt);
	        \fill (70:1.2cm) circle (1pt);
	        \fill (90:1.2cm) circle (1pt);
	        \fill (110:1.2cm) circle (1pt);
	        \fill (130:1.2cm) circle (1pt);
	        \fill (150:1.2cm) circle (1pt);
	        \fill (170:1.2cm) circle (1pt);
	        \fill (190:1.2cm) circle (1pt);
	        \fill (210:1.2cm) circle (1pt);
	        \fill (230:1.2cm) circle (1pt);
	        \fill (250:1.2cm) circle (1pt);
	        \fill (270:1.2cm) circle (1pt);
	        \fill (290:1.2cm) circle (1pt);
	        \fill (310:1.2cm) circle (1pt);
	        \fill (330:1.2cm) circle (1pt);
	        \fill (350:1.2cm) circle (1pt);
	        \fill (10:1.2cm) circle (1pt);
	        %draw lines on circle
	        \foreach \x in {10,30,50,...,350}{
	        	\draw (\x:1.2cm) -- (\x+20:1.2cm);
	        }
	        %draw further lines
	        \draw (0,0) -- (30:1.2cm);
	        \draw (0,0) -- (150:1.2cm);
	        \draw (0,0) -- (270:1.2cm);
	        %label
	        \node at (0,-2) {$r = 0.5$};
	    \end{scope}
	\end{tikzpicture}
        \caption{A one loop is detected at value $r\sim 0.208$ which goes through three covers. 
        Later, at radius $r=0.5$, this loop splits into three loops, each included in one of the three covers.}
        \label{fig:extensionCircleSplit3}
    \end{center}
\end{figure}

The solution to the problem above is to define a new persistence module $\widetilde{\bG}$.
We define $\widetilde{\bG}(s) \coloneqq \bG(s)$ for all $s \in \bR$. Then, if $\cG = \finset{g}{i}{1}{G}$ is a barcode basis for $\bG$,
we will have that $\cG(s)$ will be a basis of $\widetilde{\bG}(s)$ for all $s \in \bR$. Now, given $\RORel{g_i}{a_i}{b_i}$ a
generator in $\cG$, we define the morphism $\widetilde{\bG}(r \leq s)$ by the recursive formula
$$
\widetilde{\bG}(r \leq s)(g_i(r)) \coloneqq
\begin{cases}
	\sum \limits_{j = 1}^{G} c_{i,j}\widetilde{\bG}(b_i \leq s) (g_j(b_i)) & {\rm \ if \ } r \in [a_i,b_i) {\rm \ and \ } b_i \leq s, \\
	g_i(s) & {\rm \ if \ } r, s \in [a_i,b_i), \\
	0 & {\rm \ otherwise,}
\end{cases}
$$
where $c_{i,j} \in \F$ for all $1 \leq i,j \leq G$. We want to define $c_{i,j}$ in
such a way that $\widetilde{\bG}$ is isomorphic to $\V$.  For this we
impose the commutativity condition
$$
\widetilde{\bG}(a_i \leq b_i)(g_i(a_i)) = \cF(b_i)^{-1} \circ \V(a_i \leq b_i) \circ \cF(a_i) (g_i(a_i)),
$$
which leads to the equation
\begin{equation}
\label{eq:extension-morphisms}
\sum \limits_{j = 1}^{G} c_{i,j} g_j(b_i) = \cF(b_i)^{-1} \circ \V(a_i \leq b_i) \circ \cF(a_i) (g_i(a_i)).
\end{equation}
This determines uniquely the coefficients $c_{i,j}$ for all $1 \leq i,j \leq G$. Notice that $\widetilde{\bG}$ respects the filtration on $\V$, since the right hand side
in~(\ref{eq:extension-morphisms}) is a composition of filtration preserving morphisms.
In particular, if $g_i \in F^k \widetilde{\bG}$, then $c_{i,j} = 0$ for all  $1 \leq j \leq G$ such that $g_j \notin F^k \bG$.

Fix a generator $g_i \in \bG^k$ with associated interval $[a_i, b_i)$. 
Let us calculate the coefficients $c_{i,j}$. Suppose that we have a representative $\widetilde{g}_j = (\beta_0^j, \beta_1^j, \ldots, \beta_k^j, 0, \ldots, 0) \in \cS^\Tot_n$ for each generator $g_j \in \cG$, with $g_j = [ \beta^j_k ]^\infty_{k, n-k}$.
Also, for all $0 \leq q \leq n$ we define the subset $I^q \subseteq \{1, \ldots, G\}$ of indices $1 \leq j \leq G$ such that
$g_j \in \bG^q$.
 Then the coefficients $c_{i,j}$ for $j \in I^k \setminus \{i\}$ are determined by the equality in $\bG^k(b_i)$
 $$
 p^k(b_i) \left( \left[ \widetilde{g}_i(b_i) \right]^\Tot_n\right) = \sum_{j \in I^k \setminus \{ i\}} c_{i,j} g_j(b_i).
 $$
Thus, we have
 $$
 p^k(b_i) \left( \left[ \widetilde{g}_i(b_i) - \sum_{j \in I^k \setminus \{ i\}} c_{i,j} \widetilde{g}_j(b_i) \right]^\Tot_n \right) = 0
 $$
where $[\cdot]^\Tot_{n}$ denotes the $n$-homology class of the total complex.
Hence, by~(\ref{seq:short-exact-seq-filtration}) there must exist some $\gamma \in \cS_{n+1}^\Tot(b_i)$ such that 
   \begin{equation}
  \label{eq:subs-previous-filtration-boundary}
	   \widetilde{g}_i(b_i) - \sum_{j \in I^k \setminus \{ i\}} c_{i,j} \widetilde{g}_j(b_i)  - d^\Tot \gamma 
  \end{equation}
  is contained in $F_{k-1} \cS_n^\Tot(b_i)$.
  How do we compute $\gamma$? We start by searching for the first
  page $r \geq 2$ such that
  \begin{equation}
  \label{eq:subs-previous-filtration}
	  \left[ \beta_k^i(b_i) - \sum_{j \in I^k \setminus \{ i\}} c_{i,j} \beta_k^j(b_i)\right]^{r}_{k, n-k} = 0
  \end{equation}
  where $[\cdot]^r_{k,n-k}$ denotes the class in the $r$-page in position $(k, n-k)$. 
  Notice that this $r$ must exist since we assumed that~(\ref{eq:subs-previous-filtration}) vanishes
  on the $\infty$-page. In fact, there exists $\gamma_{k + r-1} \in E_{k+r-1,n-k-r+2}^{r-1}(b_i)$
  such that 
  $$
  \left[ \beta_k^i(b_i) - \sum_{j \in I^k \setminus \{ i\}} c_{i,j} \beta_k^j(b_i) \right]^{r-1}_{k, n-k}  - d^{r-1}(\gamma_{k + r-1}) = 0
  $$
  on $E^{r-1}_{k,n-k}(b_i)$. 
  Repeating for all pages leads to $\gamma_{k + t} \in E_{k+t, n - k - t + 1}^t(b_i)$
  for all $ 0 \leq t \leq r-1$, such that 
  \begin{equation}
  \label{eq:subs-represent-bound}
   \beta_k^i(b_i) - \sum_{j \in I^k \setminus \{ i\}} c_{i,j} \beta_k^j(b_i)  - \sum_{t=0}^{r-1} \widetilde{d^{t}(\gamma_{k + t})}=0,
  \end{equation}
  where $\widetilde{d^{t}(\gamma_{k + t})} \in \cS_{k, n-k}(b_i)$ is a representative for the class $d^{t}(\gamma_{k + t}) \in E^t_{k,n-k}(b_i)$. 
  Notice that equation~(\ref{eq:subs-represent-bound}) holds independently of the representatives, since if we changed some
  term, then the other representatives would adjust to the change. 
  In particular, we have that the $k$ component of~(\ref{eq:subs-previous-filtration-boundary}) vanishes, whereas 
  the $k-1$ component will be equal to
  $$
  \beta_{k-1}^i(b_i) - \sum_{j \in I^k \setminus \{ i\}} c_{i,j} \beta_{k-1}^j(b_i)  - \bar{\delta}(\gamma_k).
  $$
  Next we proceed to find coefficients $c_{i,j} \in \F$ so that 
  in $\bG^{k-1}(b_i)$ we get the equality
  $$
  \left[ \beta_{k-1}^i(b_i) - \sum_{j \in I^k \setminus \{ i\}} c_{i,j} \beta_{k-1}^j(b_i)  - \bar{\delta}(\gamma_k) \right]^\infty_{k-1, n-k+1} = \sum_{j \in I^{k-1}} c_{i,j} g_j(b_i).
  $$
  Then we proceed as we did on $\bG^k$. 
  Doing this for all parameters $0 \leq r \leq k$, there are coefficients $c_{i,j} \in \F$, and an 
  element $\bar{\gamma} \in \cS_n^\Tot(b_i)$ so that 
   $$
  \widetilde{g}_i(b_i)  = \sum \limits_{0 \leq r \leq k} \left( \sum \limits_{j \in I^r} c_{i,j} \widetilde{g}_j(b_i) \right) + d^\Tot \bar{\gamma}.
  $$
  Thus, 
  $$
  \widetilde{\bG}(a_i \leq b_i) (g_i(a_i)) = \sum \limits_{0 \leq j \leq G} c_{i,j} g_j(b_i).
  $$

\begin{proposition}
\label{prop:filtered-persistence-modules}
$\widetilde{\bG} \cong \V$.
\end{proposition}
\begin{proof}
	Since each $\cF(s)$ is an isomorphism, and also we have commutative squares:
$$
\xymatrix@C=1.5cm{
	\widetilde{\bG}(r) \ar[r]^{\widetilde{\bG}(r\leq s)} \ar[d]_{\cF(r)}	&
	\widetilde{\bG}(s) \ar[d]^{\cF(s)}	\\
	\V(r) \ar[r]_{\V(r \leq s)}	& \V(s)
}
$$
for all $r \leq s$, then $\cF$ must be an isomorphism of persistence modules.
\end{proof}

This gives $\widetilde{\bG} \cong \V$, but we still need to compute a barcode basis. In fact, this can be done by applying
the algorithm \texttt{image\_kernel}, but with barcode updates given by the morphisms of $\widetilde{\bG}$. The set
$\I$ which results from this procedure will be a barcode basis for $\widetilde{\bG}$, which by proposition~\ref{prop:filtered-persistence-modules}
leads to a barcode basis for $\V$.

\subsection{\permaviss}
\label{sec:permaviss}

Here we outline a procedure for implementing the persistence Mayer-Vietoris spectral 
sequence. Notice that while using the submodules $\GZ^r_{p,q}$ and $\IB^r_{p,q}$ is
a more intuitive approach from a mathematical perspective, it is more efficient
to work directly with the sets $Z^r_{p,q}$ and $B^r_{p,q}$. By storing representatives
in $Z^r_{p,q}$, we avoid repeating computations on each page and in the extension problem.  
Furthermore, this approach 
allows to easily track the complexity of the algorithm. The current implementation of 
\permaviss~(v.0.0.2) uses the sets $\GZ^r_{p,q}$ and $\IB^r_{p,q}$. However, future
versions will implement the method described here, since it is more efficient and
parallelizable. 

{\bf $0$-Page}. We start by defining the $0$-page as the quotient
$$
E^0_{p,q} = \dfrac{F^p_V \cS^\Tot_{p+q}}{F^{p-1}_V \cS^\Tot_{p+q}} \cong \cS_{p,q} = 
\bigoplus_{\sigma \in N^\U_p} S_q(U_\sigma)
$$
for all pair of integers $p,q \geq 0$. 
The $0$ differential $d^0$, is isomorphic to the standard
chain differential
$$
d^0_{p,q} \cong d_q : \cS_{p,q} \rightarrow \cS_{p,q-1}.
$$
In particular, for each simplex $\sigma \in N^\U_q$, the morphism $d^0_{p,q}$ 
restricts to a local differential 
$$
d^\sigma_q : S_q(U_\sigma) \rightarrow S_{q-1}(U_\sigma).
$$
Thus, we can compute persistent homology to obtain a local base
for the image $\Img(d^\sigma_{q+1})$ and the homology $\cE^1_{\sigma, q}$.
Putting all of these together, we get a basis for $E^1_{p,q}$
as the union $\cE^1_{p,q} = \bigcup_{\sigma \in N_p^\U} \cE^1_{\sigma, q}$. Further, 
for each generator $\alpha \in \cE^1_{p,q}\subseteq E^0_{p,q}$, we store a chain $\alpha_p \in \cS_{p,q}$
so that $\alpha = [(0, \ldots, 0, \alpha_p, 0, \ldots, 0)]^0$. 
Where we denote by $[\cdot]^r$ a class in $E^r_{p,q}$ for all $r \geq 0$. 

{\bf $1$-Page}. 
Recall that the first page elements are given as classes in the quotient
$$
E^1_{p,q} = \dfrac{Z^1_{p,q}}{Z^0_{p-1,q+1} + B^0_{p,q}}.
$$
Therefore, for each generator $\alpha \in \cE^1_{p,q} \subseteq E^0_{p,q}$, 
with $\alpha \sim [a_\alpha, b_\alpha)$, there is 
a chain $\alpha_p \in \cS_{p,q}$, 
so that $\alpha = [(0,\ldots, 0,\alpha_p, 0, \ldots,0)]^0$. Then we compute
the image of $d^1$ on $[\alpha]^1$
$$
d^1 [\alpha]^1 = \Big[\,d^\Tot (0, \ldots, 0, \alpha_p, 0, \ldots, 0)\Big]^1 = 
\Big[\,\big(0, \ldots, 0,\bar{\delta}_p(\alpha_p) , 0, \ldots, 0\big)\,\Big]^1.
$$
Now, for each simplex $\tau \in N^\U_{p-1}$, we have local 
coordinates $\big(\bar{\delta}_p(\alpha_p)\big)_\tau \in S_q(U_\tau)$. We
proceed to solve the linear equation at $a_\alpha \in \bR$
$$
\left(\, \Img\big(\,(d_{q+1}\,)_\tau\,\big)  
\,\middle\vert\,
\cE^1_{\tau,q}
\,\right)^{a_\alpha} X = \big(\bar{\delta}_p(\alpha_p)\big)_\tau,
$$
where the vector $X$ has as many entries as needed for the equation to make sense. 
Also, we have used 
$$
\left(\, \Img\big(\,(d_{q+1}\,)_\tau\,\big)  
\,\middle\vert\,
\cE^1_{\tau,q}
\,\right)^{a_\alpha} = 
\left(\, \Img\big(\,(d_{q+1}\,)_\tau\,\big)(a_\alpha)  
\,\middle\vert\,
\cE^1_{\tau,q}(a_\alpha)
\,\right)
$$
for denoting the matrix on value $a_\alpha$, and whose rows correspond to
a basis of $S_q(U_\tau)$. 
The solution $X$ leads to coefficients $c^1_{\beta} \in \F$ for all $\beta \in \cE^1_{\tau,q}$
and an element $a_{\tau} \in S_{q+1}(U_\tau)$ so that
$$
(\bar{\delta}_p(\alpha_p))_\tau + d_{ q+1}^\tau (a_{\tau}) = 
\sum \limits_{\beta \in \cE^1_{\tau,q}} c^1_\beta \beta_{p-1}.
$$
Repeating this for all $\tau \in N^\U_{p-1}$, we get 
coefficients $c^1_{\beta} \in \F$ for all $\beta \in \cE^1_{p-1,q}$
as well as  a chain $a_{p-1} = \big( a_{\tau} \big)_{\tau \in N^\U_{p-1}} \in \cS_{p-1,q}$
so that
$$
\bar{\delta}_p(\alpha_p) + d_{q+1}(a_{p-1}) = 
\sum \limits_{\beta \in \cE^1_{p-1,q}} c^1_\beta \beta_{p-1}.
$$
Here we define the representative 
$\widetilde{\alpha} = (0, \ldots, 0, a_{p-1}, \alpha_p, 0, \ldots, 0) \in \cS^\Tot_{p+q}$, and 
repeating this for all generators in $\cE^1_{p,q}$, we get a set of
corresponding representatives $\widetilde{\cE^1_{p,q}}$. 
On the other hand, the computed coefficients $c^1_\beta$ mean that
$d^1_{p,q}$ performs the assignment
$$
(1_\F)_\alpha \mapsto (c_\beta)_{\beta \in \cE^1_{p-1,q}}.
$$
Thus, we obtain an associated matrix $D^1_{p,q}$ for $d^1_{p,q}$. 
Using \imagekernel, we compute bases
for the kernel and image.  
Additionally, for each generator $j \in \Img(d^1_{p,q})$, we store 
a \emph{preimage} $p_j \in E^1_{p+1,q}$ such that $d^1(p_j) = j$. 
This can be done by storing coefficients $c_\gamma^1$ for all 
$\gamma \in \cE^1_{p+1,q}$ so that $p_j = \sum_{\gamma \in \cE^1_{p+1,q}} c_\gamma^1 \gamma$. 
Notice that these coefficients are given by \imagekernel~by asking
to return the matrix $T$. This leads to the second page by 
applying \imagekernel~to compute the quotient $\Ker(d^1)/\Img(d^1)$, 
obtaining bases $\cE^2_{p,q}$.

{\bf $2$-Page}. 
Now, we proceed to compute the third page. We start from $\alpha \in \cE^2_{p,q} \subseteq E^1_{p,q}$, 
with $\alpha \sim [a_\alpha, b_\alpha)$
and coordinates $\alpha = (b_\beta)_{\beta \in \cE^1_{p,q}}$.
Then, this leads to a total complex representative 
$$
\widetilde{\alpha} = (0, \ldots, 0, \alpha_{p-1}, \alpha_p, 0, \ldots, 0) = 
\sum \limits_{\beta \in \cE^1_{p,q}} b_\beta \widetilde{\beta}
$$
Since $\alpha \in \Ker(d^1)$, we have that
$$
d^2 [\alpha]^2 = [d^\Tot \widetilde{\alpha}]^2 = 
[(0, \ldots, 0, \delta_{p-1}(\alpha_{p-1}), 0, \ldots, 0)]^2.
$$
As before, by solving local linear equations, we can compute coefficients 
$c^1_\beta \in \F$ for all $\beta \in \cE^1_{p-2,q+1} \subseteq \cS^\Tot_{p+q-1}$
and an element $a \in \cS_{p-2,q+2}$ such that
$$
\bar{\delta}_{p-1}(\alpha_{p-1}) + d_{q+2}(a) = 
\sum \limits_{\beta \in \cE^1_{p-2,q+1}} c^1_\beta \beta.
$$
Now, we solve the linear equation on $X$ and value $a_\alpha \in \bR$
$$
\left(\, \Img(d^1) 
\,\middle\vert\,
\cE^2_{p-2,q+1}
\,\right)^{a_\alpha}X = (c_\beta^1)_{\beta \in \cE^1_{p-2,q+1}}.
$$
The solution $X$ leads to coefficients $c_\beta^2 \in \F$ for all 
$\beta \in \cE^2_{p-2,q+1}$ and $c_\gamma^1 \in \F$
for all $\gamma \in \cE^1_{p-1,q+1}$, so that 
$$
\left[\sum_{\beta \in\cE^1_{p-2,q+1}} c_\beta^1 \beta \right]^1
+ d^1\left( \left[\sum_{\gamma \in\cE^1_{p-1,q+1}} c_\gamma^1 \gamma \right]^1 \right)
= \sum_{\beta \in \cE^2_{p-2,q+1}} c_\beta^2 \beta
$$
Then, we change the total complex representative $\widetilde{\alpha}$ to be
$$
(0, \ldots, 0, a, \alpha_{p-1}, \alpha_p, 0, \ldots, 0) + 
\sum_{\gamma \in \cE^1_{p-1,q+1}} c_\gamma^1 
\, d^\Tot(0, \ldots, 0, \gamma_{p-2}, \gamma_{p-1}, 0, \ldots, 0)
$$
On the other hand, we have that $d^2_{p,q}$ performs the assignment
$$
(1_\F)_\alpha \mapsto (c^2_\beta)_{\beta \in \cE^2_{p-2,q+1}}.
$$
Repeating this for all $\alpha \in \cE^2_{p,1}$, we obtain a matrix
$D^2_{p,q}$ associated to $d^2_{p,q}$. Then applying \imagekernel~we obtain bases for the 
kernel, images and preimages. Then, applying \imagekernel~one more time we 
obtain generators for the third page $\cE^3_{p,q}$. 

{\bf $k$-Page}. 
Suppose that we have computed generators $\cE^k_{p,q} \subseteq E^{k-1}_{p,q}$, together
with total complex representatives $\widetilde{\cE^{k-1}_{p,q}}$ for $k \geq 3$. 
Let a generator $\alpha \in \cE^k_{p,q}$ with $\alpha \sim [a_\alpha, b_\alpha)$
and coordinates $(b_\beta)_{\beta \in \cE^{k-1}_{p,q}}$.
Then, we define a representative
$$
\widetilde{\alpha} = (0, \ldots, 0, \alpha_{p-k+1}, \ldots, \alpha_p, 0, \ldots, 0) = 
\sum \limits_{\beta \in \cE^{k-1}_{p,q}} b_\beta \widetilde{\beta}
$$ 
so that $\alpha = [\widetilde{\alpha}]^{k-1}$. 
Since $\alpha \in \Ker(d^{k-1})$, we have that $\widetilde{\alpha} \in Z^k_{p,q}$ and
as a consequence
$$
d^k[\alpha]^k = [d^\Tot(\widetilde{\alpha})]^k = 
[(0, \ldots, 0, \bar{\delta}_{p-k+1}(\alpha_{p-k+1}), 0, \ldots, 0)]^k.
$$
In particular, if $p-k+1 \leq 0$, then $d^k[\alpha]_k = 0$. 
On the other hand, for $p-k+1 > 0$, we `lift' $d^\Tot (\widetilde{\alpha})$
to the $k$-page. We start on the $0$-page, where we can repeat the procedure
outlined on the $1$-page, to obtain coefficients
$(c_\beta^1)_{\beta \in \cE^1_{p-k,q+k-1}}$, and an element
$a \in \cS_{p-k,q+k}$ so that
$$
\bar{\delta}_{p-k+1}\big(\,\alpha_{p-k+1}\,\big) + 
d^0 \big( \, a \,\big) = 
\sum_{\beta \in \cE^1_{p-k, q+k-1}} c_\beta^1 \beta_{p-k}.
$$
Next, for each $k \geq r \geq 2$, we solve the linear equation on $X$ and on value $a_\alpha \in \bR$
$$
\left(\,
\Img\big(\,d^{r-1}_{p -k +r - 1, q +k - r + 1}\,\big) 
\,\middle\vert\,
\cE^{r}_{p-k, q+k-1}
\,\right)^{a_\alpha} X = 
(c_\beta^{r-1})_{\beta \in \cE^{r-1}_{p-k,q+k-1}}
$$
which leads to  coefficients $(c_\beta^{r}\in \F)_{\beta \in \cE^{r}_{p-k, q+k-1}}$
and $(c_\gamma^{r-1}\in \F)_{\gamma \in \cE^{r-1}_{p -k + r - 1, q +k - r + 1}}$
such that
$$
\left[\sum_{\beta \in\cE^{r-1}_{p-k,q+k-1}} c_\beta^{r-1} \widetilde{\beta} \right]^{r-1}
+ d^r\left( \left[\sum_{\gamma \in\cE^{r-1}_{p-k+r-1,q+k-r+1}} c_\gamma^{r-1} \, \gamma \right]^{r-1} \right)
= \sum_{\beta \in \cE^r_{p-k,q+k-1}} c_\beta^{r} \beta.
$$
Eventually, we obtain the coefficients $(c_\beta^{k})_{\beta \in \cE^{k}_{p-k, q+k-1}}$. 
This leads to the associated matrices, and then we can compute \imagekernel, etc.
On the other hand, we redefine the representative of $\alpha$ as
$$
\widetilde{\alpha} = (0, \ldots, 0, a, \alpha_{p-k+1}, \ldots, \alpha_p, 0, \ldots, 0) + 
\sum_{r=1}^{k-1} \left( \, \sum_{\gamma \in\cE^r_{p-k+r,q+k-r}} c_\gamma^r \, d^\Tot (\widetilde{\gamma}) \, \right)
$$
This leads to the set of representatives 
$\widetilde{\cE}^k_{p,q} \subseteq Z^k_{p,q}$. 
 
\subsection{Extension Problem} 
After computing all pages  of the spectral sequence, we still have to 
solve the extension problem. It turns out that the procedure is
almost exactly the same as for when computing a page on the spectral sequence. 
We start from a basis $\cE^\infty_{p,q}$, with
total complex representatives $\widetilde{\cE}^\infty_{p,q}$.  
Since we assume that the spectral sequence is bounded, it 
collapses at an $L>0$ page. 
Then, for each generator $\alpha \in \cE^L_{p,q}$, with $\alpha \sim [a_\alpha, b_\alpha)$, 
we have a corresponding representative 
$$
\widetilde{\alpha} = (\alpha_0, \ldots, \alpha_p, 0, \ldots, 0) \in \cS^\Tot_{p+q}
$$
in $\widetilde{\cE}^L_{p,q}$. 
The main procedure consists in lifting $\alpha_p$ to the $L$-page.
We do this by means of local linear equations as done on the $1$-page. However, 
this time, instead of using the value $a_\alpha$ we use $b_\alpha$. 
This leads to $a \in \cS_{p,q+1}$ and coefficients $(c^1_\beta)_{\beta \in \cE^1_{p,q}}$
so that 
$$
\alpha_p + d_{q+1}(a) = 
\sum \limits_{\beta \in \cE^1_{p,q}} c^1_\beta \beta_{p}.
$$
The same happens for all the pages $1 \leq r \leq L$, where all the
linear equations are using the value $b_\alpha$. 
This leads to coefficients $(c_\gamma^r)_{\gamma \in \cE^r_{p+r,q-r+1}}$ for 
all $1 \leq r \leq L-1$, and also $(c_\beta^L)_{\beta \in \cE^L_{p,q}}$.  
Then, we define 
$$
\widetilde{\alpha}^{p-1} = 
\widetilde{\alpha} + d^\Tot(0, \ldots, 0, a, 0, \ldots, 0) + 
\sum_{r = 1}^{L-1}\left( \, \sum_{\gamma \in \cE^r_{p+r, q-r+1}} c^r_\gamma d^\Tot(\widetilde{\gamma})\,\right)
- \sum_{\beta \in \cE^L_{p,q}} c_\beta^L \widetilde{\beta}.
$$
In particular, notice that $[\widetilde{\alpha}^{p-1}]^L = 0$. In fact, 
for all integers $L-1 \geq r \geq 0$ one has that $[\widetilde{\alpha}^{p-1}]^r=0$, 
since both the adding and substracting terms are a sum of elements
in $\cE^r_{p,q}$ with the same coefficients. 
As a consequence the $p$-component of $\widetilde{\alpha}^{p-1}$ vanishes, so 
$\widetilde{\alpha}^{p-1} \in F^{p-1}\cS^\Tot_{p+q}$.  Then, 
one can repeat this process with $\widetilde{\alpha}^{r}$ for
all $p-1 \geq r \geq 0$. 
This leads to all coefficients $(c^L_\beta)_{\beta \in \cE^L_{p - r, q + r}}$
for all $0 \leq r \leq p$, which solves the extension problem. That is, 
we have an assignment 
$$
(1_\F)_\alpha \mapsto (c^L_\beta)_{\beta \in \cE^L_{p+q}}.
$$
and a matrix associated to the extensions. Then, applying
\imagekernel, we obtain a barcode basis for persistent homology. 
This is more efficient than the solution presented in section~\ref{sub:extension-problem}, 
however, the former is more intuitive. 

\subsection{Complexity Analysis}
\label{sec:complexity_permaviss}
Let $D_s$ be the maximum simplex dimension in $K$, 
and $\dimn(N^\U)$ the dimension of the nerve. 
Let $L$ be the number of pages. 
Denote $N^\U_{\geq 1} = \bigcup_{k\geq 1} N^\U_k$. 
Let 
$$
X = \max \limits_{q \geq 0, \, \sigma \in N^\cU}\left\{\, | S_q(U_\sigma)| \,\right\}
$$
and let 
$$
Y = \max \limits_{q \geq 0, \, \sigma \in N^\cU_{\geq 1}}\left\{\, | S_q(U_\sigma)| \, \right\}.
$$
Notice that $X \geq Y$. 
On the other hand, we define
$$
H = \max \limits_{p, q \geq 0} \left\{\, |E_{p,q}^1| \, \right\}.
$$
Let $n$ be the number of values in $\bR$ where some bar changes in 
the first page generators $\cE^1_{p,q}$. Notice that one
has $n \leq 4H$. Assume $P$ is the number of processors. 

{\bf $0$-page}. When computing the first page, all we need to do is calculate 
persistent homology in parallel. Then, the complexity is
$$
\left\lceil \dfrac{|\U|}{P}\right\rceil\cO(X^3) + 
\left\lceil \dfrac{|N^\U_{\geq 1}|}{P}\right\rceil\cO(Y^3)
$$
This leads to generators for the first page.

{\bf $1$-page}. For the first page, recall that we start from a generator 
$\alpha \in \cE^1_{p,q}$ with $\alpha \sim [a_\alpha, b_\alpha)$ and proceed to solve $|N^\cU_{p-1}|$ linear
equations. Notice that this can be done for all generators from
$\cE^1_{p,q}$ simultaneously. This is because as the value $a_\alpha$ changes, 
only columns are added and removed to the local linear equations, leaving the rows intact. 
On the other hand, we need to execute \imagekernel~on at most
$\dimn(N^\U) D_s$ elements on the first page. Notice that for each of these, we first compute  
a basis for the images and kernels,  and afterwards we perform the quotients. Each of these
takes a complexity of at most $\cO( 4 H^4)$.  
Also, we need to add the complexity of the \cech~differential. An option for computing this, is to compare
simplices in different covers by their vertices; two simplices are the same iff they share the same vertex set. 
This would take less than $\cO(|N^\U|D_sX^2H)$ operations.
Thus the overall complexity becomes
$$
\left\lceil \dfrac{|\U|}{P}\right\rceil \cO(X^2 H) + 
\left\lceil \dfrac{|N^\cU_{\geq 1}|}{P}\right\rceil \cO(Y^2 H) +
\left\lceil \dfrac{\dimn(N^\U) D_s}{P}\right\rceil 
\left(\cO(|N^\U|D_sX^2H) + \cO( 4 H^4)\right)
$$

{\bf $k$-page}. Now, we proceed for the complexity
of the page $k \geq 2$. This is the same as for the 
$1$ page, with the addition of Gaussian eliminations
of higher pages. These take at most $\cO(H^2)$ time for
each generator in $\cE^r_{p,q}$. If we do these
for all generators simultaneously, since
we need to update both rows and columns in a matrix, we might use
\imagekernel~and the
complexity becomes $\cO(nH^3)$. Denoting by  
$L$ the infinity page, we have the new term
$$
\left\lceil \dfrac{\dimn(N^\U) D_s}{P}\right\rceil
\cO(4L H^4)
$$
which added to the complexity of the $1$-page, 
we obtain
\begin{eqnarray*}
\left\lceil \dfrac{|\U|}{P}\right\rceil \cO(X^2 H) + 
\left\lceil \dfrac{|N^\cU_{\geq 1}|}{P}\right\rceil \cO(Y^2 H) +
\left\lceil \dfrac{\dimn(N^\U) D_s}{P}\right\rceil
\left(\cO(|N^\U|D_sX^2H) + \cO(4 H^4) + \cO(4L H^4)\right) \\
= 
\left\lceil \dfrac{\dimn(N^\U) D_s}{P}\right\rceil
\left(\cO(|N^\U|D_sX^2H) + \cO(4L H^4)\right).
\end{eqnarray*}

{\bf Extension problem}. If the spectral sequence collapses at $L>0$, then the
complexity of extending all generators in $\cE^L_{p,q}$ is bounded by that of
computing the $L$ page about $D_s$ times.

{\bf Overall complexity}. Altogether, we have a complexity bounded by that
of computing the first page plus that of computing the $L$ page $L+D_s$ times. Here
the $L$ comes from computing the $L$ page $L$ times and $D_s$ from the extension problem.
Thus, the overal complexity is bounded by 
\begin{eqnarray*}
\left\lceil \dfrac{|\U|}{P}\right\rceil\cO(X^3) + 
\left\lceil \dfrac{|N^\U_{\geq 1}|}{P}\right\rceil\cO(Y^3)  
 + \Big(L + D_s\Big)
\left\lceil \dfrac{\dimn(N^\U) D_s}{P}\right\rceil
\left(\cO(|N^\U|D_sX^2H) + \cO(4L H^4)\right).
\end{eqnarray*}
Notice that in general $D_s$, $L$ and $\dimn(N^\cU)$ are much smaller
than $H$ and $X$. Thus, 
for covers such that $Y \ll X$ and $|N^\U| \ll X$, and assuming we 
have enough processors,
the complexity can be simplified to the two dominating terms
\begin{eqnarray*}
\cO(X^3) +  \cO(H^4).
\end{eqnarray*}
Notice that this last case is satisfied for those covers whose mutual intersections are 
generally smaller than each cover.  
Also, in this case $H$ is approximately of the order of nontrivial barcodes over all 
the input complex.
This shows that \permaviss~isolates simplicial data, 
while only merging homological information. It is worth to notice that in general $H$, 
being the number of nontrivial bars, is much 
smaller than the size of the whole simplicial complex. However, in some cases this might not be true. 
Nevertheless our complexity estimates are very generous, leaving plenty of space for improvement 
on concrete applications. 

\section{Conclusion}

We started by developing linear algebra for persistence modules. In doing so, 
we introduced bases of persistence modules, as well as associated matrices to morphisms. 
Also, we presented Algorithm~\ref{cde:img-kernel}, which computes bases 
for the image and the kernel of a persistence morphism between any pair of 
tame persistence modules. 
Then a generalization of traditional persistent homology was introduced in Subsection~\ref{sub:persistence-module-homology}.
This theory, has helped us to define and understand the Persistent Mayer-Vietoris spectral sequence. 
Furthermore, we have provided specific guidelines for a distributed algorithm,
with a solution to the extension problem presented in Section~\ref{sub:extension-problem}.
The \permaviss~method presented in section~\ref{sec:permaviss} isolates simplicial 
information to local matrices, while merging only homological information between different covers. 
Thus, the complexity of this method is dominated by the size of a local complex plus the 
order of barcodes over all the data. 
A first implementation of these results can be found in~\citep{permaviss}. 
Coding an efficient implementation from the pseudo-code given in this paper, and benchmarking its performance compared to other methods,  will be a matter of future research.
Another interesting direction of research is how to merge this method with existing algorithms, such as those from \citep{Chen2011, Chen2013, DeSilva2011, Milosavljevic2011}. Especially it would be interesting to explore the possible interactions of discrete Morse theory and this approach, see~\citep{CGN2016}.
Additionally, it will be worth exploring, both theoretically and practically, which are the most 
suitable covers for different applications. 
Finally, we would also like to study the additional information given by the covering. 
This will add locality information from persistent homology. In particular, it is worth noticing 
that on experiments the two most expensive pages to compute are the first and second one. This is why we have a strong belief that most of the extra information will be contained in the first two pages. 

\section{Acknowledgements}
I would like to thank my supervisor Dr. Ulrich Pennig who suggested this topic and has been very helpful and supportive in the development of these ideas. Also I would like to thank Dr. Padraig Corcoran and Dr. Thomas E. Woolley with whom I have discussed these ideas on several occasions, and have given me invaluable advice. A special thank you goes to the anonymous reviewer who took the time to read the first version of this work. Finally I would like to express my gratitude to EPSRC for the grant EP/N509449/1 support with project number 1941653, without which I would not have been able to write this work.

\bibliography{library}

\bibliographystyle{plainnat}

\end{document}